\documentclass[a4paper,reqno,11pt]{amsart}
\usepackage{amsmath,amsthm,amssymb,eucal, amscd, mathrsfs,amsfonts}
\usepackage[T1]{fontenc}
\usepackage{fourier}

\usepackage{hyperref}
\usepackage{array}

\usepackage{latexsym}

\usepackage[all]{xy}

\hypersetup{
  colorlinks = true,
  urlcolor = blue,
  linkcolor = blue,
  citecolor = red,
  pdfauthor = {Moutuou,E.},
  pdfkeywords = {Real groupoids, groupoid cohomology},
  pdftitle = {Moutuou, E. - On groupoids with involutions and their cohomology},
  pdfsubject = {groupoid cohomology},
  pdfpagemode = UseNone
}

 \def\commutatif{\ar@{}[rd]|{\circlearrowleft}}

\newcommand{\eq}[1][r]
   {\ar@<-3pt>@{-}[#1]
    \ar@<-1pt>@{}[#1]|<{}="gauche"
    \ar@<+0pt>@{}[#1]|-{}="milieu"
    \ar@<+1pt>@{}[#1]|>{}="droite"
    \ar@/^2pt/@{-}"gauche";"milieu"
    \ar@/_2pt/@{-}"milieu";"droite"}
    
    \def\dar[#1]{\ar@<2pt>[#1]\ar@<-2pt>[#1]}
  \entrymodifiers={!!<0pt,0.7ex>+} 
\newtheorem{thm}{Theorem}[section]
\newtheorem{pro}[thm]{Proposition}
\newtheorem{lem}[thm]{Lemma}
\newtheorem{cor}[thm]{Corollary}

\newtheorem{nota}[thm]{Notations}

\allowdisplaybreaks

\theoremstyle{definition}{\normalfont}
\newtheorem{df}[thm]{Definition}
\newtheorem{dfpro}[thm]{Definition and Proposition}

\newtheorem{rem}[thm]{Remark}
\newtheorem{rmks}[thm]{Remarks}
\newtheorem{ex}[thm]{Example}
\newtheorem{exs}[thm]{Examples}
\allowdisplaybreaks

\newcommand{\cC}{{\mathcal C}}  
  
\newcommand{\cE}{{\mathcal E}}  
\newcommand{\cF}{{\mathcal F}}   
\newcommand{\cG}{{\mathcal G}}
  
\newcommand{\cI}{{\mathcal I}}

\newcommand{\cL}{{\mathcal L}}  
\newcommand{\cM}{{\mathcal M}}

\newcommand{\cP}{{\mathcal P}}

\newcommand{\cS}{{\mathcal S}}
\newcommand{\cT}{{\mathcal T}}
\newcommand{\cU}{{\mathcal U}}  
\newcommand{\cV}{{\mathcal V}}
\newcommand{\cW}{{\mathcal W}}

\newcommand{\CC}{{\mathbb C}}
\newcommand{\EE}{{\mathbb E}}

\newcommand{\KK}{{\mathbb K}}
\newcommand{\NN}{{\mathbb N}}

\newcommand{\PP}{{\mathbb P}}
\newcommand{\QQ}{{\mathbb Q}}
\newcommand{\RR}{{\mathbb R}}
\newcommand{\bbS}{{\mathbb S}}
\newcommand{\uc}{{\mathbb S}^1}

\newcommand{\ZZ}{{\mathbb Z}}

\newcommand{\bfS}{{\mathrm S}}

\newcommand{\id}{{\mathbf 1}}

\newcommand{\wGa}{{\widetilde \Gamma}}

\newcommand{\wRTw}{{\what{\operatorname{\textrm{Tw}R}}}}
\newcommand{\wRExt}{{\what{\operatorname{\textrm{Ext}R}}}}
\newcommand{\wExt}{{\what{\text{\textrm{Ext}}}}}
\newcommand{\Ext}{\text{\textrm{Ext}}}

\newcommand{\Rext}{{\what{\operatorname{\textrm{ext}R}}}}

\newcommand{\RTop}{{\mathfrak{RTop}}}

\newcommand{\PicR}{{\operatorname{\textrm{Pic}R}}}

\newcommand{\Gpdo}{X}
\newcommand{\grpd}{\xymatrix{\cG \dar[r] & X}}
\newcommand{\gamgpd}{\xymatrix{\Ga \dar[r] & Y}}
\newcommand{\rGpd}{\xymatrix{^r\mathcal{G}\dar[r] & {}^rX}}

\newcommand{\Gamo}{Y}

\newcommand{\rGam}{{^r\Gamma}}

\newcommand{\what}{\widehat}

\newcommand{\fr}{{\mathfrak r}}
\newcommand{\fs}{{\mathfrak s}}

\newcommand{\fB}{{\mathfrak B}}
\newcommand{\RG}{{\mathfrak {RG}}}

\newcommand{\frc}{\mathfrak{c}}

\newcommand{\fF}{{\mathfrak F}}
\newcommand{\fG}{{\mathfrak G}}

\newcommand{\To}{\longrightarrow}

\newcommand{\mto}{\longmapsto}

\newcommand{\cstar}{C^{\ast}}
\newcommand{\Ga}{\Gamma}

\newcommand{\del}{\delta}

\newcommand{\ph}{\phi}
\newcommand{\vr}{\varrho}
\newcommand{\ve}{\varepsilon}
\newcommand{\vp}{\varphi}
\newcommand{\g}{\gamma}

\newcommand{\Lam}{\Lambda}

\newcommand{\Hom}{\operatorname{Hom}}
\newcommand{\Isom}{\operatorname{Isom}}

\newcommand{\op}{{\text{\rm\tiny op}}}
\newcommand{\Id}{\text{\normalfont Id}}
\newcommand{\supp}{\operatorname{supp}}

\title{On groupoids with involutions and their cohomology}

\author{El-ka\"ioum M. Moutuou}

\address{LMAM, Universit\'e de Lorraine - Metz, CNRS UMR 7122, Ile du Saulcy,
F-57045 Metz Cedex 1, France
}
\email{moutuou@univ-metz.fr}

\begin{document}

\maketitle

\begin{abstract}
We extend the definitions and main properties of graded extensions to the category of locally compact groupoids endowed with involutions. We introduce Real \v{C}ech cohomology, which is an equivariant-like cohomology theory suitable for the context of groupoids with involutions. The Picard group of such a groupoid is discussed and is given a cohomological picture. Eventually, we generalize Crainic's result, about the differential cohomology of a proper Lie groupoid with coefficients in a given representation, to the topological case.
\end{abstract}


 
\section*{Introduction}

A Real~\footnote{Note the capitalization, used to avoid confusion with a module over $\RR$ or a real manifold.} object in a category $\cC$ is a pair $(A,f)$ consisting of an object $A\in Ob(\cC)$ together with an element $f\in \Isom_\cC(A,A)$, called the \emph{Real structure}, such that $f^2=\id_A$. For instance, an Atiyah Real space $(X,\tau)$~\cite{At1} is nothing but a Real object in the category of locally compact spaces. We are particularly interested in the category $\mathfrak{G}_s$~\cite{TLX} of locally compact Hausdorff groupoids with \emph{strict homomorphisms}~\cite{Moer1,MM} as morphisms; we shall refer to Real objects in $\mathfrak{G}_s$ as \emph{Real groupoids}. For example, let $\mathbb{W}\PP^n_{(a_1,\ldots,a_n)}$ be the weighted projective orbifold~\cite{ALR} associated to the pairwise coprime integers $a_1,\ldots,a_n$; then together with the coordinate-wise complex conjugation, $\mathbb{W}\PP^n_{(a_1,\ldots,a_n)}$ is a Real groupoid.

A morphism of Real groupoids is a morphism in $\mathfrak{G}_s$ intertwining the Real structures. We may also speak of a Real strict homomorphism. Real groupoids form a category $\RG_s$ in which morphisms are Real strict homomorphisms. Moreover, they are the objects of a $2$-category $\RG(2)$ defined as follows. Let $(\cG,\rho),(\Ga,\vr)\in Ob(\RG_s)$. A \emph{generalized homomorphism}~\cite{Hae,HS,MM,TLX} $\Ga \stackrel{Z}{\To} \cG$ is said to be \emph{Real} if $Z$ is given a Real structure $\tau$ such that the moment maps and the groupoid actions respect some coherent compatibility conditions with respect to the Real structures. A morphism of Real generalized homomorphisms $(Z,\tau)\To (Z',\tau')$ is a morphism of generalized homomorphisms $Z\To Z'$ intertwining the Real structures. Henceforth, $1$-morphisms in $\RG(2)$ are Real generalized homomorphisms and $2$-morphisms are morphisms of Real generalized homomorphisms. All functorial properties we deal with in this paper are however discussed in the category $\RG$ defined as $\RG(2)$ "up to $2$-isomorphisms". 

 In~\cite{Tu1}, a \v{C}ech cohomology theory for topological groupoids is defined as the \v{C}ech cohomology of simplicial topological spaces, and it is shown that the well-known isomorphism between $\uc$-central extensions of a discrete groupoid $\cG$ and the second cohomology group~\cite{Ren,Kum88} of $\cG$ with coefficients in the sheaf of germs of $\uc$-valued functions also holds in the general case; \emph{i.e.}, $\Ext(\cG,\uc)\cong \check{H}^2(\cG_\bullet,\uc)$. We define here an analogous theory $\check{H}R^\ast$ that fits well the context of Real groupoids. This theory was motivated by the classification of groupoid $\cstar$-dynamical systems endowed with involutions~\cite{Mou}. These can be thought of as a generalization of continuous-trace $\cstar$-algebras with involutions. Specifically, it is known~\cite{Ros} that given such a $\cstar$-algebra $A$, its spectrum $X$ admits a Real structure $\tau$, and its Dixmier-Douady invariant $\del(A)\in \check{H}^2(X,\uc)$ is such that $\overline{\del(A)}=\tau^\ast\del(A)$, where the "bar" is the complex conjugation in $\uc$. In fact, thinking of $X$ as a Real groupoid,  we will see that all $2$-cocycles satisfying the latter relation are classified by $\check{H}R^2(X,\uc)$, where $\uc$ is endowed with the complex conjugation. $\check{H}R^\ast$ appears then to provide the right cohomological interpretation of $\cstar$-dynamical systems with involutions.

We try, to the extent possible, to make the present paper self-contained. We start by collecting, in Section 1, a number of notions and results about Real groupoids most of which are adapted from many sources in the literature~\cite{Moer1, Ren,TLX}; specifically, we define the group $\wRExt(\cG,\bfS)$ of (equivalence classes of) Real graded $\bfS$-central extensions over a Real groupoid $\cG$, by a given Real abelian group $\bfS$. 
In Section 2, we introduce Real \v{C}ech cohomology, following closely~\cite{Tu1}. While $\check{H}R^\ast$ behaves almost like a $\ZZ_2$-equivariant cohomology theory, we will see that it is actually not. Geometric interpretations of the cohomology groups $\check{H}R^1(\cG_\bullet,\bfS)$ and $\check{H}R^2(\cG_\bullet,\bfS)$, for a Real Abelian group $\bfS$, are given. Finally, we generalize a result by  Crainic~\cite{Crai} (on the differential cohomology groups of a proper Lie groupoid) to topological proper (Real) groupoid.



\section{Real groupoids and Real graded extensions}

Recall~\cite{Ren,MM,TLX} that a \emph{strict homomorphism} between two groupoids $\grpd$ and $\gamgpd$ is a functor $\vp: \Ga \To \cG$ given by a map $\Gamo \To \Gpdo$ on objects and a map $\Ga^{(1)} \To \cG^{(1)}$ on arrows, both denoted again by $\vp$, which preserve the groupoid structure maps, \emph{i.e}. $\vp(s(\g ))=s(\vp(\g)), \ \vp(r(\g))=r(\vp(\g)), \ \vp(\id_y)=\id_{\vp(y)}$ and $\vp(\g_1 \g_2) =\vp(\g_1)\vp(\g_2) $  (hence $\vp(\g^{-1})=\vp(\g)^{-1})$,  for all $(\g_1,\g_2)\in \Ga^{(2)}$ and $y\in \Gamo$. Unless otherwise specified, all our groupoids are topological groupoids which are supposed to be Hausdorff and locally compact. 

\subsection{Real groupoids}

\begin{df}
A \emph{Real} groupoid is a groupoid $\grpd$ together with a strict $2$-periodic homeomorphism $\rho: \cG \To \cG$. The homeomorphism $\rho$ is called a \emph{Real structure on $\cG$}.
Such a groupoid will be denoted by a pair $(\cG,\rho)$.
\end{df}

\begin{ex}
Any topological Real space $(X,\rho)$ in the sense of Atiyah~\cite{At1} can be viwed as a Real groupoid whose the unit space and the space of morphisms are identified with $X$; i.e, the operations in this Real groupoid is defined by $s(x)=r(x)=x, \ x\cdot x=x$, $x^{-1}=x$.
\end{ex}

\begin{ex}
Any group with involution can be viewed as a Real groupoid with unit space identified with the unit element. Such a group will be called Real.	
\end{ex}

\begin{lem}~\label{lem:decom-Real-group}
Let $G$ be an abelian group equipped with an involution $\tau:G\To G$ (\emph{i.e.} a Real structure). Set 
\[ \Re(\tau):=\{g\in G \mid \tau(g)=g\} ={}^\RR G, \quad \Im(\tau):=\{g\in G\mid \tau(g)=-g\}.\]
Then, 
\begin{eqnarray}
		G\otimes \ZZ[{1\over 2}]\cong \left(\Re(\tau)\oplus \Im(\tau)\right)\otimes \ZZ[{1\over 2}]. 
	\end{eqnarray}	
If $\tau$ is understood, we will write ${}^\cI G$ for $\Im(\tau)$. We call $\Re(\tau)$ and $\Im(\tau)$ the Real part and the imaginary part of $G$, respectively. 	
\end{lem}

\begin{proof}
For all $g\in G$, one has $g+\tau(g)\in {}^\RR G$, and $g-\tau(g)\in {}^\cI G$. Therefore, after tensoring $G$ with $\ZZ[1/2]$, every $g\in G$ admits a unique decomposition $$g=\frac{g+\tau(g)}{2}+\frac{g-\tau(g)}{2}\in \ZZ[1/2]\otimes\left({}^\RR G\oplus {}^\cI G\right).$$	
\end{proof}

\begin{ex}
Let $n\in \NN^\ast$. Suppose $\rho$ is a Real structure on the additive group $\RR^n$. Then there exists a unique decomposition $\RR^n=\RR^p\oplus \RR^q$ such that $\rho$ is determined by the formula \[\rho(x,y)=(\id_p\oplus(-\id_q))(x,y):=(x,-y),\]
for all $(x,y)=(x_1,\cdots, x_p,y_1,\cdots,y_q)\in \RR^p\oplus \RR^q$. 

For each pair $(p,q)\in \NN$, we will write $\RR^{p,q}$ for the additive group $\RR^{p+q}$ equipped with the Real structure $(\id_p\oplus(-\id_q))$. 

Define the Real space $\bfS^{p,q}$ as the invariant subset of $\RR^{p,q}$ consisting of elements $u\in \RR^{p+q}$ of norm $1$. For $q=p$, $\bfS^{p,p}$ is clearly identified with the Real space $\bbS^p$ whose Real structure is given by the coordinate-wise complex conjugation. Notice that ${}^r\bfS^{p,q}=\bfS^{p,0}$.
\end{ex}

\begin{ex}
Let $(X,\rho)$ be a topological Real space. Consider the fundamental groupoid $\pi_1(X)$ over $X$ whose arrows from $x\in X$ to $y\in X$ are homotopy classes of paths (relative to end-points) from $x$ to $y$ and the partial multiplication given by the concatenation of paths. The involution $\rho$ induces a Real structure on the groupoid as follows: if $[\g] \in \pi_1(X)$, we set $\rho([\g])$ the homotopy classes of the path $\rho(\g)$ defined by $\rho(\g)(t):=\rho(\g(t))$ for $t\in [0,1]$. 
\end{ex}

Two Real structures $\rho$ and $\rho'$ on $\cG$ are said to be \emph{conjugate} if there exists a strict homeomorphism $\phi: \cG \To \cG$ such that $\rho' = \phi \circ \rho \circ \phi^{-1}$. In this case we say that the Real groupoids $(\cG,\rho)$ and $(\cG,\rho')$ are equivalent.

\begin{df}
We write $\rGpd$ (or ${}^\rho\cG$ when there is a risk of confusion) for the the subgroupoid of $\grpd$ by 
$\rho$.	
\end{df}

\begin{lem}
Let $\cG$ and $\Ga$ be Real groupoids, and let $\ph:\Ga\To \cG$ be a Real groupoid homomorphism, then $\ph(\rGam)$ is a full subgroupoid of $\rGpd$. If in addition $\ph$ is an isomorphism, then $\rGam\cong \rGpd$.

In particular, if $\rho_1$ and $\rho_2$ are two conjugate Real structures on $\cG$, then ${}^{\rho_1}\cG\cong {}^{\rho_2}\cG$. 
\end{lem}

\begin{proof}
This is obvious since $\ph(\bar{\g})=\overline{\ph(\g)}$ for all $\g\in \Ga$.
\end{proof}

\begin{rem}
Note that the converse of the second statement of the above lemma is false in general. For instance, consider the Real group $\uc$ whose Real structure is given by the complex conjugation, and the Real group $\ZZ_2$ (with the trivial Real structure). We have ${}^r \uc=\{\pm1\}\cong \ZZ_2={}^r \ZZ_2$.  		
\end{rem}

The following is an example of groupoids with equivalent Real structures.

\begin{ex}
Recall (~\cite[IV.3]{SH}) that a Riemannian manifold $X$ is called \emph{globally symmetric} if each point $x\in X$ is an isolated fixed point of an involutory isometry $s_x:X\To X$; i.e. $s_x$ is a diffeomorphism verifying  $s_x^2=\Id_X$ and $s_x(x)=x$. Moreover, for every two points $x, y\in X$, $s_x$ and $s_y$ are related through the formula $s_x \circ s_y\circ s_x=s_{s_x(y)}$. Given such a space, each point $x\in X$ defines a Real structure on $X$ which leaves $x$ fixed. However, let $x$ and $y$ be two different points in $X$ and let $z\in X$ be such that $y=s_z(x)$. Then, we get $s_z\circ s_x \circ s_z=s_y$ which means that the diffeomorphism $s_z:X\To X$ implements an equivalence $s_x \sim s_y$. But since $x$ and $y$ are arbitrary, it turns out that all of the Real structures $s_x$ are equivalent. Thus, all of the Real spaces $(X,s_x)$ are equivalent to each others.

Now, recall~\cite[IV. Theorem 3.3]{SH} that if $G$ denotes the identity component of $I(X)$, where the latter is the group of isometries on $X$, then the map $\sigma_{x_0}: g\mto s_{x_0}gs_{x_0}$ is an involutory automorphism in $G$, for any arbitrary $x_0\in X$. It follows that all of the points of $X$ give rise to equivalent Real groups $(G,\sigma_x)$.
\end{ex}

From now on, by a Real structure on a groupoid, we will mean a representative of a conjugation class of Real structures. Moreover, we will sometimes put $\bar{g}:=\rho(g)$, and write $\cG$ instead of $(\cG,\rho)$ when $\rho$ is understood.   	

\begin{df}[Real covers]
Let $(X,\rho)$ be a Real space. We say that an open cover $\cU=\{U_i\}_{i\in I}$ of $X$ is \emph{Real} if $\cU$ is invariant with respect to the Real structure $\rho$; \emph{i.e}. $\rho(U_i)\in \cU, \forall i\in I$.  Alternatively,  $\cU$ is Real if $I$ is equipped with an involution $i\mto \bar{i}$ such that $U_{\bar{i}}=\rho(U_i)$ for all $i\in I$.	
\end{df}

\begin{rem}
Observe that Real open covers always exist for all locally compact Real space $X$. Indeed, let $\cV=\{V_{i'}\}_{i'\in I'}$ be an open cover of the space $X$. Let $I:=I'\times \{\pm1\}$ be endowed with the involution $(i',\pm1)\mto (i',\mp1)$. Next, put $U_{(i',\pm1)}:=\rho^{(\pm1)}(V_{i'})$, where $\rho^{(+1)}(g):=g$, and $\rho^{(-1)}(g):=\rho(g)$ for $g\in \cG$.
\end{rem}

\begin{df}[Real action]
Let $(Z,\tau)$ be a locally compact Hausdorff Real space. A (continuous) right Real action of $(\cG,\rho)$ on $(Z,\tau)$ is given by a continuous open map $\fs: Z\To \Gpdo$ (called the \emph{generalized source map}) and a continuous map $Z \times_{\fs,\Gpdo,r} \cG \To Z$, denoted by $(z,g) \mto zg$, such that 

\begin{itemize}
\item[(a)] $\tau(zg)=\tau(z)\rho(g)$ for all $(z,g)\in Z \times_{\fs,\Gpdo,r} \cG$;
\item[(b)] $\rho(\fs(z))=\fs(\tau(z)) $ for all $z\in Z$;
\item[(c)] $\fs(zg)=s(g)$;
\item[(d)] $z(gh)=(zg)h$ for $(z,g)\in Z \times_{\fs,\Gpdo,r} \cG$ and $(g,h)\in \cG^{(2)}$;
\item[(e)] $z\fs(z)=z$ for any $z\in Z$ where we identify $\fs(z)$ with its image in $\cG$ by the inclusion $\Gpdo \hookrightarrow \cG$.
\end{itemize}
If such a Real action is given, we say that $(Z,\tau)$ is a (right) Real $\cG$-space.
\end{df}

Likewise a (continuous) left Real action of $(\cG,\rho)$ on $(Z,\tau)$ is determined by a continuous Real open surjection $\fr: Z\To \Gpdo$ (the \emph{generalized range map} of the action) and a continuous Real map $\cG \times_{s,\Gpdo,\fr}Z \To Z$ satisfying the appropriate analogues of conditions (a), (b), (c), (d) and (e) above.

Given a right Real action of $(\cG,\rho)$ on $(Z,\tau)$ with respect to $\fs$, let $\Psi: Z \times_{\fs,\Gpdo,r} \cG \To Z\times Z$ be defined by the formula $\Psi(z,g)=(z,zg)$. Then we say that the action is \emph{free} if this map is one-to-one (or in other words if the equation $zg=z$ implies $g=\fs(z)$. The action is called \emph{proper} if $\Psi$ is proper.

\begin{nota}
If we are given such a right (resp. left ) Real action of $(\cG,\rho)$ on $(Z,\tau)$, and if there is no risk of confusion, we will write $Z\ast\cG$ (resp. $\cG\ast Z$) for  $Z\times_{\fs,\Gpdo,r}\cG$ (resp. for $\cG\times_{s,\Gpdo,\fr}Z$).	
\end{nota}


\subsection{Real $\cG$-bundles}

\begin{df}~\label{df:G-bundle-grpd}
Let $(\cG,\rho)$ be a Real groupoid. A Real (right) $\cG$-bundle over a Real space $(Y,\varrho)$ is a Real (right) $\cG$-space $(Z,\tau)$ with respect to a map $\fs: Z\To \Gpdo$, together with a Real map $\pi: Z\To Y$ satisfying the relation $\pi(zg)=\pi(z)$ for any $(z,g)\in Z\times_{\fs,\Gpdo,r}\cG$, and such that for any $y\in Y$, the induced map $$\tau_y: Z_y \To Z_{\varrho(y)}$$ on the fibres is $\cG$-antilinear in the sense that for $(z,g)\in Z_y \times_{\fs, \Gpdo,r} \cG$ we have $$\tau_y(zg)=\tau_y(z)\rho(g)$$ as an element in $Z_{\varrho(y)}$.

Such a bundle $(Z,\tau)$ is said to be \emph{principal} if
\begin{itemize}
\item[(i)] $\pi: Z \To Y$ is \emph{locally split} (means that it is surjective and admits local sections), and
\item[(ii)] the map $Z\times_{\fs,\Gpdo,r}\cG \To Z\times_{Y}Z, \ (z,g)\mto (z,zg)$ is a Real homeomorphism.

\end{itemize}

\end{df}

\begin{rmks}~\label{prop-bdles} 
(1). \textbf{The unit bundle.} Given a Real groupoid $(\cG,\rho)$, its space of arrows $\cG^{(1)}$ is a $\cG$-principal Real bundle over $\Gpdo$. Indeed, the projection is the range map $r: \cG^{(1)}\To \Gpdo$, the generalized source map is given by $s$ and the action is just the partial multiplication on $\cG$. This bundle is denoted by $U(\cG)$ and is called the \emph{unit} bundle of $\cG$ (cf. ~\cite{MM}).

(2). \textbf{Pull-back.} Let $$\xymatrix{Z \ar[r]^\fs \ar[d]_\pi & \Gpdo \\ Y & }$$ be a $\cG$-principal Real bundle and $f: Y'\To Y$ be a Real continuous map. Then the pull-back $f^\ast Z := Y'\times_Y Z$ equipped with the involution $(\varrho',\tau)$ has the structure of a $\cG$-principal Real bundle over $Y'$. Indeed, the right Real $\cG$-action is given by the $\cG$-action on $Z$ and the generalized source map is $\fs'(y',z):=\fs(z)$.

(3). \textbf{Trivial bundles.} From the previous two remarks, we see that if $(Z,\tau)$ is any Real space together with a Real map $\varphi: Z\To \Gpdo$, then we get a $\cG$-principal Real bundle $\varphi^\ast U(\cG)$ over $Z$; its total space being the space $ Z\times_{\varphi,\Gpdo, r} \cG$. A Bundle of this form is called \emph{trivial} while a $\cG$-principal Real bundle which is locally of this form is called \emph{locally trivial}.
\end{rmks}


\subsection{Generalized morphisms of Real groupoids}

\begin{df}~\label{gen-hom}
A generalized morphism from a Real groupoid $(\Ga, \varrho)$ to a Real groupoid $(\cG, \rho)$ consists of a Real space $(Z,\tau)$, two maps $$\xymatrix{ \Gamo & Z \ar[l]_\fr \ar[r]^\fs & \Gpdo},$$
a left (Real) action of $\Ga$ with respect to $\fr$, a right (Real) action of $\cG$ with respect to $\fs$, such that

\begin{itemize}
\item[(i)] the actions commute, i.e. if $(z,g)\in Z \times_{\fs, \Gpdo,r}\cG$ and $(\g,z)\in \Ga\times_{s,\Gamo,\fr} Z$ we must have $\fs(\g z)= \fs(z)$, $\fr(zg)=\fr(z)$ so that $\g (zg)=(\g z)g$;
\item[(ii)] the maps $\fs$ and $\fr$ are Real in the sense that $\fs(\tau(z))=\rho(\fs(z))$ and $\fr(\tau(z))=\varrho(\fr(z))$ for any $z\in Z$;
\item[(iii)] $\fr: Z\To \Gamo$ is a locally trivial $\cG$-principal Real bundle.
\end{itemize}
\end{df}

\begin{ex}~\label{ex-gh}
Let $f: \Ga \To \cG$ be a Real strict morphism. Let us consider the fibre product $Z_f:= \Gamo \times_{f,\Gpdo, r} \cG$ and the maps $\fr: Z_f \To \Gamo, \ (y,g)\mto y$ and $\fs: Z_f \To \Gpdo, \ (y,g)\mto s(g)$. For $(\g, (y,g))\in \Ga \times_{s,\Gamo, \fr} Z_f)$, we set $\g.(y,g):=(r(\g),f(\g)g)$ and for $((y,g),g')\in Z_f\times_{\fs,\Gpdo,r} \cG$ we set $(y,g).g':= (y,gg')$. Using the definition of a strict morphism, it is easy to check that these maps are well defined and make $Z_f$ into a generalized morphism from $\Ga$ to $\cG$. Furthermore, the map $\tau$ on $Z_f$ defined by $\tau(y,g):=(\varrho(y),\rho(g))$ is a Real involution and then $Z_f$ is a Real generalized morphism.
\end{ex}

\begin{df}
A morphism between two such morphisms $(Z,\tau)$ and $(Z',\tau')$ is a $\Ga$-$\cG$-equivariant Real map $\vp:Z\To Z'$ such that $\fs = \fs'\circ \vp$ and $\fr=\fr'\circ \vp$. We say that the Real generalized homomorphism $(Z,\tau)$ and $(Z',\tau')$ are \texttt{isomorphic} if there exists such a $\vp$ which is at the same time a homeomorphism.
\end{df}

Compositions of Real generalized morphisms are defined by the following proposition.

\begin{pro}
Let $(Z',\tau')$ and $(Z",\tau")$ be  Real generalized homomorphisms from $(\Ga,\varrho)$ to $(\cG',\rho')$ and from $(\cG',\rho')$ to $(\cG,\rho)$ respectively. Then $$Z= Z'\times_{\cG'}Z":= (Z'\times_{\fs',\cG'^{(0)}, \fr"}Z")/_{(z',z")\sim (z'g',g'^{-1}z")}$$  with the obvious Real involution, defines a Real generalized morphism from $\gamgpd$ to $\grpd$.
\end{pro}

\begin{proof}
Let us first describe the structure maps $$\xymatrix{\Gamo & Z \ar[l]_\fr \ar[r]^\fs & \Gpdo}$$ and the actions. 

For $(z',z")\in Z$ we set $\fr (z',z"):=\fr'(z')$ and $\fs (z',z"):=\fs"(z")$. These are well defined and since $\fs(z'g',g'^{-1}z")=\fs"(g'^{-1}z")=\fs"(z")$ and $\fr(z'g',g'^{-1}z")=\fr'(z'g')=\fs'(z')$ from the point (i) in Definition~\ref{gen-hom}. The actions are defined by $\g.(z',z"):=(\g z',z")$ and $(z',z").g:=(z',z"g)$  for $(\g,(z',z"))\in \Ga \times_{s,\Gamo, \fr} Z$ and $((z',z"),g)\in Z \times_{\fs,\Gpdo,r}\cG$ while the Real involution is the obvious one: 
$$\tau(z',z"):=(\tau'(z'),\tau"(z")).$$

Now to show the local triviality of $Z$, notice that from (3) of Remarks~\ref{prop-bdles}, $Z'$ and $Z"$ are locally of the form $U\times_{\varphi',\cG'^{(0)},r'}\cG'$ and $V\times_{\varphi",\Gpdo,r}\cG$ respectively, where $\varphi': U\To \cG'^{(0)}$ and $\varphi":V\To \Gpdo$ are Real continuous maps, $U$ and $V$ subspaces of $\Gamo$ and $\cG'^{(0)}$ respectively. It turns out that by construction, $Z$ is locally of the form $W\times_{\varphi,\cG'^{(0)},r}\cG $ where  $W=U\times_{\varphi',\cG'^{(0)}}V$.
\end{proof}

\begin{df}
Given two Real generalized morphisms $(\Ga,\vr)\stackrel{(Z,\tau)}{\To} (\cG',\rho')$ and $(\cG',\rho')\stackrel{(Z',\tau')}{\To} (\cG,\rho)$, we define their composition $(Z'\circ Z,\tau):(\Ga,\vr)\To (\cG,\rho)$ to be $(Z\times_{\cG'}Z',\tau \times \tau')$.
\end{df}

\begin{rem}
It is easy to check that the composition of Real generalized homomorphisms is associative. 
For instance, if $$\xymatrix{\Ga \ar[r]^{(Z_1,\rho_1)} & \cG_1 \ar[r]^{(Z_2,\rho_2)} & \cG_2 \ar[r]^{(Z_3,\rho_3)} & \cG }$$ are given Real generalized morphisms, 
we get two Real generalized morphisms $Z=Z_1 \times_{\cG_1}(Z_2\times_{\cG_2}Z_3)$ and $Z'=(Z_1\times_{\cG_1}Z_2)\times_{\cG_2}Z_3$ between $(\Ga,\varrho)$ and $(\cG,\rho)$;
 notice that here $Z$ and $Z'$ carry the obvious Real involutions. Moreover, the map $Z\To Z', \ (z_1,(z_2,z_3))\mto ((z_1,z_2),z_3)$ is a $\Ga$-$\cG$-equivariant Real homeomorphism. Hence, there exists a category $\RG$ whose objects are Real locally compact groupoids and morphisms are isomorphism classes of Real generalized homomorphisms.
\end{rem}

\begin{lem}
Let $f_1, f_2: \Ga \to \cG$ be two Real strict homomorphisms. Then $f_1$ and $f_2$ define isomorphic Real generalized homomorphisms if and only if there exists a Real continuous map $\vp: \Gamo \To \cG$ such that $f_2(\g)=\vp(r(\g))f_1(\g)\vp(s(\g))^{-1}$.
\end{lem}

\begin{proof}
Le $\Phi: Z_{f_1}\To Z_{f_2}$ be a Real $\Ga$-$\cG$-equivariant homeomorphism, where $Z_{f_i}=\Gamo \times_{f_i,\Gpdo,r}\cG$. Then from the commutative diagrams 
$$\xymatrix{\Gamo & \ar[l]_{pr_1}  Z_{f_1} \ar[r]^{s\circ pr_2} \ar[d]^\Phi & \Gpdo \\ & Z_{f_2} \ar[ul]^{pr_1} \ar[ur]_{s\circ pr_2} & }$$
we have $\Phi(x,g)=(x,h)$ with $s(g)=s(h)$; and then there exists a unique element $\vp(x)\in \cG$ such that $h=\vp(x)g$. To see that this defines a continuous map $\vp:\Gamo \To \cG$, notice that for any $x\in \Gamo$, the pair $(x,f_1(x))$ is an element in $Z_{f_1}$, then $\vp(x)$ is the unique element in $\cG$ such that $\Phi(x,f_1(x))=(x,\vp(x)f_1(x))$. Furthermore, since $\Phi$ is Real, $\Phi(\vr(x),\rho(f_1(x)))=(\vr(x),\rho(\vp(x))\rho(f_1(x)))$ which shows that $\vp(\vr(x))=\rho(\vp(x))$ for any $x\in \Gamo$; \emph{i.e}. $\vp$ is Real.

Now for $\g \in \Ga$, take $x=s(\g)$, then from the $\Ga$-equivariance of $\Phi$, we have 
$$\Phi(\g\cdot(s(\g),f_1(s(\g)))) =\Phi(r(\g),f_1(\g))=\g\cdot\Phi(s(\g),f_1(s(\g)));$$
so that $$(r(\g),\vp(r(\g))f_1(\g))=(r(\g),f_2(\g)\vp(s(\g)))$$ and $f_2(\g)\cdot r(\vp(s(\g)))=\vp(r(\g))f_1(\g)\vp(s(\g))$; but $r(\vp(s(\g)))=s(f_2(\g))$ by definition of $\vp$ and this gives the desired relation.

The converse is easy to check by working backwards.
\end{proof}


\subsection{Morita equivalence}

Let $(\Ga,\vr)$ and $(\cG,\rho)$ be two Real groupoids. Suppose that $f:(\Ga,\vr)\To (\cG,\rho)$ is an isomorphism in the category $\RG_s$. In this case, we say that $(\Ga,\vr)$ and $(\cG,\rho)$ are \emph{strictly equivalent} and we write $(\Ga,\vr)\sim_{strict} (\cG,\rho)$. Now, consider the induced Real generalized morphisms $(Z_f,\tau_f): (\Ga,\vr)\To (\cG,\rho)$ and $(Z_{f^{-1}},\tau_{f^{-1}}):(\cG,\rho)\To (\Ga,\vr)$. Define \texttt{the inverse of $Z_f$} by $Z_f^{-1}:=\cG\times_{r,\Gpdo,f}\Gamo$ with the obvious Real structure also denoted by $\tau_f$. The map $Z_{f^{-1}}\To Z_f^{-1}$ defined by $(x,\g)\mto (f(\g),f^{-1}(x))$ is clearly a $\cG$-$\Ga$-equivariant Real homeomorphism; hence, $(Z_{f^{-1}},\tau_{f^{-1}})$ and $(Z_f^{-1},\tau_f)$ are isomorphic Real generalized morphisms from $(\cG,\rho)$ to $(\Ga,\vr)$. Notice that $Z_f^{-1}$ is $Z_f$ as space; thus, $(Z_f,\tau_f)$ is at the same time a Real generalized morphism from $(\Ga,\vr)$ to $(\cG,\rho)$ and from $(\cG,\rho)$ to $(\Ga,\vr)$. Furthermore, it is simple to check that $Z_f\circ Z_f^{-1}$ and $Z_{\Id_\cG}$ define isomorphic Real generalized morphisms from $(\cG,\rho)$ into itself, and likewise, $Z_f^{-1}\circ Z_f$ and $Z_{\Id_\Ga}$ are isomorphic Real generalized morphisms from $(\Ga,\vr)$ into itself.

\begin{df}~\label{def:Morita_equiv}
Two Real groupoids $(\Ga,\varrho)$ and $(\cG,\rho)$ are said to be \emph{Morita equivalent} if there exists
a Real space $(Z,\tau)$ that is at the same time a Real generalized morphism from $\Ga$ to $\cG$ and from $\cG$ to $\Ga$; 
that is to say that $\xymatrix{\Gamo & Z \ar[l]_\fr}$ is a $\cG$-principal \emph{Real} bundle and $\xymatrix{Z \ar[r]^\fs & \Gpdo}$ is a $\Ga$-principal Real bundle. 
\end{df}

\begin{rem} 
Given a Morita equivalence $(Z,\tau):(\Ga,\vr)\To (\cG,\rho)$, its inverse, denoted by $(Z^{-1},\tau)$, is $(Z,\tau)$ as Real space, and if $\flat: (Z,\tau)\To (Z^{-1},\tau)$ is the identity map, the left Real $\cG$-action on $(Z^{-1},\tau)$ is given by $g\cdot\flat(z):=\flat(z\cdot g^{-1})$, and the  right Real $\Ga$-action is given by $\flat(z)\cdot\g:=\flat(\g^{-1}\cdot z)$; $(Z^{-1},\tau)$ is the corresponding Real generalized morphism from $(\cG,\rho)$ to $(\Ga,\vr)$. 
\end{rem}

The discussion before Definition~\ref{def:Morita_equiv}  shows that the Real generalized morphism induced by a Real strict morphism is actually a Morita equivalence. However, the converse is not true. Moreover, there is a functor 
 \begin{eqnarray}
 \RG_s \To \RG,
 \end{eqnarray}
 where $\RG _s$ is the category whose objects are Real locally compact groupoids and whose morphisms are Real strict morphisms, given by $$f\mto Z_f.$$

\begin{df}[Real cover groupoid]~\label{Real-open-cover}
Let $\grpd$ be a Real groupoid. Let $\cU=\{U_j\}$ be a Real open cover of $X$. Consider the disjoint union $\coprod_{j\in J}U_j =\lbrace (j,x)\in J\times X \ : \ x\in U_j \rbrace$ with the Real structure $\rho^{(0)}$ given by $\rho^{(0)}(j,x):=(\bar{j},\rho(x))$ and define a Real local homeomorphism given by the projection $\pi: \coprod_jU_j \To X, \ (j,x)\mto x$. Then the set $$\cG[\cU]:= \lbrace (j_0,g,j_1)\in J\times \cG\times J \ : \ r(g)\in U_{j_0}, s(g)\in U_{j_1} \rbrace,$$ endowed with the involution $\rho^{(1)}(j_0,g,j_1):=(\bar{j_0},\rho(g),\bar{j_1})$ has a structure of a \emph{Real} locally compact groupoid whose unit space is $\coprod _j U_j$. The range and source maps are defined by $\tilde{r}(j_0,g,j_1):=(j_0,r(g))$ and $\tilde{s}(j_0,g,j_1):=(j_1,s(g))$; two triples are composable if they are of the form $(j_0,g,j_1)$ and $(j_1,h,j_2)$, where $(g,h)\in \cG^{(2)}$, and their product is given by $(j_0,g,j_1)\cdot(j_1,h,j_2):=(j_0,gh,j_2)$. The inverse of $(j_0,g,j_1)$ is $(j_1,g^{-1},j_0)$. 
\end{df}

It is a matter of simple verifications to check the following

\begin{lem}
Let $\grpd$ be a Real groupoid, and $\cU$ a Real open cover of $X$. Then the Real generalized morphism $Z_\iota:\cG[\cU]\To \cG$ induced from the canonical Real morphism $$\iota:\cG[\cU] \To \cG, \ (j_0,g,j_1)\mto g,$$
is a Morita equivalence between $(\cG[\cU],\rho)$ and $(\cG,\rho)$.	
\end{lem}

\begin{df}
Let $$\xymatrix{Z \ar[d]_\pi \ar[r]^\fs & \Gpdo \\ Y & }$$ be a locally trivial $\cG$-principal Real bundle. A section $\mathsf{s}:Y \To Z$ is said to be Real if $\mathsf{s}\circ \vr = \tau \circ \mathsf{s}$. Moreover, given a Real open cover $\lbrace U_j \rbrace_{j\in J}$ of $Y$, we say that a family of local sections $\mathsf{s}_j:U_j\To Z$ is \emph{globally Real} if for any $j\in J$, we have \begin{equation}
\mathsf{s}_{\bar{j}}\circ \vr = \tau \circ \mathsf{s}_j.
\end{equation}
\end{df}

\begin{lem}
Any locally trivial $\cG$-principal Real bundle $\pi: Z\To Y$ admits a globally Real family of local sections $\lbrace \mathsf{s}_j \rbrace_{j\in J}$ over some Real open cover $\lbrace U_j \rbrace$.  
\end{lem}

\begin{proof}
Choose a \emph{local trivialization} $(U_i,\vp_i)_{i\in I}$ of $Z$; \emph{i.e}. $\vp_i:U_i \To \Gpdo$ are continuous maps such that $\pi^{-1}(U_i)=:Z_{U_i} \cong U_i \times_{\vp_i,\Gpdo,r}\cG$ with $\tau_{Z_{U_i}}=(\vr,\rho)$. It turns out that $Z_{U_{(i,\epsilon)}}\cong U_{(i,\epsilon)} \times_{\vp_i^\epsilon,\Gpdo,r}\cG$, where $\vp_i^\epsilon:=\rho^\epsilon \circ \vp_i \circ \vr^\epsilon: U_{(i,\epsilon)}\To \Gpdo$ is a well defined continuous map and $U_{(i,\epsilon)}:=\vr^\epsilon(U_i)$ for $(i,\epsilon)\in I\times \ZZ_2$. However, for $(i,\epsilon)\in I\times \ZZ_2$, there is a homeomorphism $\xymatrix{U_{(i,\epsilon)}\times_{\vp_i^\epsilon,\Gpdo,r}\cG \ar[r]^{(\vr,\rho)} & U_{\overline{(i,\epsilon)}}\times_{\vp_i^{\epsilon +1},\Gpdo,r}\cG}$. Now, putting $\mathsf{s}_{(i,\epsilon)}: U_{(i,\epsilon)}\To Z, \ x \mto (x,\vp_i^\epsilon (x))$, we obtain the desired sections.
\end{proof}

For the remainder of this subsection we will need the following construction. 

Let $(Z,\tau)$ be a Real space and $(\Ga,\varrho)$ a Real groupoid together with a continuous Real map $\vp: Z\To \Gamo$. Then we define an induced groupoid $\vp^\ast \Ga$ over $Z$ in which the arrows from $z_1$ to $z_2$ are the arrows in $\Ga$ from $\vp(z_1)$ to $\vp(z_2)$; i.e. 
 $$\vp^\ast \Ga:= Z\times_{\varphi,\Gamo,r}\Ga \times_{s,\Gamo,\vp}Z  \ ,$$
 and the product is given by $(z_1,\g_1,z_2).(z_2,\g_2,z_3)=(z_1,\g_1\g_2,z_3)$ whenever $\g_1$ and $\g_2$ are composable, while the inverse is given by $(z,\g,z')^{-1}=(z',\g^{-1},z)$. Moreover, the triple $(\rho,\varrho,\rho)$ defines a Real structure $\vp^\ast \varrho$ on $\vp^\ast \Ga$ making it into a Real groupoid $(\vp^\ast \Ga, \vp^\ast \varrho )$ that we will call \emph{the pull-back} of $\Ga$ over $Z$ via $\varphi$.

\begin{lem}~\label{lem-pullback}
Given a continuous locally split Real open map $\vp: Z \To \Gamo$, then the Real groupoids $\Ga$ and $\vp^\ast \Ga$ are Morita equivalent.
\end{lem}

\begin{proof}
Consider the Real strict homomorphism $\tilde{\vp}: \vp^\ast \Ga \To \Ga$ defined by $(z_1,\g, z_2) \mto \g$. Then by Example ~\ref{ex-gh} we obtain a Real generalized homomorphism $\xymatrix{Z & Z_{\tilde{\vp}} \ar[l]_{\pi_1} \ar[r]^{s\circ \pi_2} & \Gamo}$ with $Z_{\tilde{\vp}} := Z\times_{\tilde{\vp}, \Gamo, r} \Ga$, $\pi_1$ and $\pi_2$ the obvious projections, and where $Z \hookrightarrow \vp^\ast \Ga$ by $z\mto (z,\vp(z), z)$. Now using the constructions of Example~\ref{ex-gh}, it is very easy to check that $Z_{\tilde{\vp}}$ is in fact a Morita equivalence.
\end{proof}

\begin{pro}~\label{pro-pullback}
Two Real groupoids $(\Ga, \vr)$ and $(\cG,\rho)$ are Morita equivalent if and only if there exist a Real space $(Z,\tau)$ and two continuous Real maps $\vp: Z\To \Gamo$ and $\vp':Z\To \Gpdo$ such that $\vp^\ast \Ga \cong (\vp')^\ast \cG$ under a Real (strict) homeomorphism.
\end{pro}

\begin{proof}
Let $\xymatrix{\Gamo & Z \ar[l]_\fr \ar[r]^\fs & \Gpdo}$ be a Morita equivalence. Let us define $$\Ga \ltimes Z \ast Z \rtimes \cG := \lbrace (\g, z_1,z_2,g)\in (\Ga \times_{s,\Gamo,\fr}Z)\times (Z\times_{\fs,\Gpdo, r}\cG) \ | \ z_1g=\g z_2 \rbrace \ .$$ 
This defines a Real groupoid over $Z$ whose range and source maps are defined by the second and the third projection respectively, the product is given by $$(\g,z_1,z_2,g)\cdot(\g',z_2,z_3,g')=(\g \g',z_1,z_3,gg'),$$ provided that $\g , \g'\in \Ga^{(2)}$ and $g,g' \in \cG^{(2)}$, and the inverse of $(\g,z_1,z_2,g)$ is $(\g^{-1},z_2,z_1,g^{-1})$. Now, for a given triple $(z_1,\g,z_2)\in \fr^\ast \Ga$, the relations $\fr(z_1)=r(\g)$ and $\fr(z_2)=s(\g)$ give $\fr(\g z_2)=\fr(z_1)$; then since $\fr: Z\To \Gamo$ is a Real $\cG$-principal bundle, there exists a unique $g\in \cG$ such that $\g z_2 = z_1 g$. This gives an injective homomorphism $\Psi: \fr^\ast \Ga \To \Ga \ltimes Z \ast Z \rtimes \cG \ , (z_1,\g,z_2)\mto (\g,z_1,z_2,g)$ which respects the Real structures. In the other hand, the map $\Phi: \Ga \ltimes Z \ast Z \rtimes \cG \To \fr^\ast \Ga \ , (\g,z_1,z_2,g)\mto (z_1,\g,z_2)$ is a well defined Real homomorphism that is injective and Real. Moreover, these two maps are, by construction, inverse to each other so that we have a Real  homeomorphism $\fr^\ast \Ga \cong \Ga \ltimes Z \ast Z \rtimes \cG$. Furthermore, since $\fs: Z\To \Gpdo$ is a Real $\Ga$-principal bundle, we can use the same arguments to show that $\fs^\ast \cG \cong \Ga \ltimes Z \ast Z \rtimes \cG$ under a Real homeomorphism.

Conversely, if $\vp: Z \To \Gamo$ and $\vp': Z\To \Gpdo$ are given continuous Real maps and  $f: \vp^\ast \Ga \To (\vp')^\ast \Gpdo$ is a Real homeomorphism of groupoids, then the induced Real generalized homomorphism $$\vp^\ast \Ga \stackrel{Z_f}{\To}(\vp')^\ast \cG $$ is a Morita equivalence and Lemma~\ref{lem-pullback} completes the proof.
\end{proof}

The following example provides a characterization of groupoids Morita equivalent to a given Real space.

\begin{ex}~\label{ex:X-Morita}
Let $(X,\rho), (Y,\vr)$ be a locally compact Hausdorff Real spaces, and let $\pi:(Y,\vr)\To (X,\rho)$ be a continuous locally split Real open map. Form the Real groupoid $\xymatrix{Y^{[2]}\dar[r] & Y}$, where $Y^{[2]}$ is the fibered-product $Y\times_{\pi,X,\pi}Y$ equipped with the obvious Real structure; the groupoid structure on $Y^{[2]}$ is: 
\begin{eqnarray*}
		s(y_1,y_2):= y_2; & r(y_1,y_2):=y_1;\\
		(y_1,y_2)^{-1}:= (y_2,y_1); & (y_1,y_2)\cdot (y_2,y_3):=(y_1,y_3). 
	\end{eqnarray*}	
Then the Real groupoids $\xymatrix{Y^{[2]}\dar[r] & Y}$	and $\xymatrix{X\dar[r]& X}$ are Morita equivalent. Indeed, we have $\pi^\ast X\sim_{Morita}X$, thanks to Lemma~\ref{lem-pullback}; but $\pi^\ast X$ clearly identifies with $Y^{[2]}$ as Real groupoids. 

Conversely, suppose $(\Ga,\vr)$ is a Real groupoids Morita equivalent to $X$. Then in view of Proposition~\ref{pro-pullback}, there is a Real space $(Z,\tau)$, two continuous locally split Real open maps $\fs:Z\To X, \fr:Z\To Y$ such that $\fs^\ast X\cong \fr^\ast \Ga$ as Real groupoids over $Z$. In particular, $\fr:Z\To Y$ is a principal Real $X$-bundle, so that the Real space $Y$ is homeomorphic to the quotient Real space $Z/X=Z$. Thus, we have isomorphism of Real spaces $\fr^\ast\Ga=Z\times_Y \Ga\times_YZ\cong Y\times_Y\Ga\times_YY\cong \Ga$. Moreover, we have $\fs^\ast X\cong Z^{[2]}$ as Real spaces. Therefore, the Real groupoids $\gamgpd$ and $\xymatrix{Z^{[2]}\dar[r] & Z}$ as isomorphic.
\end{ex}

\begin{pro}[cf. Proposition 2.3~\cite{TLX}]~\label{pro:gen_strict}
 Any Real generalized morphism $$\xymatrix{\Gamo & Z \ar[l]_\fr \ar[r]^\fs & \Gpdo}$$ is obtained by composition of the canonical Morita equivalence between $(\Ga,\varrho)$ and $(\Ga [ \cU],\varrho)$, where $\cU$ is an open cover of $\Gamo$, with a Real strict morphism $f_\cU:\Ga [ \cU ] \To \cG$ (i.e. its induced morphism in the category $\RG$).
\end{pro}

\begin{proof}
 From Lemma~\ref{lem-pullback}, there is a Real Morita equivalence $Z_{\tilde{\fr}}: \fr^\ast \Ga \To \Ga$ and the Real homeomorphism $\fr^\ast \Ga \cong \Ga \ltimes Z \ast Z\rtimes \cG$ induces a Real strict homomorphism $f: \fr^\ast \Ga \To \cG$ given by the fourth projection, and hence a Real generalized homomorphism $Z_f: \fr^\ast \Ga \To \cG$. Furthermore, by using the construction of these generalized homomorphisms, it is easy to check that the composition $Z_{\tilde{\fr}}\times_\Ga Z$ is $\fr^\ast \Ga$-$\cG$-equivariently homeomorphic to $Z$ (under a Real homeomorphism); i.e, the diagram $$\xymatrix{\Ga \ar[dr]_Z & \fr^\ast \Ga \ar[l]_{Z_{\tilde{\fr}}} ^\cong \ar[d]^{Z_f} \\ & \cG} $$ is commutative in the category $\RG$. 
 
 Consider a Real open cover $\cU= \lbrace U_j\rbrace$ of $\Gamo$ together with a globally Real family of local sections $\mathsf{s}_j:U_j\To Z$ of $\fr:Z\To \Gamo$. Then, setting $(j_0,\g,j_1)\mto (\mathsf{s}_{j_0}(r(\g)),\g,\mathsf{s}_{j_1}(s(\g)))$ for $(j_0,\g,j_1)\in \Ga[\cU]$, we get a Real strict homomorphism $\tilde{\mathsf{s}}: \Ga [\cU ] \To \fr^\ast \Ga$ such that the composition $\Ga[ \cU] \To \fr^\ast \Ga \To \Ga$ is the canonical  map $\iota$ described in Example~\ref{Real-open-cover}. Then, $f\circ \tilde{\mathsf{s}}:\Ga[ \cU] \To \cG$ is the desired Real strict homomorphism.
\end{proof}

This proposition leads us to think of a Real generalized homomorphism from a Real groupoid $(\Ga,\vr)$ to a Real groupoid $(\cG,\rho)$ as a Real strict morphism $f_\cU: (\Ga [\cU],\vr)\To (\cG,\rho)$, where $\cU$ is a Real open cover of $\Gamo$.

To refine this point of view, given two Real groupoids $(\Ga,\vr)$ and $(\cG,\rho)$, let $\Omega$ denote the collection of such pairs $(\cU,f_\cU)$. We say that two pairs $(\cU,f_\cU)$ and $(\cU',f_{\cU'})$ are \emph{isomorphic} provided that $Z_{f_\cU}\circ Z_{\iota_{\cU}}^{-1}\cong Z_{f_{\cU'}}\circ Z_{\iota_{\cU'}}^{-1}$, where $\iota_{\cU}: (\Ga[\cU],\vr)\To (\Ga,\vr)$ and $\iota_{\cU'}:(\Ga[\cU'],\vr)\To (\Ga,\vr)$ are the canonical morphisms; this clearly defines an equivalence relation. We denote by $\Omega\left((\Ga,\vr),(\cG,\rho)\right)$ the set of isomorphism classes of elements of $\Omega$.

 Suppose that $(\cU,f_\cU):(\Ga,\vr)\To (\cG'.\rho')$ is an equivalence class in $\Omega\left((\Ga,\vr),(\cG',\rho')\right)$ and $(\cV,f_\cV):(\cG',\rho')\To (\cG,\rho)$ is an element in $\Omega\left((\cG',\rho'),(\cG,\rho)\right)$. Let $\iota_{\cG'}: \cG'[\cV]\To \cG'$ be the canonical morphism, and let $Z_{\iota_{\cG'}}^{-1}: (\cG',\rho')\To (\cG'[\cV],\rho')$ be the inverse of $Z_{\iota_{\cG'}}$. Next, we apply Proposition ~\ref{pro:gen_strict} to the Real generalized morphism $Z_{\iota_{\cG'}}^{-1}\circ Z_{f_\cU}: \Ga[\cU]\To \cG'[\cV]$ to get a Real open cover $\cU'$ of $\Gamo$ containing $\cU$ and a Real strict morphism $\vp_{\cU'}:(\Ga[\cU'],\vr)\To (\cG'[\cV],\rho')$. Then, we pose 
\begin{eqnarray}
(\cV,f_\cV) \circ (\cU,f_\cU):=(\cU',f_{\cU'}),
\end{eqnarray}
with $f_{\cU'}=f_{\cV}\circ \vp_{\cU'}$; thus we get an element of $\Omega\left((\Ga,\vr),(\cG,\rho)\right)$. It follows that there exists a category $\RG_\Omega$ whose objects are Real groupoids, and in which a morphism from $(\Ga,\vr)$ to $(\cG,\rho)$ is a class $(\cU,f_\cU)$ in $\Omega\left((\Ga,\vr),(\cG,\rho)\right)$.

\begin{ex}
Any Real strict morphism $f:(\Ga,\vr)\To (\cG,\rho)$ can be identified with the pair $(\Gamo,f)$, by considering the trivial Real open cover $\Gamo$ consisting of one set, and by viewing the groupoid $\Ga$ as the cover groupoid $\Ga[\Gamo]$. In particular, $\RG_s$ is a subcategory of $\RG_\Omega$.
\end{ex}

\begin{ex}
Suppose that $(Z,\tau):(\Ga,\vr)\To (\cG,\rho)$ is a Real generalized morphism. Then, Proposition ~\ref{pro:gen_strict} provides a unique class $(\cU,f_\cU)\in \Omega((\Ga,\vr),(\cG,\rho))$.
\end{ex}

\begin{rem}
Note that a class $(\cU,f_\cU)\in \Omega\left((\Ga,\vr),(\cG,\rho)\right)$ is an isomorphism in $\RG_\Omega$ if there exists $(\cV,f_\cV)\in \Omega\left((\cG,\rho),(\Ga,\vr)\right)$ such that 
\begin{equation}
Z_{f_\cU}\circ Z_{\iota_\cU}^{-1}\circ Z_{f_\cV}\cong Z_{\iota_\cV} \ \text{and} \ Z_{f_\cV}\circ Z_{\iota_\cV}^{-1}\circ Z_{f_\cU} \cong Z_{\iota_\cU},
\end{equation}
where $\iota_\cU: (\Ga[\cU],\vr)\To (\Ga,\vr)$ and $\iota_\cV: (\cG[\cU],\rho)\To (\cG,\rho)$ are the canonical morphisms.
\end{rem}

\begin{pro}
Define $\mathsf{F}:\RG \To \RG_\Omega$ by
\begin{eqnarray}
\mathsf{F}(Z,\tau):=(\cU,f_\cU),
\end{eqnarray}
where, if $(Z,\tau): (\Ga,\vr)\To (\cG,\rho)$ is a class of Real generalized morphisms, $(\cU,f_\cU)$ is the class of pairs corresponding to $(Z,\tau)$. \

Then $\mathsf{F}$ is a functor; furthermore, $\mathsf{F}$ is an isomorphism of categories.
\end{pro}

\begin{proof}
Suppose that $(Z,\tau):(\Ga,\vr)\To (\cG',\rho'), \ (Z',\tau'):(\cG',\rho')\To (\cG,\rho)$ are morphisms in $\RG$. Let $\mathsf{F}(Z'\circ Z,\tau \times \tau')=(\cU,f_\cU)\in \Omega\left((\Ga,\vr),(\cG,\rho)\right)$, $\mathsf{F}(Z,\tau)=(\cU',f_{\cU'})\in \Omega\left((\Ga,\vr),(\cG',\rho')\right)$, and $\mathsf{F}(Z',\tau')=(\cV,f_\cV)\in \Omega\left((\cG',\rho'),(\cG,\rho)\right)$. Consider a Real open cover $\tilde{\cU}$ of $\Gamo$ containing $\cU'$ and a Real morphism  $\vp_{\tilde{\cU}}:(\Ga[\tilde{\cU}],\vr)\To (\cG'[\cV],\rho')$ such that $Z_{\vp_{\tilde{\cU}}}\circ Z_{i}^{-1}\cong Z_{\iota_\cV}^{-1}\circ Z_{f_{\cU'}}$ as Real generalized morphisms from $(\Ga[\cU'],\vr)$ to $(\cG'[\cV],\rho')$, where $i:(\Ga[\tilde{\cU}],\vr)\To (\Ga[\cU'],\vr)$ and $\iota_\cV:(\cG'[\cV],\rho')\To (\cG',\rho')$ are the canonical morphisms. Note that if $\iota_{\tilde{\cU}}:(\Ga[\tilde{\cU}],\vr)\To (\Ga,\vr)$ is the canonical morphism, then $\iota_{\tilde{\cU}}=\iota_{\cU'}\circ i$; hence, $Z_{\iota_{\tilde{\cU}}}^{-1}\cong Z_i^{-1}\circ Z_{\iota_{\cU'}}^{-1}$ by functoriality. 

On the other hand, $\mathsf{F}(Z',\tau')\circ \mathsf{F}(Z,\tau)= (\cV,f_\cV) \circ (\cU,f_\cU)=(\tilde{\cU},f_{\tilde{\cU}})$, where $f_{\tilde{\cU}}=f_\cV \circ \vp_{\tilde{\cU}}$. Henceforth, 
\[
Z_{f_{\tilde{\cU}}}\circ Z_{\iota_{\tilde{\cU}}}^{-1}  \cong Z_{f_\cV}\circ Z_{\vp_{\tilde{\cU}}}\circ Z_i^{-1}\circ Z_{\iota_{\cU'}}^{-1}  \cong Z_{f_\cV}\circ Z_{\iota_\cV}^{-1}\circ Z_{f_{\cU'}}\circ Z_{\iota_{\cU'}}^{-1} \
  \cong Z'\circ Z,\]  
which shows that $\mathsf{F}(Z'\circ Z,\tau \times \tau')\cong \mathsf{F}(Z',\tau')\circ \mathsf{F}(Z,\tau)$, and thus $\mathsf{F}$ is a functor.

Now, it is not hard to see that we get an inverse functor for $\mathsf{F}$ by defining 
\begin{eqnarray}
\mathsf{Z}:\RG_\Omega \To \RG, (\cU,f_\cU)\mto (Z_{f_\cU}\circ Z_{\iota_\cU}^{-1},\tau),
\end{eqnarray}
where $\tau$ is defined in an obvious way.
\end{proof}



\subsection{Real graded twists}

In this section we define \emph{Real graded twists}.

\begin{df}(cf.~\cite[\S 2]{Kum88})
Let $\gamgpd$ be a Real groupoid and let $\bfS$ be a Real Abelian group. A \emph{Real graded $\bfS$-twist} $(\wGa,\del)$ over $\Ga$ is the data of
\begin{itemize}
\item[(i)] a Real groupoid $\wGa$ whose unit space is $Y$, together with a Real strict homomorphism $\pi: \wGa \To \Ga$ which restricts to the identity in $Y$, 
\item[(ii)] a (left) Real action of $\bfS$ on $\wGa$ which is compatible with the partial product in $\wGa$ making $\xymatrix{ \wGa \ar[r]^\pi  & \Ga} $ a (left) Real $\bfS$-principal bundle, and 
\item[iii] a strict homomorphism $\delta: \Ga \To \ZZ_2$, called \emph{the grading}, such that $\delta(\bar{\g}) = \delta(\g)$ for any $\g \in \Ga$.
\end{itemize}
In this case we refer to the triple $(\wGa,\Ga,\del)$ as a \emph{Real graded $\bfS$-twist}, and it is sometimes symbolized by the "extension" $$\xymatrix{\bfS \ar[r] & \wGa \ar[r]^{\pi} & \Ga \ar[d]^\del \\ & & \ZZ_2}$$
\end{df}

\begin{ex}[The trivial twist]~\label{triv-twist}
Given Real groupoid $\Ga$, we form the product groupoid $\Ga \times \bfS$ and we endow it with the Real structure $\overline{(\g,\lambda)}:=(\bar{\g},\bar{\lambda})$ for. Let $\bfS$ act on $\Ga \times \bfS$ by multiplication with the second factor. Then $\cT_0:=(\Ga \times \bfS, 0)$ is a Real graded twist of $\Ga$, where $0: \ZZ_2 \To \ZZ_2$ is the zero map. This element is called  \emph{the trivial Real graded $\bfS$-twist over} $\Ga$.
\end{ex}

\begin{ex}
Let $Y$ be a locally compact Real space and $\lbrace U_i \rbrace_{i\in I\times \{\pm1\}}$ be a good Real open. Let us consider the Real groupoid $\xymatrix{Y[\cU] \dar[r] & \coprod_iU_i}$, and the space $Y\times \bfS$ together with the Real structure $(y,\lambda)\mto (\bar{y},\bar{\lambda})$ and the Real $\bfS$-action given by the multiplication on the second factor. We write $x_{i_0i_1}$ for $(i_1,x,i_1)\in Y[\cU]$. There is a canonical Real morphism $\del:Y[\cU]\To \ZZ_2$ given by $\del(x_{i_0i_1}):=\ve_0+\ve_1$ for $i_0=(i_0',\ve_0), \ i_1=(i_1',\ve_1)\in I$. Then, a Real graded $\bfS$-twist $(\wGa,Y[\cU],\del)$ consists of a family of principal Real $\bfS$-bundles $\wGa_{ij}\cong U_{ij}\times \bfS$ subject to the multiplication 
\[(x_{i_0i_1},\lambda_1)\cdot (x_{i_1i_2},\lambda_2)=(x_{i_0i_2},\lambda_1\lambda_2c_{i_0i_1i_2}(x)),\]
where $c=\{c_{i_0i_1i_2}\}$ is a family of continuous maps $c_{i_0i_1i_2}:U_{i_0i_1i_2}\To \bfS$ which is a $2$-cocycle such that $c_{\bar{i}_0\bar{i}_1\bar{i}_2}(\bar{x})=\overline{c_{i_0i_1i_2}(x)}$ for all $x\in U_{i_0i_1i_2}=U_{i_0}\cap U_{i_1}\cap U_{i_2}$. The pair $(\del,c)$ will be called \emph{the Dixmier-Douady class of $(\wGa,Y[\cU],\del)$} (cf. Section~\ref{sect:HR2-HR1-vs-Ext}).   
\end{ex}

\begin{ex}~\label{ex:triv-gpd-line}
Let $\gamgpd$ be a Real groupoid, and let $J:\Lam\To Y $ be a Real $\bfS$-principal bundle. Then the tensor product $r^\ast\Lam\otimes \overline{s^\ast\Lam}$, which is a Real $\bfS$-principal bundle over $\Ga$, naturally admits the structure of Real groupoid over $Y$, so that $(r^\ast\Lam\otimes \overline{s^\ast\Lam},0)$ is a Real graded $\bfS$-twist over $\Ga$.  	
\end{ex}

There is an obvious notion of strict morphism of Real graded $\bfS$-twists. For instance, two Real graded $\bfS$-twists $(\wGa_1,\Ga, \delta_1)$ and $(\wGa_2,\Ga,\delta_2)$ are isomorphic if there exists a Real $\bfS$-equivariant isomorphism of groupoids  $f: \wGa_1 \To \wGa_2$ such that the diagram 
$$\xymatrix{\wGa_1 \ar[r]^{\pi_1} \ar[d]_f & \Ga \\ \wGa_2 \ar[ur]_{\pi_2} & }$$ commutes in the category $\RG_s$.
In particular, we say that $(\wGa,\del)$ is \emph{strictly trivial} if it isomorphic to the trivial Real graded groupoid $(\Ga\times \bfS,0)$. By $\wRTw(\Ga,\bfS)$ we denote the set of strict isomorphism classes of Real graded $\bfS$-twists over $\Ga$. The class of $(\wGa,\del)$ in $\wRTw(\Ga,\bfS)$ is denoted by $[\wGa,\del]$.

\begin{df}(compare~\cite{Kum88,Tu,FHT})
 Given two Real graded $\bfS$-twists $\cT_1=(\wGa_1,\delta_1)$ and $\cT_2=(\wGa_2,\delta_2)$ over $\cG$, we define their tensor product $\cT_1\hat{\otimes}\cT_2=(\wGa_1\hat{\otimes}\wGa_2,\del_1+\del_2)$ by the \emph{Baer sum} of $\cT_1$ and $\cT_2$ defined as follows. Define the groupoid $\wGa_1 \hat{\otimes}\wGa_2$ as the quotient 
\begin{equation}~\label{eq:Baer}
\wGa_1 \times_{\Ga} \wGa_2/\bfS := \lbrace (\tilde{\g_1},\tilde{\g_2})\in \wGa_1 \times_{\pi_1,\Ga,\pi_2} \wGa_2 \rbrace /_{(\tilde{\g_1},\tilde{\g_2}) \sim (\lambda \tilde{\g_1},\lambda^{-1} \tilde{\g_2})},
\end{equation}
 where $\lambda \in \bfS$, together with the obvious Real structure. The projection $\pi_1 \otimes\pi_2$ is just $\pi_i$ and $\delta= \delta_1 + \delta_2$ is given by $\delta(\g)=\delta_1(\g)+\delta_2(\g)$. \\
The product in the Real groupoid $\wGa_1\hat{\otimes}\wGa_2$ is 
 \begin{equation}
 (\tilde{\g_1},\tilde{\g_2})(\tilde{\g'_1},\tilde{\g'_2}):= (-1)^{\delta_2(\g_2) \delta_1(\g'_1)}(\tilde{\g_1} \tilde{\g'_1},\tilde{\g_2} \tilde{\g'_2}),
 \end{equation}
 whenever this does make sense and where $\g_i=\pi_2(\tilde{\g_i}), \ i=1,2$.
\end{df}

\begin{lem}(~\cite[p.4]{Tu})
Given $[\wGa_i,\del_i]\in \wRTw(\Ga,\bfS), i=1,2$, set 
$$[\wGa_1,\del_1]+[\wGa_2,\del_2]:=[\wGa_1\hat{\otimes}\wGa_2,\del_1+\del_2].$$
Then, under this sum, $\wRTw(\Ga,\bfS)$ is an Abelian group whose zero element is given by the class of the trivial element $\cT_0=(\cG\times\bfS,0)$.
\end{lem}

\begin{proof}
The tensor product defined above is commutative in $\wRTw(\Ga,\bfS)$. Indeed, the groupoid $\wGa_2\hat{\otimes}\wGa_1=\wGa_2 \times_\Ga \wGa_1 /\bfS$ is endowed with the multiplication $$(\tilde{\g_2},\tilde{\g_1})(\tilde{\g'_2},\tilde{\g'_1})=(-1)^{\delta_1(\g_1)\delta_2(\g'_2)}(\tilde{\g_2} \tilde{\g'_2},\tilde{\g_1}\tilde{\g'_1}).$$
 Then the map 
$$\wGa_1\hat{\otimes}\wGa_2 \To \wGa_2\hat{\otimes}\wGa_1 \ , (\tilde{\g_1},\tilde{\g_2}) \mto (-1)^{\delta_1(\g_1) \delta_2(\g_2)}(\tilde{\g_2},\tilde{\g_1})$$
is a Real $\bfS$-equivariant isomorphism of groupoids. 

Now define the inverse of $(\wGa,\delta)$ is $(\wGa^{\op},\delta)$ where $\wGa^{op}$ is $\wGa$ as a set but, together with the same Real structure, but the $\bfS$-principal bundle structure is replaced by the conjugate one, \emph{i.e.} $\lambda \tilde{\g}^{\op}=(\bar{\lambda}\tilde{\g})^{\op}$, and the product $\ast_{\op}$ in $\wGa^{op}$ is $$\tilde{\g}\ast_{\op}\tilde{\g'}:=(-1)^{\delta(\g)\delta(\g')}\tilde{\g}\tilde{\g'}.$$
Now it is easy to see that the map $$\Ga \times \bfS \To \wGa \times_\Ga \wGa^{\op} /\bfS \ , (\g,\lambda) \mto (\lambda \tilde{\g},\tilde{\g}) \ ,$$ where $\tilde{\g} \in \wGa$ is any lift of $\g \in \Ga$, is an isomorphism.
\end{proof}

We have the following criteria of strict triviality; the proof is the same as in~\cite[Proposition 2.8]{TLX}.

\begin{pro}~\label{pro:crit-triv-twist}
Let $(\wGa,\del)$ be a Real graded $\bfS$-twist over the Real groupoid $\gamgpd$. The following are equivalent:
\begin{itemize}
 		\item[(i)] $(\wGa,\del)$ is strictly trivial.
 		\item[(ii)] $\del(\g)=0, \forall \g\in \Ga$, and there exists a Real strict homomorphism $\sigma:\Ga\To \wGa$ such that $\pi\circ \sigma=\Id$.
 		\item[(iii)] $\del(\g)=0, \forall \g\in \Ga$,, and there exists a Real $\bfS$-equivariant groupoid homomorphism $\vp:\wGa\To \bfS$.
 	\end{itemize} 	
\end{pro}

\begin{ex}~\label{ex:triv-twist2}
Let $J:\Lam\To Y$ be a Real $\bfS$-principal bundle with a Real (left) $\Ga$-action that is compatible with the $\bfS$-action; in other words $\xymatrix{Y  & \Lam \ar[l]^J \ar[r] & \star}$ is a Real generalized homomorphism from $\Ga$ to $\bfS$. Then, the Real $\Ga$-action induces an $\bfS$-equivariant isomorphism $\Lam_{s(\g)}\ni v\mto \g\cdot v\in \Lam_{r(\g)}$ for every $\g\in \Ga$. Hence, there is a Real $\bfS$-equivariant groupoid isomorphism $\vp:r^\ast\Lam\otimes \overline{s^\ast\Lam}\To \Ga\times \bfS$ defined as follows. If $(v,\flat(w))\in \Lam_{r(\g)}\otimes \overline{\Lam}_{s(\g)}$, there exists a unique $\lambda\in \bfS$ such that $\g\cdot w=v\cdot \lambda$. We then set $$\vp([v,\flat(w)]):=(\g,\lambda).$$ 
The inverse of $\vp$ is $\vp'(\g,\lambda):=[v_\g,\overline{\g^{-1}\cdot v_\g}]$, where  for $\g\in \Ga$, $v_\g$ is any lift of $r(\g)$ through the projection $J$.	
\end{ex}

Observe that the set of Real graded $\bfS$-twists of the from $(r^\ast\Lam \otimes\overline{s^\ast\Lam},0)$ over $\Ga$ (cf. Example~\ref{ex:triv-gpd-line}) is a subgroup of $\wRTw(\Ga,\bfS)$. By $\Rext(\Ga,\bfS)$ we denote the quotient of $\wRTw(\Ga,\bfS)$ by this subgroup.

Let us show that $\Rext(\cdot,\bfS)$ is functorial in the category $\RG_s$. Let $\Ga, \ \Ga'$ be two Real groupoids, and let $f:\Ga'\To \Ga$ be a morphism in $\RG_s$. Suppose that $\cT=(\wGa,\delta)$ is a Real graded $\bfS$-twist over $\Ga$. Then, the pull-back $f^\ast \wGa:=\wGa\times_{\pi,\Ga,f}\Ga'$ of the Real $\bfS$-principal bundle $\pi:\wGa\To \Ga$, on which the Real groupoid structure is the one induced from the product Real groupoid $\wGa \times \Ga'$, defines a Real graded twist
\begin{eqnarray}
f^\ast \cT :=\xymatrix{\bfS \ar[r] & f^\ast \wGa \ar[r]^{f^\ast \pi} & \Ga' \ar[d]^{f^\ast \delta} \\ & & \ZZ_2}
\end{eqnarray}
where $f^\ast \pi(\tilde{\g},\g'):=\g'$, $f^\ast \delta(\g'):=\delta(f(\g'))\in \ZZ_2$, and the Real left $\bfS$-action on $f^\ast \wGa$ being given by $\lambda\cdot(\tilde{\g},\g')=(\lambda \tilde{\g},\g')$. Suppose now that $\cT_i=(\wGa_i, \del_i), \ i=1,2$ are representatives in $\Rext(\Ga,\bfS)$. Then, $f^\ast(\cT_1\hat{\otimes} \cT_2)=f^\ast \cT_1\hat{\otimes}f^\ast \cT_2$; indeed, 
\[ f^\ast(\wGa_1\hat{\otimes}\wGa_2)=\left(\wGa_1 \times_\Ga \wGa_2/\bfS\right)\times_\Ga \Ga' \cong \left((\Ga_1\times_\Ga \Ga')\times_\Ga (\wGa_2 \times_\Ga \Ga')\right)/\bfS = f^\ast \wGa_1\hat{\otimes}f^\ast \wGa_2.
\]
Moreover, it is easily seen that if $\cT_1$ and $\cT_2$ are equivalent in $\Rext(\Ga,\bfS)$, then so are $f^\ast\cT_1$ and $f^\ast\cT_2$. Thus, $f$ induces a morphism of Abelian groups $f^\ast: \Rext(\Ga,\bfS)\To \Rext(\Ga',\bfS)$. We then have proved this

\begin{lem}~\label{lem:functor_RTw}
The correspondence
\begin{eqnarray}
\Rext(\cdot,\bfS): \RG_s \To \mathfrak{Ab}, \Ga\mto \Rext(\Ga,\bfS), f\mto f^\ast,
\end{eqnarray}
where $\mathfrak{Ab}$ is the category of Abelian groups, is a contravariant functor. In particular, $\Rext(\cG,\bfS)$ is invariant under Real strict isomorphisms.
\end{lem}


\subsection{Real graded central extensions}

In this subsection we introduce Real graded central extensions of Real groupoids, by adapting~\cite{Kum88,KMRW, FHT,Tu} to our context. 

\begin{df}~\label{df:S-equiv-gen-mor}
Let $(\wGa_i,\Ga_i,\delta_i),i=1,2$, be Real graded $\bfS$-twists. Then a Real generalized homomorphism $Z:\wGa_1 \To \wGa_2$ is said to be $\bfS$-equivariant if there is a Real action of $\bfS$ on $Z$ such that $$(\lambda \tilde{\g_1})\cdot z\cdot \tilde{\g_2}=\tilde{\g_1}\cdot(\lambda z)\cdot\tilde{\g_2}=\tilde{\g_1}\cdot z\cdot(\lambda \tilde{\g_2}),$$ for any $(\lambda,\tilde{\g_1},z,\tilde{\g_2}) \in \bfS \times \wGa_1 \times Z \times \wGa_2$ such that these products make sense. We refer to $Z:(\wGa_1,\Ga_1,\del_1)\To (\wGa_2,\Ga_2,\del_2)$ as a generalized morphism of Real graded $\bfS$-twists. In particular, if $Z$ is an isomorphism, the two Real graded $\bfS$-twists are said to be \emph{Morita equivalent}; in this case we write $(\wGa_1,\Ga_1,\del_1)\sim (\wGa_2,\Ga_2,\del_2)$.
\end{df}

\begin{lem}
Let $Z:(\wGa_1,\Ga_1,\delta_1)\To (\wGa_2,\Ga_2,\del_2)$ be a generalized morphism. Then the $\bfS$-action on $Z$ is free and the Real space $Z/\bfS$ (with the obvious involution) is a Real generalized homomorphism from $\Ga_1$ to $\Ga_2$. 
\end{lem}

\begin{proof}
Same as~\cite[Lemma 2.10]{TLX}.	
\end{proof}

\begin{df}
Let $\cG$ be a Real groupoid and $\bfS$ an abelian Real group. A \emph{Real graded $\bfS$-central extension} of $\cG$ consists of a triple $(\wGa, \Ga, \delta, P)$, where $(\wGa,\Ga,\delta)$ is a Real graded $\bfS$-twist, and $P$ is a (Real) Morita equivalence $\Ga\To\cG$.
\end{df}

\begin{df}~\label{df:RExt_equivalence}
We say that $(\wGa_1,\Ga_1,\delta_1,P_1)$ and $(\wGa_2,\Ga_2,\delta_2,P_2)$ are Morita equivalent if there exists a Morita equivalence $Z: (\wGa_1,\Ga_1,\del_1) \To (\wGa_2,\Ga_2,\del_2)$ such that the diagrams

\begin{equation}
\xymatrix{\Ga_1 \ar[r]^{Z/\bfS} \ar[dr]_{P_1} &  \Ga_2 \ar[d]^{P_2} \\ & \cG}
\end{equation}
and 
\begin{equation}
\xymatrix{\Ga_1 \ar[r]^{Z/\bfS} \ar[dr]_{\delta_1} & \Ga_2 \ar[d]^{\delta_2} \\ & \ZZ_2}
\end{equation}

commute in the category $\RG$. Such a $Z$ is also called \emph{an equivalence bimodule} of Real graded $\bfS$-central extensions. The set of Morita equivalence classes of Real graded $\bfS$-central extensions of $\cG$ is denoted by $\wRExt(\cG,\bfS)$.
\end{df}

The set $\wRExt(\cG,\bfS)$ admits a natural structure of abelian group described in the following way. Assume that $\EE_i=(\wGa_i,\Ga_i,\delta_i,P_i), \ i=1,2$, are two given Real graded $\bfS$-central extensions of $\cG$, then $\xymatrix{Y_1 & Z \ar[l]_\fr \ar[r]^\fs & Y_2}$ is a Morita equivalence between $\Ga_1$ and $\Ga_2$, where $Z=P_1\times_\cG P_2$. But from Proposition ~\ref{pro-pullback} there exists a Real homeomorphism $f: \fs^\ast \Ga_2 \To \fr^\ast \Ga_1$. Now one can see that the maps $\pi: \fr^\ast \wGa_1 \To \fr^\ast \Ga_1, \ (z,\tilde{\g_1},z')\mto (z,\pi_1(\tilde{\g_1}),z')$ and $\pi': \fs^\ast \wGa_2 \To \fr^\ast \Ga_1 (z,\tilde{\g_2},z')\mto \pi \circ f (z,\tilde{\g_2},z')$ define two Real $\bfS$-principal bundles and then $(\fr^\ast \wGa_1,\delta)$ and $(\fs^\ast \wGa_2,\delta)$, where $\delta:= \delta_1 \circ pr_2$, define elements of $\Rext(\fr^\ast \Ga_1,\bfS)$. Therefore, we can form the tensor product $(\fr^\ast \wGa_1\hat{\otimes} \fs^\ast \wGa_2, \delta \otimes\delta)$ are Real graded $\bfS$-groupoid over $\fr^\ast \Ga_1$. Moreover, $\fr^\ast \Ga_1\sim_{Morita}\Ga_1$; then, if $P: \fr^\ast \Ga_1 \To \cG$ is a Real Morita equivalence, we obtain a Real graded $\bfS$-central extension of $\cG$ by setting
\begin{eqnarray}
\EE_1\hat{\otimes}\EE_2:=(\fr^\ast \wGa_1\hat{\otimes} \fs^\ast \wGa_2, \fr^\ast \Ga_1, \delta, P),
\end{eqnarray}
that we will call \emph{the tensor product of $\EE_1$ and $\EE_2$}. Thus, we define the sum 
\[[\EE_1]+[\EE_2]:=[\EE_1\hat{\otimes}\EE_2],\] which is easily seen to be well defined in $\wRExt(\cG,\bfS)$. The inverse $\EE^{\op}$ of $\EE$ is $(\wGa^{\op},\Ga,\del,P)$. Notice that $\Rext(\cG,\bfS)$ is naturally a subgroup of $\wRExt(\cG,\bfS)$ by identifying a Real graded $\bfS$-twist $(\wGa,\cG,\del)$ with the Real graded $\bfS$-central extension $(\wGa,\cG,\del,\cG)$. We summarize this in the next lemma.

\begin{lem}
Under the sum defined above, $\wRExt(\cG,\bfS)$ is an abelian group whose zero element is the class of the trivial Real graded $\bfS$-central extension $(\cG \times \bfS,\cG,0, \cG)$.
\end{lem}

When the Real structure is trivial, then we recover the usual definition of graded central extensions (see~\cite{FHT} for instance) of $\cG$ by the group $\ZZ_2$.

\begin{pro}~\label{pro:Ext-triv-inv}
Suppose that $\grpd$ is equipped with a trivial Real structure. Then $$\wRExt(\cG,\uc)\cong \wExt(\cG,\ZZ_2).$$	
\end{pro}

\begin{ex}
Suppose $\cG$ reduces to a Real space $X$. Then following Example~\ref{ex:X-Morita}, a Real graded $\bfS$-central extension of $X$ is a triple $(\wGa,Y^{[2]},\del)$, where $Y$ is a Real space together with a continuous locally split Real open map $\pi:Y\To X$, and $\del:Y^{[2]}\To \ZZ_2$ is a Real morphism. 

In particular, suppose $\rho$ is trivial. Then, by Proposition~\ref{pro:Ext-triv-inv}, giving a Real graded $\uc$-central extension of $X$ amounts to giving a \emph{real bundle gerbe} 
\[
\xymatrix{\ZZ_2 \ar[r] & \wGa \ar[d] &   \\ & Y^{[2]} \dar[r] & Y \ar[d]^\pi  \\ & &  X
}
\]
in the sense of Mathai, Murray, and Stevenson~\cite{MMS}, together with an augmentation $\del:Y^{[2]}\To \ZZ_2$.	
\end{ex}
 
%

\subsection{Functoriality of $\wRExt(\cdot,\bfS)$}

The aim of this subsection is to show that $\wRExt(\cdot,\bfS)$ is functorial in the category $\RG$, and hence that the group $\wRExt(\cG,\bfS)$ invariant under Morita equivalence. To do this, we will need the following

\begin{pro}~\label{pro:RExt_vs_RTw}
Let $\grpd$ be a Real groupoid. Then, there is an isomorphism of abelian groups
\begin{eqnarray}
\wRExt(\cG,\bfS)\cong \underset{\cU}{\varinjlim}\Rext(\cG[\cU],\bfS).
\end{eqnarray}
\end{pro}

Before giving the proof of this proposition, we have to describe the sum in the inductive limit $\underset{\cU}{\varinjlim}\Rext(\cG[\cU],\bfS)$. Let $\cU_1$ and $\cU_2$ be two Real open covers of $X$, and let $\cT_i=(\tilde{\cG}_i,\cG[\cU_i],\delta_i)$ be Real graded $\bfS$-groupoids over $\cG[\cU_i], i=1,2$. Let $(\cV,f_\cV)\in \Omega\left(\cG[\cU_1],\cG[\cU_2]\right)$ be the unique class corresponding to the Real Morita equivalence $Z_{\iota_{\cU_1}}^{-1}\circ Z_{\iota_{\cU_2}}$ from $\cG[\cU_1]$ to $\cG[\cU_2]$. $\cV$ is a Real open cover of $X$ containing $\cU_1$, and $f_\cV: \cG[\cV]\To \cG[\cU_2]$ is a Real strict morphism. Denote by $\iota_{\cV,\cU_1}$ the canonical Real morphism $\cG[\cV] \To \cG[\cU_1]$. Then, the tensor product of $\cT_1$ and $\cT_2$ is
\begin{equation}
\cT_1\hat{\otimes}\cT_2:=\iota_{\cV,\cU_1}^\ast \cT_1 \hat{\otimes}f_{\cV}^\ast \cT_2,
\end{equation} 
which defines a Real graded $\bfS$-groupoids over the Real groupoid $\cG[\cV]$.

\begin{proof}[Proof of Proposition~\ref{pro:RExt_vs_RTw}]
For a Real graded $\bfS$-central extension $\EE=(\wGa, \Ga,\delta,P)$ of $\cG$ , let $(\cV,f_\cV)\in \Omega\left(\cG,\Ga\right)$ be the isomorphism in $\RG_\Omega$ corresponding to the Morita equivalence $P^{-1}:\cG\To \Ga$. Setting
\begin{eqnarray}
\cT_\EE:=\xymatrix{\bfS \ar[r] & f_\cV^\ast \wGa \ar[r]^{f_\cV^\ast \pi} & \cG[\cV] \ar[d]^{\delta \circ f_\cV} \\ & & \ZZ_2} 
\end{eqnarray}
we get a Real graded $\bfS$-groupoid over $\cG[\cV]$. It is not hard to check that this provides us the desired isomorphism of abelian groups; the inverse is given by the formula
\begin{eqnarray}
\EE_\cT:=(\tilde{\cG},\cG[\cU],\delta, Z_{\iota_\cU}),
\end{eqnarray}
for a Real graded $\bfS$-twist $\cT =(\tilde{\cG}, \cG[\cU],\delta)$.
\end{proof}

From this proposition, it is now possible to define the \emph{pull-back} of a Real graded $\bfS$-central extension via a Real generalized morphism. More precisely, we have

\begin{dfpro}~\label{dfpro:pull-back-extension}
Let $\cG$ and $\cG'$ be Real groupoids, and let  $Z:\cG'\To \cG$ be a Real generalized morphism. Let $\EE=(\wGa,\Ga,\delta,P)$ is be a representative in $\wRExt(\cG,\bfS)$, and $\cT_\EE=(f_\cV^\ast \wGa,\cG[\cV],\delta \circ f_\cV)$ its image in $\underset{\cU}{\varinjlim} \Rext(\cG[\cU],\bfS)$ (see the proof of Proposition~\ref{pro:RExt_vs_RTw}). Let $(\cW,f_{\cW})\in \Omega\left(\cG',\cG[\cV]\right)$ be the morphism in $\RG_\Omega$ corresponding to the Real generalized morphism $Z_{\iota_\cV}^{-1}\circ Z:\cG'\To \cG[\cV]$.
Then  
\begin{eqnarray}
Z^\ast \EE:= \EE_{f_{\cW}^\ast \cT_\EE}.
\end{eqnarray}
is a Real graded $\bfS$-central extension of the Real groupoid $\cG'$; it is called the \emph{pull-back of $\EE$ along $Z$}	
\end{dfpro}

Now the following is straightforward.

\begin{cor}~\label{cor:functor_RExt}
There is a contravariant functor
\begin{eqnarray}
\wRExt(\cdot,\bfS): \RG \To \mathfrak{Ab},
\end{eqnarray}
which sends a Real groupoid $\cG$ to the abelian group $\wRExt(\cG,\bfS)$. In particular, $\wRExt(\cG,\bfS)$ is invariant under Morita equivalences.
\end{cor}


\section{Real \v{C}ech cohomology}

\subsection{\emph{Real} simplicial spaces}

 We start by recalling some preliminary notions. For each zero integer $n\in \NN$, we set $[n]=\{0,...,n\}$. Recall~\cite{Tu1} that the simplicial (resp. pre-simplicial) category $\Delta$ (resp. $\Delta'$) is the category whose objects are the sets $[n]$, and whose morphisms are the nondecreasing (resp. increasing) maps $f:[m]\To [n]$. For $n\in \NN$,  we denote by $\Delta^{(N)}$ the $N$-truncated full subcategory of $\Delta$ whose objects are those $[k]$ with $k\leq N$.

\begin{df}
A \emph{Real simplicial (resp. pre-simplicial, $N$-simplicial) topological space} consists of a contravariant functor from $\Delta$ (resp. $\Delta'$, $\Delta^{(N)}$) to the category $\RTop$ whose objects are topological Real spaces and morphisms are continuous Real maps. A morphism of Real simplicial (resp. pre-simplicial,...) spaces is a morphism of such functors.
\end{df}

More concretely, a Real (pre-)simplicial space is given by a family $(X_\bullet, \rho_\bullet)=(X_n,\rho_n)_{n\in \NN}$ of topological Real spaces, and for every map $f:[m]\To [n]$ we are given a continuous Real map (called \emph{face} or \emph{degeneracy map} depending which of $m$ and $n$ is larger) $\tilde{f}:(X_n,\rho_n)\To (X_m,\rho_m)$  , satisfying the relation $\widetilde{f\circ g}=\tilde{g}\circ \tilde{f}$ whenever $f$ and $g$ are composable.

\begin{df}
Let $(X_\bullet,\rho_\bullet)$ be a Real simplicial space. For any $N\in \NN$, the \emph{$N$-skeleton} of $(X_\bullet,\rho_\bullet)$ is the Real simplicial space $(X_\bullet,\rho_\bullet)^N$ "of dimension $N$" ; that is, $(X_n,\rho_n)^N=(X_n,\rho_n)$ for $n\leq N$, and $(X_n,\rho_n)^N=(X_N,\rho_N)$ for all $n\ge N+1$.
\end{df}

Let $\ve_i^n:[n-1]\To [n]$ be the unique increasing injective map that avoids $i$, and let $\eta_i^n:[n+1]\To [n]$ be the unique nondecreasing surjective map such that $i$ is reached twice; that is, 
\begin{eqnarray}~\label{eq:epsilon}
\ve_i^n(k)=\left\lbrace \begin{array}{ll} k, & \ \text{if} \ k \leq i-1 , \\ k+1, & \ \text{if} \  k\geq i, \end{array} \right. \ \text{and} \ 
 \eta_i^n(k)=\left\lbrace \begin{array}{ll} k, & \ \text{if} \ k\leq i; \\ k-1, & \ \text{if} \ k\geq i+1.  \end{array} \right.
\end{eqnarray}
We will omit the superscript $n$ if there is no ambiguity.

If $(X_\bullet,\rho_\bullet)$ is a Real simplicial space, the face and degeneracy maps $$\tilde{\ve}_i^n: (X_n,\rho_n)\To (X_{n-1},\rho_{n-1}), \quad  {\rm and} \ \tilde{\eta}_i^n:(X_n,\rho_n)\To (X_{n+1},\rho_{n+1}), \ i=0,...,n$$ clearly satisfy the following \emph{simplicial identities}:
\begin{align}~\label{eq:face_degen}
\begin{array}{l}
\tilde{\ve}_i^{n-1}\tilde{\ve}_j^n=\tilde{\ve}_{j-1}^{n-1}\tilde{\ve}_i^n \ \text{if} \ i\leq j-1, \ \tilde{\eta}_i^{n+1}\tilde{\eta}_j^n=\tilde{\eta}_{j+1}^{n+1}\tilde{\eta}_i^n \ \text{if} \ i\leq j, \ \tilde{\ve}_i^{n+1}\tilde{\eta}_j^n=\tilde{\eta}_{j-1}^{n-1}\tilde{\ve}_i^n \ \text{if} \ i\leq j-1,  \\  \tilde{\ve}_i^{n+1}\tilde{\eta}_j^n=\tilde{\eta}_j^{n-1}\tilde{\ve}_{i-1}^n \ \text{if} \ i\geq j+2, \ \text{and} \ \tilde{\ve}_j^{n+1}\tilde{\eta}_j^n=\tilde{\ve}_{j+1}^{n+1}\tilde{\eta}_j^n=\Id_{X_n}.
\end{array}
\end{align}

Conversely, let $(X_n,\rho_n)_{n\in \NN}$ be a sequence of topological Real spaces together with maps satisfying (~\ref{eq:face_degen}). Then thanks to~\cite[Theorem 5.2]{Mac}, there is a unique Real simplicial structure on $(X_\bullet,\rho_\bullet)$ such that $\tilde{\ve}_i$ and $\tilde{\eta}_i$ are the face and degeneracy maps respectively.

\begin{ex}(compare~\cite[\S2.3]{TX}).
Consider the pair groupoid $\xymatrix{[n]\times[n] \dar[r] & [n]}$; that is, the product is $(i,j)(j,k)):= (i,k)$ and the inverse of $(i,j)$ is $(j,i)$.

If $(\cG,\rho)$ is a topological Real groupoid, we define
\[\cG_n:= \Hom([n]\times [n],\cG) \] as the space of strict morphisms from the groupoid $\xymatrix{[n]\times [n] \dar[r] & [n]}$ to $\grpd$. We obtain a Real structure on $\cG_n$ by defining $\rho_n(\vp):=\rho \circ \vp$, for $\vp\in \cG_n$.
Any $f\in \Hom_{\Delta}([m],[n])$ (or $f\in \Hom_{\Delta'}([m],[n])$) naturally gives rise to a strict morphism $f\times f: [m]\times [m] \To [n]\times [n]$, which, in turn, induces a Real map $\tilde{f}: (\cG_n,\rho_n)\To (\cG_m,\rho_m)$ given by $\tilde{f}(\vp):=\vp \circ (f\times f)$ for $\vp\in \cG_n$. Hence, we obtain a Real simplicial space $(\cG_\bullet,\rho_\bullet)$.

Notice that the groupoid $\xymatrix{[n]\times [n] \dar[r] & [n]}$ is generated by elements $(i-1,i), 1\leq i \leq n$; indeed, given an element $(i,j)\in [n]\times [n]$, we can suppose that $i\leq j$ (otherwise, we take its inverse $(j,i)$), and then $(i,j)=(i,i+1)...(j-1,j)$. It turns out that any strict morphism $\vp: [n]\times [n]\To \cG$ is uniquely determined by its images $\vp(i-1,i)\in \cG$; hence, the well defined Real map 
\[\cG_n\To \cG^{(n)}, \ \vp \mto (g_1,...,g_n), \] where $g_i:=\vp(i-1,i), \ 1\leq i\leq n$, and $\cG^{(n)}:=\{(h_1,...,h_n) \ | \ s(h_i)=r(h_{i-1}), \ i=1,...,n\}$, identifies $(\cG_n,\rho_n)$ with $(\cG^{(n)},\rho^{(n)})$, where $\rho^{(n)}$ is the obvious Real structure on the fibred product $\cG^{(n)}$.
Therefore, using this identification, the face maps $\tilde{\ve}_i^n: (\cG_n,\rho_n)\To (\cG_{n-1},\rho_{n-1})$  of $(\cG_\bullet,\rho_\bullet)$ are given by:
\begin{eqnarray}
\begin{array}{l}
\tilde{\ve}_0^n(g_1,g_2,...,g_n) = (g_2,...,g_n),\\
\tilde{\ve}_i^n(g_1,g_2,...,g_n) = (g_1,...,g_ig_{i+1},...,g_n), \ 1\leq i \leq n-1,\\
\tilde{\ve}_n^n(g_1,g_2,...,g_n) = (g_1,...,g_{n-1}),
\end{array}
\end{eqnarray}
and for $n=1$, by $\tilde{\ve}_0^1(g)=s(g)$, $\tilde{\ve}^1_1(g)=r(g)$; while the degeneracy maps $\tilde{\eta}_i^n:(\cG_n,\rho_n)\To (\cG_{n+1},\rho_{n+1})$ are given by:
\begin{eqnarray}
\begin{array}{l}
\tilde{\eta}_0^n(g_1,g_2,...,g_n) = (r(g_1),g_1,...,g_n),\\
\tilde{\eta}_i^n(g_1,g_2,...,g_n) = (g_1,...,s(g_i),g_{i+1},...,g_n), \ 1\leq i\leq n,
\end{array}
\end{eqnarray}
and $\tilde{\eta}_0^0:\cG_0\To \cG_1$ is the unit map of the Real groupoid.
\end{ex}

Now for $n\in \NN$, we define the space $(E\cG)_n$ of $(n+1)$-tuples of elements of $\cG$ that map to the same unit; \emph{i.e}. $(E\cG)_n:=\{(\g_0,...,\g_n)\in \cG^{n+1} \ | \ r(\g_0)=r(\g_1)=...=r(\g_n)\}$. Suppose we are given $(g_1,...,g_n)\in \cG_n$. Then we can choose an $(n+1)$-tuple $(\g_0,...,\g_n)\in (E\cG)_n$ such that $g_i=\g_{i-1}^{-1}\g_i$ for each $i=1,...,n$. If $(\g_0',...\g_n')$ is another $(n+1)$-tuples verifying these identities, then $s(\g'_i)=s((\g'_{i-1})^{-1}\g'_i)=s(\g_{i-1}^{-1}\g_i)=s(\g_i)$, for all $i=1,...,n$, and that means that there exists a unique $g\in \cG$, such that $s(g)=r(\g_i)$ and $\g_i'=g\cdot\g_i$. This hence gives us a well defined injective map \[\cG_n \To (E\cG)_n/_\sim, \ (g_1,...,g_n)\mto [\g_0,...,\g_n],\] where $(\g_0,...,\g_n)\sim (g\cdot\g_0,...,g\cdot\g_n)$. Moreover, this map is surjective, for if $(\g_0,...,\g_n)\in (E\cG)_n$, one can consider morphisms $g_i$ from $s(\g_i)$ to $s(\g_{i-1}), \ i=1,...,n$, so that we have $$\g_1=\g_0g_1, \ \g_2=\g_1g_2=\g_0g_1g_2,...,\g_n=\g_0g_1\cdots g_n,$$ and then $$[\g_0,...,\g_n]=[r(g_1),g_1,g_1g_2,...,g_1\cdots g_n]$$ which gives the inverse $(E\cG)_n/_\sim \ni [\g_0,...,\g_n]\mto (g_1,...,g_n)\in \cG_n$. It hence turns out that we can identify $\cG_n$ with the quotient $(E\cG)_n$. Note that the quotient space $(E\cG)_n/_\sim$ naturally inherits the Real structure $\rho_{n+1}$ and that the isomorphism defined above is compatible with the Real structures.

Henceforth, an element of $\cG_n$ will be represented by a vector $\overrightarrow{g}=(g_1,...,g_n)$, where we view $\overrightarrow{g}$ as a morphism $[n]\times [n]\To \cG$, and $g_i=\overrightarrow{g}(i-1,i), \ i=1,...,n$, or $\overrightarrow{g}=[\g_0,...,\g_n]$ as a class in $(E\cG)_n/_\sim$. For the first picture,  if $f\in \text{Hom}_\Delta([m],[n])$, then the Real face/degeneracy map $\tilde{f}: (\cG_n,\rho_n)\To (\cG_m,\rho_m)$ is given by:
\begin{eqnarray}~\label{eq:1st_face_degen}
\tilde{f}(\overrightarrow{g})=\left(\overrightarrow{g}\left(f(0),f(1)\right),...,\overrightarrow{g}\left(f(m-1),f(m)\right)\right). 
\end{eqnarray}
For instance, if $f$ in injective, then $$\overrightarrow{g}\left(f(i-1),f(i)\right)=\overrightarrow{g}\left(f(i-1),f(i-1)+1\right)\cdots\overrightarrow{g}\left(f(i)-1,f(i)\right) \ \text{for} \ f(i)\geq 1,$$ and thus 
\begin{eqnarray}
\tilde{f}(\overrightarrow{g})=(g_{_{f(0)+1}}\cdots g_{_{f(1)}},...,g_{_{f(m-1)+1}}\cdots g_{_{f(m)}}).
\end{eqnarray}
However, the second picture offers a more general formula for the face and degeneracy maps; roughly speaking, for any $f\in \text{Hom}_\Delta([m],[n])$, we have $\overrightarrow{g}(i,j)=\g_i^{-1}\g_j$ for every $(i,j)\in [n]\times [n]$. In particular, $\overrightarrow{g}(f(k-1),f(k))=\g_{f(k-1)}^{-1}\g_{f(k)}$, for every $k\in [m]$; then (~\ref{eq:1st_face_degen}) gives :
\begin{eqnarray}~\label{eq:face_degenracy_map}
\tilde{f}(\overrightarrow{g})=[\g_{f(0)},...,\g_{f(m)}].
\end{eqnarray}


\subsection{\emph{Real} sheaves on \emph{Real} simplicial spaces}

In this subsection we closely follow~\cite[\S3]{Tu1} to study Real sheaves on Real (pre-)simplicial spaces. We start by introducing some preliminary notions.

Let $\cC$ be a topological category. We define the category $\cC_R$ by setting:

\begin{itemize}
\item $Ob(\cC_R)$ consists of triples $(A,\sigma_A,A')$, where $ A,\ A' \in Ob(\cC)$ and $\sigma_A\in \Hom_\cC(A,A')$;
\item $\Hom_{\cC_R}\left((A,\sigma_A,A'),(B,\sigma_B,B')\right)$ consists of pairs $(f,\tilde{f})$ of morphisms $f:A\To B, \ \tilde{f}:A'\To B'$ in $\cC$ such that the diagrams \[\xymatrix{A \ar[d]_{\sigma_A} \ar[r]^f & B\ar[d]^{\sigma_B} \\ A' \ar[r]^{\tilde{f}} & B'}\] commute.
\end{itemize}

Now, let $\phi: \cC \To \cC$ be a functor. Then we define the subcategory $\cC_\phi$ of $\cC_R$ whose objects are pairs $(A,\phi(A))$, where $A\in Ob(\cC)$, and in which a morphism from $(A,\phi(A))$ to $(B,\phi(B))$ is a pair $(f,\tilde{f})$ of morphisms $f: A\To B, \ \tilde{f}:\phi(A)\To \phi(B)$ such that $\tilde{f}\circ \phi=\phi\circ f$. A fundamental example of this is the category $\mathfrak{OB}(X)$ of open subsets of a given topological Real space $(X,\rho)$. Recall that objects of this category are the collection of the open sets $U\subset X$, and morphisms are the canonical injections $V\hookrightarrow U$ when $V\subset U$. Given such a Real space $(X,\rho)$, the map $\rho$ induces a functor (which is an isomorphism) $\rho: \mathfrak{OB}(X)\To \mathfrak{OB}(X)$ given by 
$$\left(\xymatrix{V \ar@{^{(}->}[r]^\iota & U}\right)\mto \left(\xymatrix{\rho(V)\ar@{^{(}->}[r]^{\rho\circ\iota\circ \rho} & \rho(U)}\right).$$

\begin{df}[Real presheaves]
Let $(X,\rho)$ be a topological Real space, and let $\cC$ be a topological category. A \emph{Real presheaf} $(\fF,\sigma)$ on $(X,\rho)$ with values in $\cC$ is a contravariant functor from $\mathfrak{OB}(X)_\rho$ to $\cC_R$; a morphism of Real presheaves is a morphism of such functors.
\end{df}

Specifically, from the fact that $\rho: X\To X$ is a homeomorphism and from the canonical properties of the injections $V\hookrightarrow U$ of open sets $V\subset U \subset X$, a Real presheaf on $(X,\rho)$ with values in $\cC$ assigns to each open subset $U\subset X$ a triple $(\fF(U),\sigma_{_U},\fF(\rho(U)))$, where $\fF(U), \ \fF(\rho(U))$ are objects of $\cC$, and $\sigma_{_U}\in \Isom_{\cC}(\fF(U),\fF(\rho(U)))$, and for $V\subset U$ we are given two morphisms  $\vp_{_{V,U}}:\fF(U)\To \fF(V)$ and $\vp_{_{\rho(V),\rho(U)}}:\fF(\rho(U))\To \fF(\rho(V))$, called the restriction morphisms, such that:
\begin{itemize}
\item $\vp_{_{U,U}}=\Id_{_{\fF(U)}}$;
\item $\sigma_{_V} \circ \vp_{_{V,U}}=\vp_{_{\rho(V),\rho(U)}}\circ \sigma_{_U}$, 
\item $\vp_{_{W,U}}=\vp_{_{W,V}}\circ \vp_{_{V,U}}$, and $\vp_{_{\rho(W),\rho(U)}}=\vp_{_{\rho(W),\rho(V)}}\circ \vp_{_{\rho(V),\rho(U)}}$.
\end{itemize}

A morphism of Real presheaves $\phi: (\fF,\sigma^\fF)\To (\fG,\sigma^\fG)$ is then a family of $\phi_{_U}\in \Hom_\cC(\fF(U),\fG(U))$ such that, for all pairs of open sets $U,\ V$ with $V\subset U$, the diagrams below commute:
\begin{eqnarray}
\xymatrix{\fF(\rho(U)) \ar[d]^{\phi_{_{\rho(U)}}} & \fF(U) \ar[l]_{\sigma_{_U}^\fF} \ar[d]^{\phi_{_U}} \ar[r]^{\vp_{_{V,U}}^\fF} & \fF(V) \ar[d]^{\phi_{_V}} \\
         \fG(\rho(U)) & \fG(U)\ar[l]_{\sigma_{_U}^\fG} \ar[r]^{\vp_{_{V,U}}^\fG} & \fG(V) }
\end{eqnarray}

As in the standard case, if $(\fF,\sigma)$ is a Real presheaf over $X$, and if $U$ is an open subset of $X$, an element $\mathsf{s}\in \fF(U)$ is called a \emph{section of $(\fF,\sigma)$ on $U$}, and for $x\in X$. If $V$ is an open subset of $U$, and $\mathsf{s}\in \fF(U)$, one often writes $\mathsf{s}_{|_V}$ for $\vp_{_{V,U}}(\mathsf{s})$.

\begin{df}(~\cite[Definition 2.2]{KS}).
A \emph{Real sheaf} \ over $(X,\rho)$ with values in $\cC$ is a Real presheaf $(\fF,\sigma)$ satisfying the following conditions:
\begin{itemize}
\item[(i)] For any open set $U\subset X$, any open cover $U=\bigcup_{i\in I}U_i$, any section $\mathsf{s}\in \fF(U)$, $\mathsf{s}_{|_{U_i}}=0$ for all $i$ implies $\mathsf{s}=0$.
\item[(ii)] For any open set $U\subset X$, any open cover $U=\bigcup_{i\in I}U_i$, any family of sections $\mathsf{s}_i\in \fF(U_i)$ satisfying $\mathsf{s}_i {}_{| U_{ij}}=\mathsf{s}_j {}_{|U_{ij}}$ for all nonempty intersection $U_{ij}$, there exists $\mathsf{s}\in \fF(U)$ such that $\mathsf{s}_{|U_i}=\mathsf{s}_i$ for all $i$.
\end{itemize}
A morphism of Real sheaves is a morphism of the underlying presheaves. We denote by $\cC_R(X)$ (or simply by $\mathsf{Sh}_\rho(X)$ if there is no risk of confusion) for the category of Real sheaves on $(X,\rho)$ with values in $\cC$.
\end{df}

Notice that if $(\fF,\sigma)$ is a Real sheaf (resp. presheaf) on $(X,\rho)$, then $\fF$ is a sheaf (resp. presheaf) on $X$ in the usual sense. Recall that the \emph{stalk} of $\fF$ at a point $x\in X$, denoted by $\fF_x$, is the direct limit of the direct system $(\fF(U),\vp_{_{V,U}})$ where $U$ runs along the family of open neighborhoods of $x$; \emph{i.e}. 
\begin{eqnarray}
\fF_x:= \underset{x\in U}{\varinjlim}\fF(U),
\end{eqnarray}
The image of a section $\mathsf{s}\in \fF(U)$ in $\fF_x$ by the canonical morphism $\fF(U)\To \fF_x$ (where $x\in U)$ is called the \emph{germ} of $\mathsf{s}$ at $x$ and denoted by $\mathsf{s}_x$. 

Note that if $U$ is an open neighborhood of $x$, $\rho(U)$ is an open neighborhood of $\rho(x)$, and the isomorphism $\sigma_U: \fF(U) \ni \mathsf{s}\mto \sigma_U(\mathsf{s}) \in \fF(\rho(U))$ extends to an isomorphism $\sigma_x: \fF_x \To \fF_{\rho(x)}$, defined by $\sigma_x(\mathsf{s}_x)=(\sigma_{_U}(\mathsf{s}))_{\rho(x)}$, whose inverse is $\sigma_{\rho(x)}$. We thus have a well defined $2$-periodic isomorphism, also denoted by $\sigma$, on the topological~\footnote{Recall that if $\fF$ is a presheaf over $X$, any section $\mathsf{s}\in \fF(U)$ induces a map $[\mathsf{s}]:U\To \coprod_x \fF_x, \ y\mto \mathsf{s}_y$. We give $\cF:=\coprod_{x\in X}\fF_x$ the largest topology such that all the maps $[\mathsf{s}]$ are continuous. On the other hand, associated to $\fF$, there is a sheaf $\what{\fF}$ given by $\what{\fF}(U):=\Ga(U,\cF)$, and we have that $\fF(U)\cong \Ga(U,\cF)$ if and only if $\fF$ is a sheaf. Then, given a Real presheaf $(\fF,\sigma)$, one can define its associated Real sheaf in the same fashion.} space $\cF:= \coprod_{x\in X}\fF_x$, given by 
\begin{eqnarray}
\sigma: \cF\To \cF, \ (x,\mathsf{s}_x)\mto (\rho(x),\sigma_x(\mathsf{s}_x))
\end{eqnarray}
 which gives a Real space $(\cF,\sigma)$.

\begin{ex}
Let $(X,\rho)$ be a Real space. Then the space $\mathit{C}(X)$ of continuous complex values functions on $X$ defines a Real sheaf of abelian groups on $(X,\rho)$ by $(U,\rho(U))\mto (\mathit{C}(U),\tilde{\rho}_{_U},\mathit{C}(\rho(U)))$, where $\tilde{\rho}_{_U}(f)(\rho(x)):=\overline{f(x)}$.
\end{ex}

\begin{df}[Pushforward, pullback]~\label{df:image_sheaf}
Let $(X,\rho)$, $(Y,\vr)$ be topological Real spaces, $f:(Y,\vr)\To (X,\rho)$ a continuous Real map. Suppose that $(\fF,\sigma)$  and $(\fG,\varsigma)$ are Real sheaves on $(X,\rho)$ and $(Y,\varrho)$ respectively, with values in the same category $\cC$.
\begin{itemize}
\item[(i)] The \emph{pushforward} of $(\fG,\varsigma)$ by $f$, denoted by $(f_\ast \fG,f_\ast \varsigma)$, is the Real sheaf on $(X,\vr)$ defined by the contravariant functor: 
\begin{eqnarray}
\mathfrak{OB}(X)_\rho\To \cC_R, \ (U,\rho(U))\mto \left(f_\ast \fG(U),f_\ast \varsigma_{_U},f_\ast \fG(\rho(U))\right),
\end{eqnarray}
where $f_\ast \fG(U):=\fG(f^{-1}(U)), \ f_\ast \varsigma_{_U}:=\varsigma_{_{f^{-1}(U)}}$, and $$f_\ast \fG(\rho(U))=\fG(f^{-1}(\rho(U)))\cong \fG(\vr(f^{-1}(U))).$$
\item[(ii)] The \emph{pullback} of $(\fF,\sigma)$ along $f$, denoted by $(f^\ast \fF,f^\ast \sigma)$, is the Real sheaf on $(Y,\vr)$ associated to the Real presheaf defined by:
\begin{eqnarray}
\mathfrak{OB}(Y)_\vr \To \cC_R, \ (V,\vr(V))\mto (f^\ast \fF(V),f^\ast \sigma_{_V},f^\ast \fF(\vr(V))),
\end{eqnarray}
where $f^\ast \fF(V):=\underset{\underset{U \ \text{open}}{f(V)\subset U\subset X}}{\varinjlim}\fF(U)$, and $f^\ast \sigma_{_V}: f^\ast \fF(V)\To f^\ast \fF(\vr(V))$ is the morphism in $\cC$ extending functorially $\sigma_{_U}: \fF(U)\To \fF(\rho(U))$ along the family of open neighborhoods of $f(V)$ in $X$.
\end{itemize}
\end{df}

It immediately follows from this definition that we have a covariant functor 

\begin{eqnarray}
\RTop \To \mathfrak{RSh}, \ \left(\xymatrix{(Y,\vr)\ar[r]^f & (X,\rho)}\right) \mto \left(\xymatrix{\mathsf{Sh}_\vr(Y) \ar[r]^{f_\ast} & \mathsf{Sh}_\rho(X)}\right),
\end{eqnarray}

and a contravariant functor 

\begin{eqnarray}
\RTop \To \mathfrak{RSh}, \ \left(\xymatrix{(Y,\vr)\ar[r]^f & (X,\rho)}\right)\mto \left(\xymatrix{\mathsf{Sh}_\rho(X)\ar[r]^{f^\ast} & \mathsf{Sh}_\vr(Y)}\right),
\end{eqnarray}
where $\mathfrak{RSh}$ is the category whose objects are the categories of Real sheaves on given Real spaces and morphisms are functors of such categories.

We will also need the following proposition.

\begin{pro}
Let $f:(Y,\vr)\To (X,\rho)$ be a a continuous Real map. Suppose that $(\fF,\sigma)$  and $(\fG,\varsigma)$ are Real sheaves on $(X,\rho)$ and on $(Y,\varrho)$ respectively, with values in the same category $\cC$. Then 
\begin{eqnarray}
\Hom_{\mathsf{Sh}_\rho(X)}\left((\fF,\sigma),(f_\ast \fG,f_\ast \varsigma)\right)\cong \Hom_{\mathsf{Sh}_\vr(Y)}((f^\ast \fF,f^\ast \sigma),(\fG,\varsigma)).
\end{eqnarray}
\end{pro}

\begin{proof}
The proof is the same as in the general case where Real structures are not concerned (see for instance ~\cite[Proposition 2.3.3]{KS}).
\end{proof}

\begin{df}
Given a continuous Real map $f:(Y,\vr)\To (X,\rho)$ and Real sheaves $(\fF,\sigma)$ and $(\fG,\varsigma)$ as above, we define the set $\Hom_f(\fF,\fG)_{\sigma,\varsigma}$ of \emph{Real $f$-morphisms} from $(\fF,\sigma)$ to $(\fG,\varsigma)$ to be 
$$\Hom_{\mathsf{Sh}_\rho(X)}\left((\fF,\sigma),(f_\ast \fG,f_\ast \varsigma)\right)= \Hom_{\mathsf{Sh}_\vr(Y)}\left((f^\ast \fF,f^\ast \sigma),(\fG,\varsigma)\right).$$
\end{df}

\begin{df}
Let $(X_\bullet,\rho_\bullet)$ be a Real simplicial (resp. pre-simplicial) space. A Real sheaf on $(X_\bullet,\rho_\bullet)$ is a family $(\fF^n,\sigma^n)_{n\in \NN}$ such that $(\fF^n,\sigma^n)$ is a Real sheaf on $(X_n,\rho_n)$ for all $n$, and such that for each morphism $f: [m]\To [n]$ in $\Delta$ (resp. $\Delta'$) we are given Real $\tilde{f}$-morphisms $\tilde{f}^\ast \in \text{Hom}_{\tilde{f}}(\fF^m,\fF^n)_{\sigma^m,\sigma^n}$ such that
\begin{eqnarray}
\widetilde{f\circ g}^\ast=\tilde{f}^\ast \circ \tilde{g}^\ast,
\end{eqnarray}
whenever $f$ and $g$ are composable.
\end{df}

One can use the definition of the push-forward to give a concrete interpretation of this definition. Roughly speaking, a sequence $(\fF^n,\sigma^n)_{n\in \NN}$ is a Real sheaf on a Real simplicial (resp. pre-simplicial, ...) space $(X_\bullet,\rho_\bullet)$, if for a given morphism $f:[m]\To [n]$ in $\Delta$ (resp. $\Delta'$, ...), then for any pair of open sets $U\subset X_n$ and $V\subset X_m$ such that $\tilde{f}(U)\subset V$ there is a \emph{restriction map} $\tilde{f}^\ast: \fF^m(V)\To \fF^n(U)$ such that the diagram
\begin{eqnarray}
\xymatrix{\fF^m(V)\ar[d]_{\sigma_{_V}^m} \ar[r]^{\tilde{f}^\ast} & \fF^n(U) \ar[d]^{\sigma^n_{_U}} \\ \fF^m(\rho(V)) \ar[r]^{\tilde{f}^\ast} & \fF^n(\rho(U))}
\end{eqnarray}
commute, and $\tilde{f}^\ast \circ \tilde{g}^\ast = \widetilde{f\circ g}^\ast: \fF^k(W)\To \fF^n(U)$ whenever $\tilde{g}(V)\subset W \subset X_k$.\
Morphisms of Real sheaves over $(X_\bullet,\rho_\bullet)$ are defined in the obvious way; we denote by $\mathsf{Sh}_{\rho_\bullet}(X_\bullet)$ for the category of Real sheaves over $(X_\bullet,\rho_\bullet)$.


\subsection{\emph{Real} $\cG$-sheaves and reduced \emph{Real} sheaves}

\begin{df}
\begin{itemize}
\item[(i)] A Real space $(Y,\vr)$ is said to be \emph{\'{e}tale} over $(X,\rho)$ if there exists an \emph{\'{e}tale} Real map $f:(Y,\vr)\To (X,\rho)$; that is to say, every point $y\in Y$ has an open neighborhood $V$ such that $f_{_V}:V\To U$ is homeomorphism, where $U$ in an open neighborhood of $f(y)$ in $X$.
\item[(ii)] A Real groupoid $(\cG,\rho)$ is \'{e}tale if the range (equivalently the source) map is \'{e}tale.
\item[(iii)] A morphism $\pi_\bullet: (Y_\bullet, \vr_\bullet)\To (X_\bullet,\rho_\bullet)$ of Real (pre-)simplicial spaces is \'{e}tale if for all $n$, $\pi_n: (Y_n,\vr_n)\To (X_n,\rho_n)$ is \'{e}tale.  
\end{itemize}
\end{df}

\begin{ex}~\label{ex:etale}
Any Real sheaf $(\fF,\sigma)$ on $(X,\rho)$ can be viewed as an \'{e}tale Real space over $(X,\rho)$. Indeed, considering the underlying topological Real space $(\cF,\sigma)$, it is easy to check that the canonical projection $$\cF \To X, \ (x, \mathsf{s}_x)\mto x$$ is an \'{e}tale Real map.
\end{ex}

\begin{df}
Let $(\cG,\rho)$ be a topological Real groupoid. A Real $\cG$-sheaf (or an \'{e}tale Real $\cG$-space) is an \'{e}tale Real space $(\cE_0,\nu_0)$ over $(\Gpdo,\rho)$ equipped with a continuous Real $\cG$-action. \

We say that $(\cE_0,\nu_0)$ is an Abelian Real $\cG$-sheaf if in addition it is an Abelian Real sheaf on $(\Gpdo,\rho)$ such that the action $\alpha_g:(\cE_0)_{s(g)}\To (\cE_0)_{r(g)}$ is a group homomorphism, for any $g\in \cG$.

A morphism of Real $\cG$-sheaves $(\cE_0,\nu_0)$ and $(\cE_0',\nu_0')$ is a $\cG$-equivariant continuous Real map $\psi: (\cE_0,\nu_0)\To (\cE_0',\nu_0')$ such that $p'\circ \psi=p$.\

 The category of Real $\cG$-sheaves is denoted by $\fB_\rho\cG$, and is called the \emph{classifying topos} of $(\cG,\rho)$.
\end{df}

\begin{exs}

\begin{enumerate}
	\item Considering a Real space $(X,\rho)$ as a Real groupoid, a Real $X$-sheaf is the same thing as a Real sheaf over $(X,\rho)$; in other words we have that $\mathfrak{B}_\rho X \cong \mathsf{Sh}_\rho(X)$.\\
	\item If $(\cG,\rho)$ is a Real group, then a Real $\cG$-sheaf is just a Real space equipped with a continuous Real $\cG$-action.
\end{enumerate}
\end{exs}

\begin{lem}~\label{lem:morphism_toposes}
Any generalized Real morphism $(Z,\tau):(\Ga,\vr)\To (\cG,\rho)$ induces a morphism of toposes \[Z^\ast: \mathfrak{B}_\rho(\cG)\To \mathfrak{B}_\vr(\Ga).\] 
Consequently, there is a contravariant functor \[\mathfrak{B}: \RG \To \mathfrak{RBG},\] defined by 
$$(\xymatrix{(\Ga,\vr)\ar[r]^{(Z,\tau)} & (\cG,\rho)})\mto (\xymatrix{\mathfrak{B}_\rho \cG \ar[r]^{Z^\ast} & \mathfrak{B}_\vr \Ga}),$$
where $\mathfrak{RBG}$ is the category whose objects are classifying toposes of Real groupoids.
\end{lem}

\begin{proof}
As noted in~\cite[2.2]{Moer1} for the usual case, any Real morphism $f:(\Ga,\vr)\To (\cG,\rho)$ gives rise to a functor $f^\ast: \fB_\rho \cG \To \fB_\vr \Ga$. Indeed, if $(\cE_0,\nu_0)$ is a Real $\cG$-sheaf through an \'{e}tale Real $\cG$-map $p:(\cE_0,\nu_0)\To (\Gpdo,\rho)$, then we obtain a Real $\Ga$-sheaf $(f^\ast \cE_0,f^\ast \nu_0)$ by pulling back $(\cE_0,\nu_0)$ along $f$; \emph{i.e}. $f^\ast \cE_0=\Gamo \times_{f,\Gpdo,p}\cE_0, \ f^\ast \nu_0=\vr\times \nu_0, \ f^\ast p(y,e):=y$, and the right Real $\Ga$-action is $\g\cdot(s(\g),e):=(r(\g),f(\g)\cdot e)$ when $p(e)=s(f(\g))$. If $\psi:(\cE_0,\nu_0)\To (\cE_0',\nu_0')$ is a morphism of Real $\cG$-sheaves, then the map $f^\ast \psi: (f^\ast \cE_0,f^\ast \nu_0)\To (f^\ast \cE_0',f^\ast\nu_0')$ defined by $f^\ast \psi (y,e):=(y,\psi(e))$ is obviously a morphism a Real $\Ga$-sheaves. It follows that any $(\cU,f_\cU)\in \Hom_{\RG_\Omega}((\Ga,\vr),(\cG,\rho))$ gives rise to a covariant functor $f_\cU^\ast: \fB_\rho \cG\To \fB_\vr\Ga[\cU]$. Now if $(Z,\tau)$ corresponds to $(\cU,f_{\cU})$, and if as in the previous chapter, $\iota: \Ga[\cU]\To \Ga$ is the canonical Real morphism, then we can push forward $(f_{\cU}^\ast \cE_0,f_\cU^\ast \nu_0)$ through $\iota$ to get a Real $\Ga$-sheaf $(Z^\ast \cE_0,Z^\ast \nu_0)$; i.e 
\begin{equation}
Z^\ast \cE_0:=\iota_\ast f_{\cU}^\ast \cE_0,
\end{equation}
and the Real structure $Z^\ast \nu_0$ is the obvious one.
\end{proof}

\begin{lem}~\label{lem:G_sheaf}
Let $(\cG,\rho)$ be a topological Real groupoid. Then, any Real $\cG$-sheaf canonically defines a Real sheaf over the Real simplicial space $(\cG_n,\rho_n)_{n\in \NN}$.
\end{lem}

To prove this Lemma, we need some more preliminary notions.

\begin{df}(~\cite{Tu1}).
A morphism $\pi_\bullet: (\cE_\bullet,\nu_\bullet)\To (X_\bullet,\rho_\bullet)$ of Real simplicial spaces is called \emph{reduced} if for all $m, \ n$ and for all $f\in \Hom_\Delta([m],[n])$, the morphism $\tilde{f}$ induces an isomorphism \[(\cE_n,\nu_n)\cong (X_n\times_{\tilde{f},X_m,\pi_m}\cE_m,\rho_n\times \nu_m).\]
In this case, we say that $(\cE_\bullet,\nu_\bullet)$ is a reduced Real simplicial space over $(X_\bullet,\rho_\bullet)$.	
\end{df}

Morphisms of reduced Real simplicial spaces over $(X_\bullet,\rho_\bullet)$ are defined in the obvious way.

\begin{df}(~\cite{Tu1}).
We say that a Real sheaf $(\fF^\bullet,\sigma^\bullet)$ over a Real simplicial space $(X_\bullet,\rho_\bullet)$ is \emph{reduced} if for all $m,\ n$ and all $f\in \Hom_\Delta([m],[n])$, $\tilde{f}^\ast \in\Hom\left((\tilde{f}^\ast \fF^m,\tilde{f}^\ast \sigma^m),(\fF^n,\sigma^n)\right)$ is an isomorphism.
\end{df}

\begin{lem}~\label{lem:red-Sheaves_etalSimpli}( ~\cite[Lemma 3.5]{Tu1}).
Let $(X_\bullet,\rho_\bullet)$ be a Real simplicial space. Then,  there is a one-to-one correspondence between reduced Real sheaves over $(X_\bullet,\rho_\bullet)$ and reduced \'{e}tale Real simplicial spaces over $(X_\bullet,\rho_\bullet)$.
\end{lem}

\begin{proof}
Suppose that we are given a Real sheaf $(\fF^\bullet,\sigma^\bullet)$ over the Real simplicial space $(X_\bullet,\rho_\bullet)$, and let $(\cF_n,\sigma_n)_{n\in \NN}$ be its underlying sequence of topological Real spaces. We already know from Example ~\ref{ex:etale} that each of the canonical projection maps $\pi_n: (\cF_n,\sigma_n)\To (X_n,\rho_n)$ is \'{e}tale. Now suppose that $(\fF^\bullet,\sigma^\bullet)$ is reduced; that is to say that for any morphism $f\in \text{Hom}_\Delta([m],[n])$, and every open set $V\subset X_m$, $\tilde{f}^\ast: \fF^m(V)\To \fF^n(\tilde{f}^{-1}(V))$ is an isomorphism, so that we have a commutative diagram 
\begin{eqnarray}
\xymatrix{ \fF^m(V) \ar[d]_{\sigma^m_{_V}} \eq[r]^{\tilde{f}^\ast} & \fF^n(\tilde{f}^{-1}(V)) \ar[d]^{\sigma^n_{_{f^{-1}(V)}}} \\ \fF^m(\rho^m(V)) \eq[r]^{\tilde{f}^\ast} & \fF^n(\rho^n(\tilde{f}^{-1}(V)))}
\end{eqnarray}
Let $x\in X_n, \ y\in X_m$ such that $\tilde{f}(x)=y$, and let $U\subset X_n$ and $V\subset X_m$ be open neighborhoods of $x$ and $y$ respectively such that $\tilde{f}(U)\subset V$. Then, for a section $\mathsf{s}^m\in \fF^m(V)$, we have an element $(x,(y,\mathsf{s}^m_y))\in X_n\times_{\tilde{f},X_m,\pi_m}\cF_m$ to which we assign an element $(x,\mathsf{s}^n_x)\in \cF_n$ as follows: since $U\subset \tilde{f}^{-1}(V)$, the section $\mathsf{s}^m\in \fF^m(V)\cong \fF^n(\tilde{f}^{-1}(V))$ has a restriction $\mathsf{s}^n:=\mathsf{s}_U^m\in \fF^n(U)$. In this way we get a well defined map $X_n\times_{\tilde{f},X_m,\pi_m}\cF_m\To \cF_n$. Moreover, it is easy to check that this map is an isomorphism; the inverse is the map $$\cF_n\ni (x,\mathsf{s}^n_x)\mto (x,(\tilde{f}(x),(\tilde{f}^\ast\mathsf{s}^n)_{\tilde{f}(x)}))\in X_n\times_{_{\tilde{f},X_m,\pi_m}}\cF_m,$$ where if $x\in U\subset X_n$ and $\tilde{f}(U)\subset V\subset X_m$, $\tilde{f}^\ast \mathsf{s}^n$ is any section in $\fF^m(V)\cong \fF^n(\tilde{f}^{-1}(V))$ that has the same class as $\mathsf{s}^n$ at the point $x$ when restricted to $\fF^n(U)$ through the restriction map $\fF^n(\tilde{f}^{-1}(V))\To \fF^n(U)$. Furthermore, for every $f\in \text{Hom}_\Delta([m],[n])$, there is a face/degeneracy map $\tilde{f}: (\cF_n,\sigma_n)\To (\cF_m,\sigma_m)$ given by $\tilde{f}(x,\mathsf{s}_x):=(\tilde{f}(x),(\tilde{f}^\ast \mathsf{s})_{_{\tilde{f}(x)}})$; hence $(\cF_\bullet,\sigma_\bullet)$ is a reduced \'{e}tale Real simplicial space over $(X_\bullet,\rho_\bullet)$.\

Conversely, if $\pi_\bullet:(\cE_\bullet,\nu_\bullet)\To (X_\bullet,\rho_\bullet)$ is a reduced \'{e}tale morphism of Real simplicial spaces, we let $\fF^n(U)$ be the space $C(U,\cE_n)$ of continuous sections over $U$ (where $U$ is an open subset of $X_n$) of the projection $\pi_n:(\cE_n,\nu_n)\To (X_n,\rho_n)$.  Next we define $\sigma^n_{_U}: \fF^n(U)\To \fF^n(\rho^n(U))$ by $\sigma^n_{_U}(\mathsf{s})(\rho^n(x)):=\nu_n(\mathsf{s}(x))$. Notice that since the $\pi_n$'s are \'{e}tale, one can recover the Real spaces $(\cE_n,\nu_n)$ by considering the underlying Real spaces of the Real sheaves $(\fF^n,\sigma^n)$. Now for any $f\in \text{Hom}_\Delta([m],[n])$ and  for any open set $V\subset X_m$, we have an isomorphism $\tilde{f}^\ast: \fF^m(V)\To \fF^n(\tilde{f}^{-1}(V)), \ \mathsf{s}\mto \tilde{f}^\ast{\mathsf{s}}$, where $(\tilde{f}^\ast\mathsf{s})(x)=(x,\mathsf{s}(\tilde{f}(x)))\in X_n\times_{_{\tilde{f},X_m,\pi_m}}\cE_m\cong \cE_m$.
\end{proof}

Using the same construction as in the second part of this proof, we deduce the following

\begin{lem}~\label{lem:simpli_sheaf}
Any reduced Real simplicial space over $(X_\bullet,\rho_\bullet)$, \'{e}tale or not, determines a Real sheaf over $(X_\bullet,\rho_\bullet)$.
\end{lem}

\begin{proof}[Proof of Lemma~\ref{lem:G_sheaf}]
Let $(Z,\tau)$ be a Real $\cG$-sheaf, and let $\pi:(Z,\tau)\To (\Gpdo,\rho)$ be an \'{e}tale Real map. Put for all $n\geq 0$, $\cE_n:=(\cG\ltimes Z)_n:=\cG_n\times_{\tilde{\pi}_n,\Gpdo,\pi}Z$, where $\tilde{\pi}_n(g_1,...,g_n)=\tilde{\pi}_n[\g_0,...,\g_n]=s(\g_n)=s(g_n)$, and define $\nu_n:=\rho_n \times \tau$. We thus obtain a Real simplicial space $(\cE_n,\nu_n)$: the simplicial structure is given by 
\begin{eqnarray}
\cE_n \ni \left([\g_0,...,\g_n],z\right) \mto \left((\g_{f(0)},...,\g_{f(m)}),\g_{f(m)}^{-1}\g_n\cdot z\right)\in \cE_m,
\end{eqnarray}
for $f\in \Hom_\Delta([m],[n])$. Furthermore, it is straightforward  to see that the projections $\pi_n: \cE_n\To \cG_n$ are compatible with the Real structures $\nu_n$ and $\rho_n$, and that they define a morphism of Real simplicial spaces. If $f\in \Hom_\Delta ([m],[n])$, then the assignment $$([\g_0,...,\g_n],z)\mto \left([\g_0,...,\g_n],([\g_{f(0)},...,\g_{f(m)}],\g_{f(m)}^{-1}\g_n\cdot z)\right)$$ obviously defines a Real homeomorphism $\cE_n \cong \cG_n\times_{_{\tilde{f},\cG_m,\pi_m}}\cE_m$ which shows that $(\cE_\bullet,\nu_\bullet)$ is a reduced Real simplicial space over $(\cG_n,\rho_n)$. It follows from Lemma~\ref{lem:simpli_sheaf} that $(\cE_\bullet,\nu_\bullet)$ determines an object of $\mathsf{Sh}_{\rho_\bullet}(\cG_\bullet)$.
\end{proof}

\begin{rem}~\label{rem:G-space_G-sheaf}
Notice that in the proof above we did not use the fact that $(Z,\tau)$ is \'{e}tale. In fact, the Real $\cG$-action suffices for $(Z,\tau)$ to give rise to a Real sheaf over $(\cG_\bullet,\rho_\bullet)$. However, the property of being \'{e}tale will be necessary to show that the Real sheaf obtained is reduced (as it is mentioned in the following corollary).
\end{rem}

\begin{cor}~\label{cor:BG-->Sh}
Let $(\cG,\rho)$ be a topological Real groupoid. Then there is a functor
\[ \cE: \mathfrak{B}_\rho \cG \To \mathfrak{red}\mathsf{Sh}_{\rho_\bullet}(\cG_\bullet),\] where $\mathfrak{red}\mathsf{Sh}_{\rho_\bullet}(\cG_\bullet)$ is the full subcategory of $\mathsf{Sh}_{\rho_\bullet}(\cG_\bullet)$ consisting of all reduced Real sheaves over $(\cG_\bullet,\rho_\bullet)$.
\end{cor}

\begin{proof}
Let us keep the same notations as in the proof of Lemma~\ref{lem:G_sheaf}. Since $\pi$ is \'{e}tale, so is $\pi_n$ for all $n$. The reduced Real simplicial space $(\cE_\bullet,\nu_\bullet)$ is then \'{e}tale over $(\cG_\bullet,\nu_\bullet)$. Now, it suffices to apply Lemma~\ref{lem:red-Sheaves_etalSimpli}.
\end{proof}


\subsection{\emph{Real} $\cG$-modules}

\begin{df}(Compare with ~\cite[Definition 3.9]{Tu1}).
Let $(\cG,\rho)$ be a topological Real groupoid. A Real $\cG$-module is a topological Real groupoid $(\cM,\ {}^-)$, with unit space $(\Gpdo,\rho)$, and with source and range maps equal to a Real map $\pi: (\cM,\ {}^-)\To (\Gpdo,\rho)$, such that
\begin{itemize}
	\item $\cM_x$ ($=\cM^x=\cM_x^x$) is an abelian group for all $x\in \Gpdo$;
	\item for all $x\in \Gpdo$, the map $({}^-):\cM_x\To \cM_{\rho(x)}$ is a group morphism;
	\item as a Real space, $(\cM,\ {}^-)$ is endowed with a Real $\cG$-action $\alpha: \cG\times_{s,\pi}\cM\To \cM$;
	\item for each $g\in \cG$, the map $\alpha_g: \cM_{s(g)}\To \cM_{r(g)}$ given by the action is a group morphism.
\end{itemize}
\end{df}

By Remark~\ref{rem:G-space_G-sheaf}, any Real $\cG$-module $(\cM,\ {}^-)$ determines an abelian Real sheaf $(\fF^\bullet,\sigma^\bullet)$ on $(\cG_\bullet,\rho_\bullet)$ constructed as follows: consider the reduced Real simplicial space $(\cE_\bullet,\nu_\bullet)=((\cG\ltimes\cM)_n,\rho_n \times ( {}^-))$, where the Real simplicial structure is given by: 
\[\tilde{f}\left([\g_0,...,\g_n],t\right)=\left([\g_{f(0)},...,\g_{f(m)}],\g_{f(m)}^{-1}\g_n\cdot t\right), \]
for any $f\in \Hom_\Delta([m],[n])$. Next, $(\fF^\bullet,\sigma^\bullet)$ is defined as the sheaf of germs of continuous sections of the projections $\pi_\bullet:(\cE_\bullet,\nu_\bullet)\To (\cG_\bullet,\rho_\bullet)$.

\begin{ex}~\label{ex:G-module_S}
Let $(\cG,\rho)$ be a topological Real groupoid and let $\cM=\Gpdo \times \uc$ be endowed with the canonical Real structure $\overline{(x,\lambda)}):=(\rho(x),\bar{\lambda})$, and Real $\cG$-action $g\cdot(s(g),\lambda)=(r(g),\lambda)$. Then $(\cM,{}^-)$ is a Real $\cG$-module. The corresponding Real sheaf is called the constant sheaf of germs of $\uc$-valued functions and denoted (abusively) $\uc$. More generally, if $S$ is any Real group, $\Gpdo\times S$ is a Real $\cG$-module, and the induced Real sheaf over $(\cG_\bullet,\rho_\bullet)$ is denoted by $S$.
\end{ex}


\subsection{Pre-simplicial \emph{Real} covers}

\begin{df}[Compare with Definition 4.1~\cite{Tu1}].
Let $(X_\bullet,\rho_\bullet)$ be a Real pre-simplicial space. A \emph{Real open cover} of $(X_\bullet,\rho_\bullet)$ is a sequence $\cU_\bullet=(\cU_n)_{n\in \NN}$ such that $\cU_n=(U_j^n)_{j\in J_n}$ is a Real open cover of $(X_n,\rho_n)$.

We say that $\cU_\bullet$ is \emph{pre-simplicial} if $(J_\bullet,\ {}^-)=(J_n, \ {}^-)_{n\in \NN}$ is a Real pre-simplicial set such that for all $f\in \Hom_{\Delta'}([m],[n])$ and for all $j\in J_n$, one has $\tilde{f}(U_j^n)\subseteq U^m_{\tilde{f}(j)}$. In the same way, one defines the notions of simplicial Real cover and $N$-simplicial Real cover.
\end{df}

We will use the same construction as in~\cite[\S4.1]{Tu1} to show the following lemma.

\begin{lem}~\label{lem:pre-smpl_cover}
Any Real open cover $\cU_\bullet$ of a Real (pre-)simplicial space $(X_\bullet,\rho_\bullet)$ gives rise to a pre-simplicial Real open cover ${}_\natural \cU_\bullet$.
\end{lem}

\begin{proof}
For each $n\in \NN$, let $\cP_n=\bigcup_{k=0}^n \cP_n^k$, where $\cP_n^k=\Hom_{\Delta'}([k],[n])$. Let $\cP=\bigcup_n \cP_n$, and let $\Lambda_n$ (or $\Lam_n(J_\bullet)$ if there is a risk of confusion) be the set of maps 
\begin{eqnarray}
\lambda: \cP \To \bigcup_k J_k \ \text{such that} \ \lambda(\cP_n^k)\in J_k, \ \text{for all} \ k.
\end{eqnarray}
It is immediate to see that $\Lam_n$ is non-empty; indeed, for each $k\in \NN$, we fix a map $\overrightarrow{j}^k: [n]\To J_k$ which can be written as $\overrightarrow{j}^k=(j^k_0,...,j^k_n)$. Next, we define $\overrightarrow{j}=(\overrightarrow{j}^k)_{k\in \NN}$. Then the map $\lambda: \cP\To \bigcup_k J_k$ given by $\lambda(\vp):=\overrightarrow{j}\circ \vp$ lies in $\Lam_n$. Moreover, $\Lam_n$ has a Real structure defines as follows: if $\vp \in \cP_n^k$, then we set
\begin{eqnarray}
\bar{\lambda}(\vp):= \overline{\lambda(\vp)}\in J_k
\end{eqnarray}
Now, for all $\lambda\in \Lam_n$, we let
\begin{eqnarray}
U_\lambda^n:= \bigcap_{k\leq n}\bigcap_{\vp\in \cP_n^k}\tilde{\vp}^{-1}(U_{\lambda(\vp)}^k).
\end{eqnarray}
Let $x\in X_n$. For each $k\leq n$ and $\vp\in \cP_n^k$, there is $j_\vp^k\in J_k$ such that $\tilde{\vp}(x)\in U^k_{j_\vp^k}\subset X_k$. Define the map $\lambda_x: \cP\To \bigcup_k J_k$ by $\lambda_x(\vp):=(j_\vp^k)_{k}$. Then, one can see that  $x\in \bigcap_{k\leq n}\bigcap_{\vp\in \cP_n^k}\tilde{\vp}^{-1}(U^k_{\lambda_x(\vp)})=U^n_{\lambda_x}$. Furthermore, $\rho_n(U^n_\lambda)=U^n_{\bar{\lambda}}$; hence, $(U^n_\lambda)_{\lambda \in \Lam_n}$ is a Real open cover of $(X_n,\rho_n)$.  If for any $f\in \Hom_{\Delta'}([m],[n])$, we define a map $\tilde{f}:\Lam_n\To \Lam_m$ by \[(\tilde{f}\lambda)(\vp):= \lambda(f\circ \vp), \ \text{for all} \ \lambda \in \Lam_n, \ \text{and} \ \vp \in \cP_n^k,\]
one sees that $\tilde{f}(U^n_\lambda)\subseteq U^m_{\tilde{f}(\lambda)}$. Thus, ${}_\natural \cU_\bullet=((U^n_\lambda)_{\lambda \in \Lam_n})_{n\in \NN}$ is a pre-simplicial Real open cover of $(X_\bullet,\rho_\bullet)$. 
\end{proof}

In the same way, for $N\in \NN$ and $n\leq N$, we denote by $\Lam_n^N$ the set of all maps $$\lambda: \bigcup_{k\leq n}\Hom_\Delta([k],[n])\To \bigcup_{k\leq n}J_k$$ that satisfy $\lambda(\Hom_\Delta([k],[n]))\subset J_k$, and we set
\[U^n_\lambda:=\bigcap_{k\leq n}\bigcap_{\vp\in \Hom_\Delta([k],[n])}\tilde{\vp}^{-1}(U^n_{\lambda(\vp)}).\]
Then we equip $\Lambda^N_\bullet$ with the Real structure defined in the same fashion, and we give it the $N$-simplicial structure defined as follows: for any $f\in \text{Hom}_{\Delta^N}([m],[n])$, the map $\tilde{f}:\Lam_n^N\To \Lam_m^N$ is given by $(\tilde{f}\lambda)(\vp):=\lambda(f\circ \vp)$. We thus obtain a $N$-simplicial Real cover ${}_{\natural^N}\cU_\bullet=({}_{\natural^N}\cU_n)_{n\in \NN}$ of the $N$-skeleton of $(X_\bullet,\rho_\bullet)$, where ${}_{\natural^N}\cU_n=(U^n_\lambda)_{\lambda \in \Lam^N_n}$.

We endow the collection of Real open covers of $(X_\bullet,\rho_\bullet)$ with the partial pre-order given by the following definition.

\begin{df}~\label{df:refinement}
Let $\cU_\bullet$ and $\cV_\bullet$ be Real open covers of a Real simplicial space $(X_\bullet,\rho_\bullet)$, with $\cU_n=(U^n_j)_{j\in J_n}$ and $\cV_n=(V^n_i)_{i\in I_n}$. We say that $\cV_\bullet$ is \emph{finer} than $\cU_\bullet$ if for each $n\in \NN$, there exists a Real map $\theta_n:(I_n, \ {}^-)\To (J_n, \ {}^-)$ such that $V^n_i\subseteq U^n_{\theta_n(i)}$ for every $i\in I_n$. The Real map $\theta_\bullet=(\theta_n)_{n\in \NN}$ is required to be pre-simplicial (resp. $N$-simplicial) if $\cU_\bullet$ and $\cV_\bullet$ are pre-simplicial (resp. $N$-simplicial).
\end{df}


\subsection{"Real" \v{C}ech cohomology}

\begin{df}[Real local sections]
Let $(\fF,\sigma)$ be an abelian Real (pre-)sheaf over $(X,\rho)$ and let $\cU=(U_j)_{j\in J}$ be a Real open cover of $(X,\rho)$.
We say that a family $\mathsf{s}_j \in \fF(U_j)$ is a \emph{globally Real family} of local sections of $(\fF,\sigma)$ over $\cU$ if for every $j\in J$, $\mathsf{s}_{\bar{j}}$ is the image of $\mathsf{s}_j$ in $\fF(U_{\bar{j}})$ by $\sigma_{U_j}$.

We define $CR_{ss}(\cU,\fF)_{\rho,\sigma}$ to be the  set of all globally Real families of local sections of $(\fF,\sigma)$ relative to $\cU$; \emph{i.e}. 
\[CR_{ss}(\cU,\fF)_{\rho,\sigma}:=\left\{(\mathsf{s}_j)_{j\in J}\subset \prod_{j\in J}\fF(U_j) \ | \ \mathsf{s}_{\bar{j}}=\sigma_{U_j}(\mathsf{s}_j), \ \forall j\in J \right\}.\]
\end{df}

 To avoid irksome notations, we will write $CR_{ss}(\cU,\fF)$ or $CR_{ss}(\cU,\fF)_\sigma$ instead of $CR_{ss}(\cU,\fF)_{\rho,\sigma}$. It is clear that $CR_{ss}(\cU,\fF)$ is an abelian group.\\

Now let $(X_\bullet,\rho_\bullet)$ be a Real simplicial space, and let $\cU_\bullet$ be a pre-simplicial Real open cover of $(X_\bullet,\rho_\bullet)$. Suppose $(\fF^\bullet,\sigma^\bullet)$ is a (pre-simplicial) abelian Real (pre-)sheaf over $(X_\bullet,\rho_\bullet)$.

\begin{df}
We define the complex $CR_{ss}^\ast(\cU_\bullet,\fF^\bullet)_{\rho_\bullet,\sigma^\bullet}$, also denoted by $CR_{ss}^\ast(\cU_\bullet,\fF^\bullet)$ if there is no risk of confusion, by 
\begin{eqnarray}
CR_{ss}^n(\cU_\bullet,\fF^\bullet):= CR_{ss}(\cU_n,\fF^n)_{\rho_n,\sigma^n}, \ \text{for} \ n\in \NN.
\end{eqnarray}
A \emph{Real $n$-cochain} of $(X_\bullet,\rho_\bullet)$ relative to a pre-ssimplicial Real open cover $\cU_\bullet$ with coefficients in $(\fF^\bullet,\sigma^\bullet)$ is an element in $CR_{ss}^n(\cU_\bullet,\fF^\bullet)$.
\end{df}

Let us consider again the maps $\ve_k:[n]\To [n+1]$ defined by~\eqref{eq:epsilon}, for $k=0,...,n+1$. We have Real maps $\tilde{\ve}_k: (J_{n+1},\ {}^-)\To (J_n,\ {}^-)$, $\tilde{\ve}_k: (X_{n+1},\rho_{n+1})\To (X_n,\rho_n)$, and $\tilde{\ve}_k:(\fF^{n+1},\sigma^{n+1})\To (\fF^n,\sigma^n)$; and since $\tilde{\ve}_k(U_j^{n+1})\subseteq U^n_{\tilde{\ve}_k(j)}$ for every $j\in J_{n+1}$, we have a restriction map $$\tilde{\ve}^\ast_k: \fF^n(U^n_{\tilde{\ve}(j)})\To \fF^{n+1}(U_j^{n+1})$$ such that $\sigma^{n+1}_{U_j^{n+1}}\circ \tilde{\ve}^\ast_k=\tilde{\ve}^\ast_k \circ \sigma^n_{U^n_{\tilde{\ve}_k(j)}}$.

\begin{df}
Let $\cU_\bullet$ be a pre-simplicial Real open cover of $(X_\bullet,\rho_\bullet)$. For $n\ge 0$, we define the \emph{differential map} 
\begin{eqnarray}
d^n: CR^n_{ss}(\cU_\bullet,\fF^\bullet)\To CR_{ss}^{n+1}(\cU_\bullet,\fF^\bullet)
\end{eqnarray}
also denoted by $d$, by setting for $c=(c_j)_{j\in J_n}\in CR^n_{ss}(\cU_\bullet,\fF^\bullet)$ and for $j\in J_{n+1}$: 

\begin{equation}~\label{df:differential}
(dc)_j:= \sum_{k=0}^{n+1}(-1)^k\tilde{\ve}_k^\ast (c_{\tilde{\ve}_k(j)}). 
\end{equation}
\end{df}

\begin{rem}
The differential $d$ of~\eqref{df:differential} do maps $CR^n_{ss}(\cU_\bullet,\fF^\bullet)$ to $CR^{n+1}_{ss}(\cU_\bullet,\fF^\bullet)$; indeed, combining the fact that the $\tilde{\ve}_k$ are Real maps and the discussion preceeding the last definition, one has $$(dc)_{\bar{j}}=\sum_{k=0}^{n+1}(-1)^k\tilde{\ve}^\ast_k(c_{\tilde{\ve}_k(\bar{j})})=\sum_{k=0}^{n+1}(-1)^k\tilde{\ve}^\ast_k(\sigma^{n}_{U^n_{\tilde{\ve}_k(j)}}c_{\tilde{\ve}_k(j)})=\sigma_{U^{n+1}_j}((dc)_j).$$
\end{rem}

\begin{lem}
The differential maps $d$ are group homomorphisms that satisfy  $d^n\circ d^{n-1}=0$ for $n\ge 1$.
\end{lem}

\begin{proof}
That for any $n\in \NN$, $d^n$ is a group homomorphism is straightforward. Let $(c_{j'})_{j'\in J_{n-1}}\in CR_{ss}^{n-1}(\cU_\bullet,\fF^\bullet)$. Then, for $j\in J_{n+1}$ one has

\[
(d^nd^{n-1}c)_j  = \sum_{l=0}^{n+1}(-1)^l(\tilde{\ve}^{n+1}_l)^\ast \left(\sum_{k=0}^n(-1)^k (\tilde{\ve}^n_k)^\ast (c_{\tilde{\ve}_k^n\circ \tilde{\ve}^{n+1}_l(j)}) \right) \]

 \[ =  \sum_{l=0}^{n+1}\sum_{k=0}^n (-1)^{l+k}(\tilde{\ve}^{n+1}_l)^\ast \circ (\tilde{\ve}^n_k)^\ast (c_{\tilde{\ve}^n_k\circ \tilde{\ve}^{n+1}_l(j)})   \]
  \[ =  \sum_{p=0}^n \ \sum_{k=0, k\leq 2p}^n (\tilde{\ve}^{n+1}_{2p-k})^\ast (\tilde{\ve}^n_k)^\ast (c_{\tilde{\ve}^n_k \circ \tilde{\ve}^{n+1}_{2p-k}(j)}) \]
   
    \[ - \sum_{p=0}^n \ \sum_{k=0, k\leq 2p+1}^n(\tilde{\ve}^{n+1}_{2p+1-k})^\ast \circ (\tilde{\ve}^n_k)^\ast (c_{\tilde{\ve}^n_k\circ \tilde{\ve}^{n+1}_{2p+1-k}(j)}) \]
     
 \[ =  0, \text{since} \ \ve^{n+1}_{r}\circ \ve^n_q=\ve^{n+1}_{r+1}\circ \ve^n_q, \ \text{for any} \ r, q \leq n.
\]

\end{proof}

We thus can give the following

\begin{df}
A Real $n$-cochain $c$ in the kernel of $d^n$ is called a \emph{Real $n$-cocycle} relative to the pre-simplicial Real open cover $\cU_\bullet$ with coefficients in $(\fF^\bullet,\sigma^\bullet)$; the Real $n$-cocyles form a subgroup $ZR^n_{ss}(\cU_\bullet,\fF^\bullet)$ of $CR_{ss}^n(\cU_\bullet,\fF^\bullet)$. The Real $n$-cochains belonging to the image of $d^{n-1}$ are called \emph{Real $n$-coboundaries} relative to $\cU_\bullet$ and form a subgroup $BR^n_{ss}(\cU_\bullet,\fF^\bullet)$ (since $d^2=0$). 
The $n^{th}$ \emph{Real cohomology group} of the pre-simplicial Real open cover $\cU_\bullet$ with coefficients in $(\fF^\bullet,\sigma^\bullet)$ is  defined by the $n^{th}$ cohomology group of the complex \[... \overset{d^{n-2}}{\To} CR^{n-1}_{ss}(\cU_\bullet,\fF^\bullet) \overset{d^{n-1}}{\To} CR_{ss}^{n}(\cU_\bullet,\fF^\bullet) \overset{d^n}{\To} CR^{n+1}_{ss}(\cU_\bullet,\fF^\bullet) \overset{d^{n+1}}{\To} ... \] 
That is, $$HR^n_{ss}(\cU_\bullet,\fF^\bullet):=\frac{ZR^n_{ss}(\cU_\bullet,\fF^\bullet)}{BR_{ss}^n(\cU_\bullet,\fF^\bullet)}:=\frac{\ker d^n }{\text{\emph{Im}} \  d^{n-1}}.$$
\end{df}

\begin{ex}(Compare with~\cite[Example 4.3]{Tu1}).
Let $(X_\bullet,\rho_\bullet)$ be the constant Real simplicial space associated with a topological Real space $(X,\rho)$; that is $(X_n,\rho_n)=(X,\rho)$ for every $n\ge 0$. Suppose $\cU=:\cU_0=(U^0_j)_{j\in J_0}$ is a Real open cover of $(X,\rho)$. Define $J_n:=J_0^{n+1}$ together with the obvious Real structure. Then $(J_n,\ {}^-)$ is admits a simplicial structure by \[\tilde{f}(j_0,...,j_n):=(j_{f(0)},...,j_{f(m)}), \ \text{for all} \ f\in \Hom_\Delta([m],[n]).\]
Let $U^n_{(j_0,..,j_n)}:=U^0_{j_0}\cap ... \cap U^0_{j_n}$ and $\cU_n=(U^n_j)_{j\in J_n}$. Of course $\cU_n$ is a Real open cover of $(X_n,\rho_n)$, and for any $f\in \Hom_\Delta([m],[n])$ one has $\tilde{f}(U^n_{(j_0,...,j_n)})=U^n_{(j_0,...,j_n)}\subseteq U^0_{f(0)}\cap ... \cap U^0_{f(m)}=U^m_{\tilde{f}(j_0,...,j_n)}$; hence $\cU_\bullet$ is a simplicial Real open cover of $(X_\bullet,\rho_\bullet)$. \\
Let $(\cF,\sigma)$ be an Abelian Real sheaf on $(X,\rho)$ and let $(\fF^n,\sigma^n):=(\fF,\sigma)$ for all $n\ge 0$. Then, $HR^\ast_{ss} (\cU_\bullet,\fF^\bullet)$ can be viewed as the "Real" analogue of the usual (\emph{i.e.}, when all the Real structures are trivial) cohomology group $H^\ast(\cU_0,\fF)$ and is denoted by $HR^\ast(\cU,\fF)$. A Real $0$-cochain is a globally Real family $(\mathsf{s}_j)_{j\in J}$ of local sections. Given such a family, the differential $d^0$ gives: $(d^0\mathsf{s})_{(j_0,j_1)}=\mathsf{s}_{{j_1}_{|U_{j_0j_1}}}-\mathsf{s}_{{j_0}_{|U_{j_0j_1}}}$; it hence defines a Real $0$-cocycle if there exists a Real global section $f \in \Ga(X,\fF)$ such that $\mathsf{s}_j=f_{U_j}$ for all $j\in J$. \\

A Real $1$-coboundary is then a family $(c_{j_0j_1})_{j_0,j_1\in J}$ of sections $c_{j_0j_1}\in \fF(U_{j_0j_1})\cong \Ga(U_{j_0j_1},\cF)$ verifying $c_{\bar{j}_0\bar{j}_1}(\rho(x))=\sigma(c_{j_0j_1}(x))$ for every $x\in U_{j_0j_1}$, and such that there exists a globally Real family $(\mathsf{s}_j)_{j\in J}$ of sections $\mathsf{s}_j\in \Ga(U_j,\cF)$ such that $c_{j_0j_1}=\mathsf{s}_{j_1}-\mathsf{s}_{j_0}$ over all non-empty intersection $U_{j_0j_1}$.\\

Finally, a Real $1$-cochain $c=(c_{j_0j_1})\in CR^1_{ss}(\cU,\fF)$ can be seen as a family of sections $c_{j_0j_1}\in \Ga(U_{j_0j_1},\cF)$ satisfying $c_{\bar{j}_0\bar{j}_1}(\rho(x))=\sigma(c_{j_0j_1}(x))$. Such a cocyle is $1$-cocyle if and only if one has $(dc)_{j_0j_1j_2}=0$ for all $j_0,j_1,j_2\in J$; in other words, $c_{j_0j_1}+c_{j_1j_2}= c_{j_0j_2}$ over all non-empty intersection $U_{j_0j_1j_2}$.
\end{ex}

We can apply Lemma~\ref{lem:pre-smpl_cover} to generalize the definition of the Real cohomology groups relative to pre-simplicial Real open covers to arbitrary Real open covers of $(X_\bullet,\rho_\bullet)$.

\begin{df}
Let $(X_\bullet,\rho_\bullet)$ be a Real (pre-)simplicial space and let $(\fF^\bullet,\sigma^\bullet)\in Ob(\mathsf{Sh}_{\rho_\bullet}(X_\bullet))$. For any Real open cover $\cU_\bullet$ of $(X_\bullet,\fF^\bullet)$, we let
\begin{equation}
CR^\ast(\cU_\bullet,\fF^\bullet):=CR_{ss}^\ast({}_\natural \cU_\bullet,\fF^\bullet),
\end{equation}
and we define the \emph{Real cohomology} groups of $\cU_\bullet$ with coefficients in $(\fF^\bullet,\sigma^\bullet)$ by
\begin{equation}
HR^\ast(\cU_\bullet,\fF^\bullet):=HR^\ast_{ss}({}_\natural \cU_\bullet,\fF^\bullet).
\end{equation}
\end{df}

We head now toward the definition of the \emph{Real \v{C}ech cohomology}; roughly speaking, given an Abelian Real (pre-)sheaf $(\fF^\bullet,\sigma^\bullet)$ over a Real simplicial space $(X_\bullet,\rho_\bullet)$ , we want to define the Real cohomology groups $HR^n(X_\bullet,\fF^\bullet)$ as the inductive limit of the groups $HR^n(\cU_\bullet,\fF^\bullet)$ over some category of Real open covers of $(X_\bullet,\rho_\bullet)$. To do this, we need some preliminaries elements.

\begin{lem}
Let $(X_\bullet,\rho_\bullet)$ and $(\fF^\bullet,\sigma^\bullet)$ be as above. Assume $\cU_\bullet$ and $\cV_\bullet$ are Real open covers of $(X_\bullet,\rho_\bullet)$, with $\cU_n=(U^n_j)_{j\in J_n}$ and $\cV_n=(V^n_i)_{i\in I_n}$. Then all refinements $\theta_\bullet: (I_\bullet, \ {}^-)\To (J_\bullet, \ {}^-)$ induces  group homomorphisms 
\begin{equation}~\label{eq:theta-ast}
\theta_n^\ast: HR^n(\cU_\bullet,\fF^\bullet)\To HR^n(\cV_\bullet,\fF^\bullet).
\end{equation} 
\end{lem}

\begin{proof}
In virtue of Lemma~\ref{lem:pre-smpl_cover}, one can assume that $\cU_\bullet$ and $\cV_\bullet$ are pre-simplicial, and so that $\theta_\bullet$ is a pre-simplicial Real map. Define $\theta_n^\ast: CR^n(\cU_\bullet,\fF^\bullet)\To CR^n(\cV_\bullet,\fF^\bullet)$ as follows: for any $c=(c_j)_{j\in J_n}\in CR^n(\cU_\bullet,\fF^\bullet)$, we put 
 \[(\theta^\ast_nc)_i:=c_{_{\theta_n(i)}|V^n_i};\]
i.e. $(\theta_n^\ast c)_i$ is the image of $c_{\theta_n(i)}$ by the canonical restriction $\fF^n(U^n_{\theta_n(i)})\To \fF^n(V^n_i)$. A straightforward calculation shows that this well defines an element in $CR^n(\cV_\bullet,\fF^bullet)$. Moreover, it is clear that $\theta^\ast_n$ is a group homomorphism for any $n$. Moreover, since $\theta_\bullet$ is pre-simplicial, $\tilde{\ve}_k\circ \theta_{n+1}=\theta_n\circ \tilde{\ve}_k$. Then, for $i\in I_{n+1}$, one has
\[(d\theta_n^\ast(c))_i=\sum_{k=0}^{n+1}(-1)^k\tilde{\ve}_k^\ast(c_{_{\theta_n\circ \tilde{\ve}_k(i)}|V^n_{\tilde{\ve}_k(i)}})=\sum_{k=0}^{n+1}(-1)^k\tilde{\ve}^\ast_k(c_{_{\tilde{\ve}_k\circ \theta_{n+1}(i)}})_{|V^{n+1}_i}=(\theta_{n+1}^\ast d(c)), \] then $d^n\circ \theta_n^\ast=\theta^\ast_{n+1}\circ d^n$ for all $n\in \NN$. It turns out that $\theta^\ast_n$ maps $ZR^n(\cU_\bullet,\fF^\bullet)$ into $ZR^n(\cV_\bullet,\fF^\bullet)$ and maps $BR^n(\cU_\bullet,\fF^\bullet)$ into $BR^n(\cV_\bullet,\fF^\bullet)$. Consequently, $\theta^\ast_n$ passes through the quotients: $\theta_n^\ast([c]):=[\theta_n^\ast(c)]$, for $c\in ZR^n(\cU_\bullet,\fF^\bullet)$.
\end{proof}

As noted in~\cite{Tu1}, the map $HR^\ast(\cU_\bullet,\fF^\bullet)\To HR^\ast(\cV_\bullet,\fF^\bullet)$ may depends on the choice of the given refinement.

\begin{df}
Let $(X_\bullet,\rho_\bullet)$ and $(\fF^\bullet,\sigma^\bullet)$ be as previously. Let $\cU_\bullet$ and $\cV_\bullet$ be Real open covers of $(X_\bullet,\rho_\bullet)$. Let $\phi_n, \psi_n: CR^n(\cU_\bullet,\fF^\bullet)\To CR^n(\cV_\bullet,\fF^\bullet)$ be two families of group homomorphisms commuting with $d$. We say that $(\phi_n)_{n\in \NN}$ and $(\psi_n)_{n\in \NN}$ are \emph{equivalent} (resp. \emph{$N$-equivalent}, for a given $N\in \NN$ such that the $N$-keleton of $\cV_\bullet$ admits an $N$-simplicial Real structure) if for all $n\in \NN$ (resp. for all $n \leq N$), there exists a group homomorphism $h^n: CR^n(\cU_\bullet,\fF^\bullet)\To CR^{n-1}(\cV_\bullet,\fF^\bullet)$, with the convention that $CR^{-1}(\cV_\bullet,\fF^\bullet)=\{0\}$ (and $h^{N+1}=h^N$ in case of $N$-equivalence), such that 
\begin{equation}~\label{eq:equiv_families}
\phi_n-\psi_n=d^{n-1}\circ h^n+h^{n+1}\circ d^n, \ \forall n\in \NN \ (\text{resp.} \ \forall n\leq N). 
\end{equation} 
\end{df}

Observe that such $N$-equivalent families $\phi_\bullet$ and $\psi_\bullet$ induces group homomorphisms $$HR^n(\cU_\bullet,\fF^\bullet)\To HR^n(\cV_\bullet,\fF^\bullet),$$ also denoted by $\phi_n$ and $\psi_n$ respectively, and given by $\phi_n([c]):=[\phi_n(c)]$, and $\psi_n([c]):=[\psi_n(c)]$ for all $c\in ZR^n(\cU_\bullet,\fF^\bullet)$. Assume $h^n: CR^n(\cU_\bullet,\fF^\bullet)\To CR^{n-1}(\cV_\bullet,\fF^\bullet)$ is such that~\eqref{eq:equiv_families} holds for all $n\leq N$, then for all $c\in ZR^n(\cU_\bullet,\fF^\bullet)$, one has
\[(\phi_n-\psi_n)([c])=[d^{n-1}(h^nc)]+[h^{n+1}(d^nc)]=0;\]
in other words, $\phi_n$ and $\psi_n$ define the same homomorphism from $HR^n(\cU_\bullet,\fF^\bullet)$ to $HR^n(\cV_\bullet,\fF^\bullet)$ when $n\leq N$.\\
It is clear that ($N$-)equivalence of morphisms $\phi_n: CR^n(\cU_\bullet,\fF^\bullet)\To CR^n(\cV_\bullet,\fF^\bullet)$ is an equivalence relation. We also denote by $\phi_\bullet$ for the ($N$-)class of $\phi_\bullet$.

\begin{df}
Denote by $\mathfrak{N}$ the collection of all Real open covers of $(X_\bullet,\rho_\bullet)$. Let $\cU_\bullet, \ \cV_\bullet \in \mathfrak{N}$. We say that $\cV_\bullet$ is \emph{$h$-finer} than $\cU_\bullet$ if $\cV_\bullet$ is finer than $\cU\bullet$ in the sense of Definition~\ref{df:refinement}, and if there exists $N\in \NN$ such that the $N$-skeleton of $\cV_\bullet$ admits an $N$-simplicial Real strucutre. In this case, we will write $\cU_\bullet \preceq_N \cV_\bullet$ or $\cU_\bullet \preceq_h \cV_\bullet$.
\end{df}

We refer to~\cite[Lemma 4.5]{Tu1}) for the proof of the following

\begin{lem}~\label{lem:N-class}
Let $\cU_\bullet$ and $\cV_\bullet$ be Real open covers of $(X_\bullet,\rho_\bullet)$ such that  $\cU_\bullet \preceq_N \cV_\bullet$. If $\theta_\bullet,\theta_\bullet':(I_\bullet,\ {}^-)\To (J_\bullet,\ {}^-)$ are two arbitrary refinements, then their induced group homomorphisms $\theta_\bullet^\ast$ and $(\theta'_\bullet)^\ast$  are $N$-equivalent. Consequently, there is a canonical morphism $$HR^n(\cU_\bullet,\fF^\bullet)\To HR^n(\cV_\bullet,\fF^\bullet)$$ for each $n\leq N$.
\end{lem}

\begin{ex}
By Lemma~\ref{lem:pre-smpl_cover}, from anyy Real open cover $\cU_\bullet$ of $(X_\bullet,\rho_\bullet)$ and anyy $N\in \NN$, one  can form an $N$-simplicial Real open cover ${}_{\natural^N}\cU_\bullet$ of the $N$-skeleton of $(X_\bullet,\rho_\bullet)$. Next, we define a new Real open cover ${}_\natural\cU_\bullet^N$ by setting
\begin{eqnarray}~\label{eq:U^N}
{}_\natural\cU_n^N := \left\{ \begin{array}{ll}  {}_{\natural^N}\cU_n, & \text{if} \ n\leq N \\ \cU_n, & \text{if} \ n\ge N+1  
\end{array} \right.
\end{eqnarray}
It is clear that the $N$-skeleton of ${}_\natural\cU_\bullet^N$ admits an $N$-simplicial Real structure. Recall that ${}_\natural \cU_\bullet^N$ is indexed by $I_\bullet$, with $I_n=\Lam_n^N$ if $n\leq N$ and $I_n=J_n$ if $n\ge N+1$. Now we get a refinement ${}_N\theta_\bullet: (I_\bullet, \ {}^-)\To (J_\bullet, \ {}^-)$ by setting
\begin{equation}~\label{eq:theta_N}
{}_N\theta_n:= \left\{\begin{array}{ll} 
                             \Lam_n^N\To J_n, \ \lambda \mto \lambda(\Id_{[n]}), & \text{if} \ n\leq N \\
                              \Id: J_n\To J_n, & \text{if} \ n\ge N+1 
 \end{array} \right. \end{equation}
 hence $\cU_\bullet \preceq_N {}_\natural \cU_\bullet^N$ for all $N\in \NN$. In particular, $\cU_\bullet \preceq_0 \cU_\bullet$.
\end{ex}

We deduce from the example above that "$\preceq_h$" is a pre-order in the collection $\mathfrak{N}$. Suppose that $\cU_\bullet \preceq_h \cV_\bullet \preceq_h \cW$ and $K_\bullet \overset{\theta'_\bullet}{\To}I_\bullet \overset{\theta_\bullet}{\To}J_\bullet$ are refinements. Then it is easy to check that the maps $\theta_\bullet^\ast$ and $(\theta'_\bullet)^\ast$ defined by~\eqref{eq:theta-ast} verify the relation $(\theta_n \circ \theta'_n)^\ast=(\theta'_n)^\ast \circ \theta_n^\ast$ for all $n\in \NN$.\\

For $n\in \NN$, we denote by $\mathfrak{N}(n)$ the collection of all elements $\cU_\bullet \in \mathfrak{N}$ such that $\cU_\bullet \preceq_N \cU_\bullet$ for some $N\ge n+1$; \emph{i.e}., $\cU_\bullet \in \mathfrak{N}(n)$ if there is $N\ge n+1$ such that the $N$-skeleton of $\cU_\bullet$ admits an $N$-simplicial Real structure. It is obvious that "$\preceq_h$" is also a preorder in $\mathfrak{N}(n)$. Furthermore, Lemma~\ref{lem:N-class}, states that if $\cU_\bullet \preceq_h \cV_\bullet$ in $\mathfrak{N}(n)$, there is a canonical map $HR^n(\cU_\bullet,\fF^\bullet)\To HR^n(\cV_\bullet,\fF^\bullet)$. It follows that for all $n\in \NN$, the collection $$\left\{HR^n(\cU_\bullet,\fF^\bullet) \ | \ \cU_\bullet \in \mathfrak{N}(n)\right\}$$ is a directed system of groups; this allows us to give the following definition.

\begin{df}
We define the \emph{$n^{th}$ \v{C}ech cohomology group} of $(X_\bullet,\rho_\bullet)$ with coefficients in $(\fF^\bullet,\sigma^\bullet)$ to be the direct limit
\begin{eqnarray}
\check{H}R^n(X_\bullet,\fF^\bullet):=\underset{\cU_\bullet\in \mathfrak{N}(n)}{\varinjlim}HR^n(\cU_\bullet,\fF^\bullet).
\end{eqnarray}
\end{df}

\begin{lem}
For every $\cU_\bullet\in \mathfrak{N}$, pre-simplicial or not, there is a canonical group homomorphism 
\[\theta_{\cU_\bullet}: HR^n(\cU_\bullet,\fF^\bullet)\To \check{H}R^n(X_\bullet,\fF^\bullet),\] for all $n\in \NN$.
\end{lem}

\begin{proof}
For every $\cU_\bullet \in \mathfrak{N}$ (simplicial or not), and for every $n\in \NN$, we define the map $$\theta_{\cU_\bullet}:HR^n(\cU_\bullet,\fF^\bullet)\To \check{H}R^n(X_\bullet,\fF^\bullet)$$ by composing  the canonical homomorphism $${}_N\theta_n^\ast: HR^n(\cU_\bullet,\fF^\bullet)\To HR^n({}_\natural \cU_\bullet^N,\fF^\bullet)$$ with the canonical projection $$p^N_{\cU_\bullet}: HR^n({}_\natural \cU_\bullet^N,\fF^\bullet)\To \check{H}R^n(X_\bullet,\fF^\bullet),$$ for some $N\ge n+1$; \emph{i.e.} $\theta_{\cU_\bullet}=p_{\cU_\bullet}^N \circ {}_N\theta_n^\ast$ (recall that ${}_N\theta_n$ is defined by~\eqref{eq:theta_N}).
\end{proof}

 Let $(\fF^\bullet,\sigma^\bullet)$ and $(\fG^\bullet,\varsigma^\bullet)$ be Abelian Real sheaves on a Real simplicial space $(X_\bullet,\rho_\bullet)$. Suppose that $\phi_\bullet=(\phi_n)_{n\in \NN}: (\fF^\bullet,\sigma^\bullet)\To (\fG^\bullet,\varsigma^\bullet)$ is a morphism of Abelian Real (pre)sheaves, and that $\cU_\bullet$ is a Real open cover of $(X_\bullet,\rho_\bullet)$. Consider the pre-simplicial Real open cover ${}_\natural\cU_\bullet$ associated to $\cU_\bullet$. Then for any $n\in \NN$, and any $\lambda\in \Lam_n$, there is a morphism of Abelian groups
\begin{equation}
\tilde{\phi}_n: \fF^n(U^n_\lambda)\To \fG^n(U^n_\lambda), \mathsf{s}_\lambda \mto \phi_{n|U^n_\lambda}(\mathsf{s}_\lambda),
\end{equation}
satisfying $\varsigma^n_{U^n_\lambda}\circ \tilde{\phi}_n=\tilde{\phi}_n\circ \sigma_{U^n_{\bar{\lambda}}}$. This gives a group homomorphism $$\tilde{\phi}_n:CR^n_{ss}({}_\natural \cU_\bullet,\fF^\bullet)_{\sigma^\bullet}\To CR^n_{ss}({}_\natural\cU_\bullet,\fG^\bullet)_{\varsigma^\bullet}.$$
Moreover, for any $\lambda\in \Lam_{n+1}$ and any $k\in [n+1]$, one has a commutative diagram
\[ 
\xymatrix{ \fF^n(U^n_{\tilde{\ve}_k(\lambda)})\ar[d]_{\tilde{\ve}^\ast_k} \ar[r]^{\phi_{n|U^n_{\tilde{\ve}_k(\lambda)}}} & \fG^n(U^n_{\tilde{\ve}_k(\lambda)}) \ar[d]^{\tilde{\ve}_k^\ast} \\
\fF^{n+1}(U^{n+1}_\lambda) \ar[r]^{\phi_{n+1|U^{n+1}_\lambda}} & \fG^{n+1}(U^{n+1}_\lambda)}
\]
Thus, $d^n\circ \tilde{\phi}_n=\tilde{\phi}_{n+1}\circ d^n$; i.e. one has a commutative diagram
\begin{equation}
\xymatrix{CR^n_{ss}({}_\natural\cU_\bullet,\fF^\bullet)_{\sigma^\bullet}\ar[d]^{\tilde{\phi}_n} \ar[r]^{d^n} & CR^{n+1}_{ss}({}_\natural\cU_\bullet,\fF^\bullet)_{\sigma^\bullet} \ar[d]^{\tilde{\phi}_{n+1}} \\ CR^n_{ss}({}_\natural\cU_\bullet,\fG^\bullet)_{\varsigma^\bullet} \ar[r]^{d^n} & CR^{n+1}_{ss}({}_\natural\cU_\bullet,\fG^\bullet)_{\varsigma^\bullet} }
\end{equation}
that shows that $\phi$ gives rise to a homomorphism of Abelian groups 
\begin{equation}
(\phi_{\cU_\bullet})_\ast: HR^n(\cU_\bullet,\fF^\bullet)_{\sigma^\bullet}\To HR^n(\cU_\bullet,\fG^\bullet)_{\varsigma^\bullet}, \ [c]\mto [\tilde{\phi}_n(c)];
\end{equation}
and therefore a group homomorphism $\phi_\ast: \check{H}R^n(X_\bullet,\fF^\bullet)_{\sigma^\bullet}\To \check{H}R^n(X_\bullet,\fG^\bullet)_{\varsigma^\bullet}$ defined in the obvious way. We thus have shown that $\check{H}R^\ast$ is functorial in the category $\mathsf{Sh}_{\rho_\bullet}(X_\bullet)$.  

\begin{pro}~\label{pro:HR-long-exct-seq}
Suppose $(X_\bullet,\rho_\bullet)$ is a Real simplicial space such that each $X_n$ is paracompact. If 
\[ 0\To (\fF'^\bullet,\sigma'^\bullet)\stackrel{\phi_\bullet'}{\To} (\fF^\bullet,\sigma^\bullet) \stackrel{\phi_\bullet}{\To} (\fF"^\bullet,\sigma"^\bullet)\To 0
\]
is an exact sequence of Real (pre-)sheaves over $(X_\bullet,\rho_\bullet)$, then there is a long exact sequence of Abelian groups 
\[0 \To \check{H}R^0(X_\bullet,\fF'^\bullet) \stackrel{\phi'_\ast}{\To} \check{H}R^0(X_\bullet,\fF^\bullet)\stackrel{\phi_\ast}{\To}\check{H}R^0(X_\bullet,\fF"^\bullet) \stackrel{\partial}{\To} \check{H}R^1(X_\bullet,\fF'^\bullet) \stackrel{\phi'_\ast}{\To} \cdots\]
\end{pro}

The proof of this proposition is almost the same as in~\cite[\S4]{Tu1}.


\subsection{Comparison with usual groupoid cohomologies}

In this subsection we compare our cohomology with the usual cohomology theory in some special cases, especially with that developed in~\cite{Tu1}. 

\begin{pro}~\label{pro:HR-vs-H}
Suppose $\bfS$ is an Abelian Real group. Let ${}^r\bfS$ be the fixed point subgroup of $\bfS$. Let $(\cG,\rho)$ be a Real groupoid. Then if $\rho$ is trivial , we have 
\[
\check{H}R^\ast(\cG_\bullet,\bfS) = \check{H}^\ast(\cG_\bullet,{}^r\bfS).
\]
In particular, if \ $\bfS$ has no non-trivial fixed point, we have $\check{H}R^\ast(\cG_\bullet,\bfS)=0$.	
\end{pro}

Notice that this result generalizes easily to the Real cohomology with coefficients in a Real sheaf induced from a Real $\cG$-module.

\begin{proof}
Let $(c_\lambda)\in ZR^n(\cU_\bullet,\bfS)$. Since $\rho=\Id$, we may take the involution on $J_\bullet$ to be trivial. For every $\overrightarrow{g}\in U^n_\lambda$, we have 
\[
c_\lambda(\overrightarrow{g})=c_\lambda(\overline{\overrightarrow{g}})=\overline{c_\lambda(\overrightarrow{g})} \in {}^r\bfS.
\]
Thus $c_\lambda \in ZR^n(\cU_\bullet,{}^r\bfS)$.

Conversely, we obviously have $\check{H}^n(\cG_\bullet,{}^r\bfS)\subset \check{H}R^n(\cG_\bullet,\bfS)$ since $\rho$ is trivial.	
\end{proof}

\begin{cor}
If $\rho$ and the Real structure of \ $\bfS$  are trivial, then $\check{H}^\ast(\cG_\bullet,\bfS)=\check{H}^\ast(\cG_\bullet,\bfS)$.	
\end{cor}

Let us focus now on the case where $\cG$reduces to a Real space $(X,\tau)$ and $\bfS=\ZZ^{0,1}$. Then $\tau$ induces an action of $\ZZ_2$ on $X$ by $(-1)\cdot x:=\tau(x), (+1)\cdot x:=x$.

\begin{pro}

\begin{itemize}
\item[(i)] $\check{H}R^\ast(X,\ZZ^{0,1})\cong \check{H}^\ast_{(\ZZ_2,-)}(X,\ZZ)$, where the sign "$-$" stands for the $\ZZ_2$-equivariant cohomology with respect to the action of $\ZZ_2$ on $\ZZ$ given by $(-1)\cdot n:=-n, (+1)\cdot n:=n$.
\item[(ii)] $\check{H}^\ast(X,\ZZ)\cong_\QQ \check{H}^\ast_{(\ZZ_2,-)}(X,\ZZ)\oplus \check{H}^\ast_{(\ZZ_2,+)}(X,\ZZ)$, where the sign "$+$" means the trivial $\ZZ_2$-action on $\ZZ$.	
\end{itemize}	
\end{pro}

\begin{proof}
\begin{itemize}
	\item[(i)] Let $c\in \check{H}R^n(X,\ZZ^{0,1})$ be represented on the Real open cover $(U_j)$ of $X$. Then $c_{\bar{j}_0...\bar{j}_n}(\tau(x))=-c_{j_0...j_n}(x)$ implies $\tau^\ast c_{j_0...j_n}(x)=-c_{j_0...j_n}(x), \forall x\in X$; in other words, $c$ is $\ZZ_2$-equivariant with respect to the $\ZZ_2$-action "$-$" on $\ZZ$. The converse is easy to check. 
	\item[(ii)] We define the involution $\tilde{\tau}$ on $\check{H}^n(X,\ZZ)$ by $\tilde{\tau}(c):=-\tau^\ast c$. Then it is straightforward that the Real part ${}^r \check{H}^n(X,\ZZ)\cong \check{H}R^n(X,\ZZ^{0,1})$, while the imaginary part ${}^\cI \check{H}^n(X,\ZZ)$ is exactly $\check{H}^n_{(\ZZ_2,+)}(X,\ZZ)$.	
	\end{itemize}	
\end{proof}


\subsection{The group $\check{H}R^0$}

We shall recall the notations of~\cite[Section 4]{Tu1} that we will use throughout the rest of the section. Let $\cU_\bullet$ be a Real open cover of a Real simplicial space $(X_\bullet,\rho_\bullet)$ and let ${}_\natural \cU_\bullet$ be its associated pre-simplicial Real open cover. Recall that any $\vp\in \cP_n^k$ is represented by its image in $[n]$; \emph{i.e.} $\vp=\{\vp(0),...,\vp(k)\}$. Then $\cP_n$ is nothing but the collection of all non empty subsets of $[n]$. Henceforth, any subset $S=\{i_0,...,i_k\}\subseteq [n]$, with $i_0\leq...\leq i_k$, designates the maps $\vp\in \cP_n^k$ such that $\vp(0)=i_0,...,\vp(k)=i_k$.

\begin{nota}~\label{notation-lambda}
With the above observations, any element $\lambda\in \Lam_n$ is represented by a $(2^{n+1}-1)-tuple$ $(\lambda_S)_{\emptyset \neq S \subseteq [n]}$, where the subsets $S$ are ordered first by cardinality, then by lexicographic order; \emph{i.e.} $$S\in \left\{ \{0\},...,\{n\},\{0,1\},...,\{0,n\},\{1,2\},...,\{1,n\},..., \{0,1,2\},....\{0,1,n\}, ...,\{0,...,n\}\right\},$$ and $\lambda_S:=\lambda(S)$. For instance, any element $\lambda\in \Lam_1$ is represented by a triple $(\lambda_0,\lambda_1,\lambda_{01})$, with $\lambda_0=\lambda(\{0\}), \ \lambda_1=\lambda(\{1\})$ and $\lambda_{01}=\lambda(\{0,1\})$.
\end{nota}

Recall that if $(\fF^\bullet,\sigma^\bullet)$ is an abelian Real sheaf over $(X_\bullet,\rho_\bullet)$, we are given two "restriction" maps on the space of global Real sections $\tilde{\ve}_0^\ast, \tilde{\ve}_1^\ast: \fF^0(X_0)_{\sigma^0}\To\fF^1(X_1)_{\sigma^1}$. Let us set
\begin{eqnarray*} 
\Ga_{\text{inv}}(\fF^\bullet)_{\sigma^\bullet}:=\ker\left(\xymatrix{ \fF^0(X_0)_{\sigma^0} \dar[r]^{\tilde{\ve}_0^\ast}_{\tilde{\ve}_1^\ast} & \fF^1(X_1)_{\sigma^1} }\right)=\left\{ \mathsf{s}\in \fF^0(X_0)_{\sigma^0} \ | \ \tilde{\ve}_0^\ast(\mathsf{s})=\tilde{\ve}_1^\ast(\mathsf{s}) \right\}.
\end{eqnarray*}

\begin{pro}(~\cite[Proposition 5.1]{Tu1})
Let $(\fF^\bullet,\sigma^\bullet)$ be an abelian Real sheaf over $(X_\bullet,\rho_\bullet)$ and let $\cU_\bullet$ be a Real open cover of $(X_\bullet,\rho_\bullet)$. Then
\begin{equation}
\check{H}R^0(X_\bullet,\fF^\bullet)_{\sigma^\bullet}\cong HR^0(\cU_\bullet,\fF^\bullet)_{\sigma^\bullet}\cong \Ga_{\text{inv}}(\fF^\bullet)_{\sigma^\bullet}.
\end{equation}
\end{pro}

\begin{proof}
One identifies $\Lam_0$ with $J_0$. Note that $\cP_1=\{\ve_0^1,\ve_1^1,\Id_{[1]}\}$, and that for any $\lambda=(\lambda_0,\lambda_1,\lambda_{01})$ in $\Lam_1$ one has $\tilde{\ve}_0(\lambda)=\lambda(\ve_0)=\lambda_1$, $\tilde{\ve}_1(\lambda)=\lambda(\ve_1)=\lambda_0$. We thus have $U^1_\lambda=U^1_{\lambda_{01}}\cap \tilde{\ve}_0^{-1}(U^0_{\lambda_1})\cap \tilde{\ve}^{-1}_1(U^0_{\lambda_0})$. Now, let $(\mathsf{s}_{\lambda_0})_{\lambda_0\in J_0}\in ZR^0(\cU_\bullet,\fF^\bullet)_{\sigma^\bullet}$. Then
\begin{equation}~\label{eq:0-cocycl1}
0=(d\mathsf{s})_{(\lambda_0,\lambda_1,\lambda_{01})}=\tilde{\ve}_0^\ast(\mathsf{s}_{\lambda_1})-\tilde{\ve}_1^\ast(\mathsf{s}_{\lambda_0}), \ \text{on} \ U^1_\lambda,
\end{equation}

Therefore, $\tilde{\ve}_0^\ast(\mathsf{s}_{\lambda_1})=\tilde{\ve}^\ast_1(\mathsf{s}_{\lambda_0})$ on $\tilde{\ve}^{-1}_0(U^0_{\lambda_1})\cap \tilde{\ve}^{-1}_1(U^0_{\lambda_0})$, and $\tilde{\ve}_0^\ast(\mathsf{s}_{\bar{\lambda}_1})=\tilde{\ve}^\ast_1(\mathsf{s}_{\bar{\lambda}_0})$ on $\tilde{\ve}^{-1}_0(U^0_{\bar{\lambda}_1})\cap \tilde{\ve}^{-1}_1(U^0_{\bar{\lambda}_0})$, for all $\lambda_0,\lambda_1\in J_0$. Applying $\tilde{\eta}^\ast_0$ to both sides of the above identity, we get that $\mathsf{s}_{\lambda_0}=\mathsf{s}_{\lambda_1}$ and $\mathsf{s}_{\bar{\lambda}_0}=\mathsf{s}_{\bar{\lambda}_1}$; in other words, $\mathsf{s}_{\lambda_0}=\mathsf{s}_{\lambda_1}$ on $U^0_{\lambda_0}\cap U^0_{\lambda_1}$ for all $\lambda_0,\lambda_0\in J_0$. Since $(\fF^0,\sigma^0)$ is a Real sheaf on $(X_0,\rho_0)$, there exists a global Real sections $\mathsf{s}\in \fF^0(X_0)_{\sigma^0}$ such that $\mathsf{s}_{U^0_{\lambda_0}}=\mathsf{s}_{\lambda_0}$ for all $\lambda_0\in J_0$. Now, equation~\eqref{eq:0-cocycl1} is equivalent to $\tilde{\ve}_0^\ast(\mathsf{s})=\tilde{\ve}^\ast_1(\mathsf{s})$; \emph{i.e.}, $\mathsf{s}\in \Ga_{\text{inv}}(\fF^\bullet)_{\sigma^\bullet}$ and this ends the proof.
\end{proof}


\subsection{$\check{H}R^1$ and the Real Picard group}~\label{section:HR1}

Let us consider the same data as in the previous subsection. Let $\cU_\bullet$ be a Real open cover of $(X_\bullet,\rho_\bullet)$. For $\lambda=(\lambda_0,\lambda_1,\lambda_2,\lambda_{01},\lambda_{02},\lambda_{12},\lambda_{012})\in \Lam_2$, one has 
\begin{equation}~\label{eq:U^2_lambda}
U^2_\lambda=\tilde{\vp}_{00}^{-1}(U^0_{\lambda_0})\cap \tilde{\vp}^{-1}_{01}(U^0_{\lambda_1})\cap \tilde{\vp}^{-1}_{02}(U^0_{\lambda_2})\cap \tilde{\ve}^{-1}_2(U^1_{\lambda_{01}})\cap \tilde{\ve}^{-1}_1(U^1_{\lambda_{02}})\cap \tilde{\ve}_0^{-1}(U^1_{\lambda_{12}})\cap U^2_{\lambda_{012}},
\end{equation}
where $\vp_{00}=\ve^2_1\circ \ve^1_1, \vp_{01}=\ve^2_0\circ \ve^1_0$ and $\vp_{02}=\ve^2_1\circ \ve^1_0$. \\
Let $c=(c_{\lambda})_{\lambda \in \Lam_1}\in ZR^1(\cU_\bullet,\fF^\bullet)_{\sigma^\bullet}$. Then
\begin{eqnarray}
0=(dc)_{\lambda_0 \lambda_1\lambda_2\lambda_{01}\lambda_{02}\lambda_{12}\lambda_{012}}=\tilde{\ve}^\ast_0 c_{\lambda_1\lambda_2\lambda_{12}}-\tilde{\ve}^\ast_1 c_{\lambda_0\lambda_2\lambda_{02}}+\tilde{\ve}^\ast_2 c_{\lambda_0\lambda_1\lambda_{02}}, \ \text{on} \ U^2_\lambda,
\end{eqnarray}~\label{eq:1-cocycl1}
and of course we get a similar identities for $(dc)_{\bar{\lambda}_0\bar{\lambda}_1\bar{\lambda}_2\bar{\lambda}_{01}\bar{\lambda}_{02}\bar{\lambda}_{12}\bar{\lambda}_{012}}$ on $U^2_{\bar{\lambda}}$.
Now applying $\tilde{\eta}_1^\ast$ to~\eqref{eq:1-cocycl1}, we obtain 
\[c_{\lambda_0\lambda_1\lambda_{01}}=c_{\lambda_0\lambda_1\lambda_{02}}-c_{\lambda_1\lambda_2\lambda_{12}}\]
on $\tilde{\ve}^{-1}_1(U^0_{\lambda_0})\cap \tilde{\ve}^{-1}_0(U^0_{\lambda_1})\cap \tilde{\ve}^{-1}_0(U^0_{\lambda_2})\cap U^1_{\lambda_{01}}\cap U^1_{\lambda_{02}}\cap U^1_{\lambda_{12}}\cap \tilde{\eta}^{-1}_1(U^2_{\lambda_{012}})$, which means that for any $\lambda_0,\lambda_1,\lambda_{01}\in J_0$, $\mathsf{s}_{\lambda_0\lambda_1\lambda_{01}}$ does not depends on the choice of $\lambda_{01}$. Therefore, there exists a Real family $(f_{\lambda_0\lambda_1})\in \prod_{\lambda_0,\lambda_1\in\Lam_0} \fF^1\left(\tilde{\ve}^{-1}_1(U^0_{\lambda_0})\cap \tilde{\ve}^{-1}_0(U^0_{\lambda_1})\right)$ such that $f_{\lambda_0\lambda_1|U^1_{\lambda_0\lambda_1\lambda_{01}}}=c_{\lambda_0\lambda_1\lambda_{01}}$ for any $(\lambda_0,\lambda_1,\lambda_{01})\in \Lam_1$. Now, the cocycle relation~\eqref{eq:1-cocycl1} becomes
\begin{equation}~\label{eq:1-cocycl2}
\tilde{\ve}^\ast_0 f_{\lambda_1\lambda_2}-\tilde{\ve}^\ast_1 f_{\lambda_0\lambda_2}+\tilde{\ve}^\ast_2 f_{\lambda_0\lambda_1}
\end{equation}
on $U^1_{\lambda_0\lambda_1\lambda_{01}}\cap U^1_{\lambda_{02}}\cap U^1_{\lambda_{12}}$.\\

Let $(\cG,\rho)$ be a locally compact Hausdorff Real groupoid. We are interested in the $1^{st}$ Real \v{C}ech cohomology group of $(\cG_\bullet,\rho_\bullet)$ with coefficients in the Abelian Real sheaf $(\cS^\bullet,\sigma^\bullet)=(\bfS,\sigma)$ over $(\cG_\bullet,\rho_\bullet)$ associated to the Real $\cG$-module $(\Gpdo \times \bfS, \rho\times \ {}^-)$, where $(\bfS,\ {}^-)$ is an Abelian group endowed with the trivial $\cG$-action. Note that in this case, for any pre-simplicial Real open cover $\cU_\bullet\in \mathfrak{N}(n)$ of $(\cG_\bullet,\rho_\bullet)$, elements of the group $CR^n(\cU_\bullet,\cS^\bullet)$ are of the form $(c_\lambda)_{\lambda \in \Lam_n}$, where $c_\lambda\in \Ga(U_\lambda^n,\bfS)$ are such that $c_{\bar{\lambda}}(\rho_n(\overrightarrow{g}))=\overline{c_\lambda(\overrightarrow{g})}\in \bfS$ for any $\overrightarrow{g}\in U^n_\lambda\subset \cG_n$.

\begin{pro}~\label{pro:HR1-H1}
With the above notations, the Real \v{C}ech cohomology group $\check{H}R^1(\cG_\bullet,\bfS)$ is isomorphic to the group $\Hom_{\RG}(\cG,\bfS)$ of isomorphism classes of Real generalized homomorphisms $(\cG,\rho)\To (\bfS,\ {}^-)$.
\end{pro}

\begin{proof}
The operations in $\Hom_{\RG}(\cG,S)$ are defined as follows. If $(Z,\tau), (Z',\tau'):(\cG,\rho)\To (S,\ {}^-)$ are Real generalized homomorphisms, their sum is 
\begin{equation}~\label{eq:sum-S-bdl}
(Z,\tau)+(Z',\tau'):=Z\times_{\Gpdo}Z'/_{\sim}
\end{equation}
where $(z,z')\sim (z\cdot t^{-1},z'\cdot t)$ for all $t\in \bfS$, together with the obvious Real structure $\tau\times \tau'$. The inverse of $(Z,\tau)$ is $(Z^{-1},\tau)$, where $Z^{-1}$ is $Z$ as a topological space, and if $\flat: Z\hookrightarrow Z^{-1}$ is the identity map, then the $S$-action on $Z^{-1}$ is defined by $\flat(z)\cdot t:=\flat(z\cdot t^{-1})$ and the $\cG$-action is defined as follows: $(g,\flat(z))\in \cG\ltimes Z^{-1}$ if and only if $(g,z)\in \cG\ltimes Z$, in which case we set $g\cdot \flat(z):=\flat(g\cdot z)$. Finally, the Real structure on $Z^{-1}$ is $\tau(\flat(z)):=\flat(\tau(z))$. Then we define the sum in $\Hom_{\RG}(\cG,S)$ by $[Z,\tau]+[Z',\tau']:=[(Z,\tau)+(Z',\tau')]$, and we put $[Z,\tau]^{-1}:=[(Z^{-1},\tau)]$. It is not hard to check that subject to these operations, $\Hom_{\RG}(\cG,S)$ is an Abelian group.\

Now, suppose we are given a Real open cover $\cU_0=(U^0_j)_{j\in J_0}$ of $(\Gpdo,\rho)$ trivializing the Real generalized homomorphism $(Z,\tau):(\cG,\rho)\To (S,\ {}^-)$. Let $(\mathsf{s}_j)_{j\in J_0}$ be a Real family of local sections of the $S$-principal Real bundle $\fr: (Z,\tau)\To(\Gpdo,\rho)$. Form a pre-simplicial Real open cover $\cU_\bullet$ of the Real simplicial space $(\cG_\bullet,\rho_\bullet)$ by setting $J_n:=J_0^{n+1}$, $\cU_n:=(U^n_{(j_0,...,j_n)})_{(j_0,...,j_n)\in J_n}$, where
\begin{equation}~\label{eq:G-cover}
U^n_{(j_0,...,j_n)}:=\left\{(g_1,...,g_n)\in \cG_n \ | \ r(g_1)\in U^0_{j_0},...,r(g_n)\in U^0_{j_{n-1}}, s(g_n)\in U^0_{j_n}\right\}.
\end{equation}
Then, for all $g\in U^1_{(j_0,j_1)}$, $\fr(g\cdot\mathsf{s}_{j_1}(s(g)))=r(g)=\fr(\mathsf{s}_{j_0}(r(g)))$; hence, there exists a unique element $c_{j_0j_1}(g)\in S$ such that $g\cdot \mathsf{s}_{j_1}(s(g))=\mathsf{s}_{j_0}(r(g))\cdot c_{j_0j_1}(g)$. We then obtain a family of continuous functions $c_{j_0j_1}:U^1_{(j_0,j_1)}\To S$ such that 
\begin{equation}
g\cdot\mathsf{s}_{j_1}(s(g))=\mathsf{s}_{j_0}(r(g))\cdot c_{j_0j_1}(g), \ \forall g\in U^1_{(j_0,j_1)}.
\end{equation}
Furthermore, notice that $U^1_{(j_0,j_1)}=\tilde{\ve}^{-1}_0(U^0_{j_1})\cap \tilde{\ve}^{-1}_1(U^0_{j_0})$. Let $(g_1,g_2)\in U^2_{(j_0,j_1,j_2)}$. Then
\begin{align*}
(g_1g_2)\cdot \mathsf{s}_{j_2}(s(g_2)) & =  g_1\cdot\mathsf{s}_{j_1}(r(g_2))\cdot c_{j_1j_2}(g_2) = g_1\cdot\mathsf{s}_{j_1}(s(g_1))\cdot c_{j_1j_2}(g_2) \\
 & = \mathsf{s}_{j_0}(r(g_1))\cdot c_{j_0j_1}(g_1)\cdot c_{j_1j_2}(g_2);
\end{align*}
hence $c_{j_0j_2}(g_1g_2)=c_{j_0j_1}(g_1)\cdot c_{j_1j_2}(g_2)$. In other words, $$\tilde{\ve}^\ast_0 c_{\tilde{\ve}_0(j_0,j_1,j_2)}\cdot(\tilde{\ve}^\ast_1 c_{\tilde{\ve}_1(j_0,j_1,j_2)})^{-1}\cdot\tilde{\ve}^\ast_2 c_{\tilde{\ve}_2(j_0,j_1,j_2)}=1$$ over all $U^2_{(j_0,j_1,j_2)}$. Moreover, we clearly have $c_{\bar{j}_0\bar{j}_1}(\rho(g))=\overline{c_{j_0j_1}(g)}\in S$. This gives us a Real $1$-cocycle $(c_{j_0j_1})_{(j_0,j_1)\in J_1}\in ZR^1(\cU_\bullet,\cS^\bullet)$. \

Suppose $f:(Z,\tau)\To (Z',\tau')$ is an isomorphism of Real generalized morphisms (see chapter 2). Up to a refinement, we can choose $\cU_0$ in such a way that we have two Real families $(\mathsf{s}_j)_{j\in J_0}, \ (\mathsf{s}')_{j\in J_0}$ of local sections of the Real projections $\fr: (Z,\tau)\To (X,\rho)$ and $\fr':(Z',\tau')\To (X,\rho)$ respectively. Since for all $j\in J_0$ and $x\in U_j$, $\fr'(f_{U_j}(\mathsf{s}_j)(x))=\fr(\mathsf{s}_j(x))=x=\fr'(\mathsf{s}_j'(x))$, there exists a unique element $\vp_j(x)\in S$ such that $\mathsf{s}_j'(x)=f_{U_j}(\mathsf{s}_j(x))\cdot\vp_j(x)$, and this gives a Real family of continuous functions $\vp_j:U_j\To S$. It follows that if $c=(c_{j_0j_1})$ and $c'=(c_{j_0j_1}')$ are the Real $1$-cocycle associated to $(Z,\tau)$ and $(Z',\tau')$ respectively. Then, over $U^1_{(j_0,j_1)}$, one has
\[
g\cdot f_{U_{j_1}}(\mathsf{s}_{j_1}(s(g)))\cdot \vp_{j_1} = f_{U_{j_0}}(\mathsf{s}_{j_0}(r(g)))\cdot\vp_{j_0}(r(g))\cdot c_{j_0j_1}'(g);
\]
But, since $f$ is $\cG$-$S$-equivariant, we get
\[
f_{U_{j_0}(\mathsf{s}_{j_0}(r(g)))}\cdot c_{j_0j_1}(g)\cdot\vp_{j_1}(s(g))=f_{U_{j_0}}(\mathsf{s}_{j_0}(r(g)))\cdot\vp_{j_0}(r(g))\cdot c_{j_0j_1}'(g);
\]
thus $c_{j_0j_1}'(g)\cdot c_{j_0j_1}^{-1}(g)=\vp_{j_1}(s(g))\cdot \vp_{j_0}(r(g))^{-1}$, or $(c'\cdot c^{-1})_{(j_0,j_1)}=\tilde{\ve}_0^\ast \vp_{\tilde{\ve}_0(j_0,j_1)}\cdot \tilde{\ve}_1^\ast \vp_{\tilde{\ve}_1(j_0,j_1)}^{-1}$ for all $(j_0,j_1)\in J_1$. This shows that $c'.c^{-1}\in BR^1(\cU_\bullet,\bfS)$. We then deduce a well defined group homomorphism
\begin{equation}
c_1: \Hom_{\RG}(\cG,\bfS)\To \check{H}R^1(\cG_\bullet,\bfS), \ c_1([Z,\tau]):=[c_{j_0j_1}]\in HR^1(\cU_\bullet,\bfS),
\end{equation}
where $\cU_\bullet$ is the Real open cover defined from any Real local trivialization of $(Z,\tau)$.\\

Conversely, given a Real \v{C}ech $1$-cocycle $c=(c_{\lambda_0\lambda_1})$ over a pre-simplicial Real open cover $\cU_\bullet \in \mathfrak{N}(1)$, we let $Z:=\coprod_{\lambda_0\in \Lam_0}U_{\lambda_0}\times \bfS$, together with the Real structure $\nu$ defined by $\nu(x,t):=(\rho(x),\bar{t})$, and equipped with the Real $\cG$-action $g\cdot(s(g),t):=(r(g),c_{\lambda_0\lambda_1}(g)\cdot t)$ for any $g\in U^1_{\lambda_0\lambda_1\lambda_{01}}, \ t\in S$, and the obvious Real $S$-action. It is easy to see that the canonical projections define a Real generalized morphism $(Z,\nu):(\cG,\rho)\To (\bfS,\ {}^-)$. One can check that if $[c]=[c']$ then $(Z,\tau)\cong (Z',\tau')$ by working backwards.
\end{proof}

\begin{rem}~\label{rem:non-abelian}
Suppose that $(\bfS,\sigma)$ is a non-abelian Real group. Then we still can talk about \v{C}ech Real $1$-cocycles on $(\cG_\bullet,\rho_\bullet)$ with coefficients on the non-Abelian Real sheaf $(\cS^\bullet,\sigma^\bullet)$, and then form in the same way $\check{H}R^1(\cG_\bullet,\cS^\bullet)$ as a set. However, there is no reason for $\check{H}R^1(\cG_\bullet,\bfS)$ to be an Abelian group, it is not even a group since the sum of a Real $1$-cocycle is not necessarily a Real $1$-cocycle. Nevertheless, the result above remains valid in the sense that there is a bijection between the set $\Hom_{\RG}(\cG,\bfS)$ of isomorphism classes of generalized Real morphism $(\cG,\rho)\To (\bfS,\sigma)$ and the set $\check{H}R^1(\cG_\bullet,\bfS)$.
\end{rem}

A particular example of Proposition~\ref{pro:HR1-H1} is when $S=\uc$ together with the complex conjugation as Real structure; in this case, the associated Real sheaf is denoted by $\uc$ as mentioned earlier. It is well known that the \emph{Picard} group $\operatorname{Pic}(X)$ of a locally compact topological space $X$ is isomorphic to the $1^{st}$ sheaf cohomology group $H^1(X,\underline{\uc}_X)$ (see for instance~\cite[chap.2]{Bry}). In the Real case, we shall introduce the Real Picard group $\PicR(\cG)$ of a Real groupoid, and we will apply Proposition~\ref{pro:HR1-H1} to get an analogous result.

\begin{df}[Real line $\cG$-bundle]
\begin{enumerate}
\item By a \emph{Real line $\cG$-bundle} we mean a Real $\cG$-space $(\cL,\nu)$, and a continuous surjective Real map $\pi: (\cL,\nu)\To (\Gpdo,\rho)$ such that $\pi: \cL\To \Gpdo$ is a  complex vector bundle of rank $1$, and such that for every $x\in \Gpdo$, the induced isomorphism $\nu_x: \cL_x\To \cL_{\rho(x)}$ is $\CC$-anti-linear in the sense that $\nu_x(v\cdot z)=\nu_x(v)\cdot\bar{z}$.\\

\item A homomorphism from a Real line $\cG$-bundle $(\cL,\nu)$  to a Real line $\cG$-bundle $(\cL',\nu')$ is a homormophism of complex vector bundles $\phi: \cL\To \cL'$ intertwining the Real structures and which is $\cG$-equivariant; \emph{i.e.} $\phi(g\cdot v)=g\cdot\phi(v)$ for any $(g,v)\in \cG\ltimes \cL$.\\

\item We say that a Real line $\cG$-bundle $(\cL,\nu)$ is \emph{locally trivial} if there exists a Real open cover $\cU$ of $(\Gpdo,\rho)$, and a family of isomorphisms of complex vector bundles $\vp_j: U_j\times \CC \To \cL_{|U_j}$ such that 

\begin{itemize}
	\item $\vp_{\bar{j}}(\rho(x),\bar{z})=\nu_{_{U_j}}(\vp_j(x,z))$ for all $x\in U_j$ and $(x,z)\in U_j\times \CC$,
	\item if $r(g)\in U_{j_0}$ and $s(g)\in U_{j_1}$, then one has $g.\vp_{j_1}(s(g),z)=\vp_{j_0}(r(g),z)$.
\end{itemize}
	
\end{enumerate}
\end{df}

\begin{ex}
The trivial action $\cG$ on $\Gpdo\times \CC$ (\emph{i.e.} $g\cdot(s(g),z):=(r(g),z)$) is Real; moreover, the canonical projection $\Gpdo\times \CC\To \Gpdo$ defines a Real line $\cG$-bundle that we call \emph{trivial}.
\end{ex}

\begin{df}[Real hermitian $\cG$-metric]
Let $(\cL,\nu)$ be a locally trivial Real line $\cG$-bundle. A \emph{Real hermitian $\cG$-metric} on $(\cL,\nu)$ is a continuous function $\textrm{h}: \cL \To \RR_+ $ such that
\begin{itemize}
\item $\textrm{h}(\nu(v))=\textrm{h}(v)$, and $\textrm{h}(v\cdot z)=\textrm{h}(v)\cdot|z|^2$, for all $v\in \cL, \ z\in \CC$;
\item $\textrm{h}(g\cdot v)=\textrm{h}(v)$, for all $(g,v)\in \cG\ltimes \cL$, and
\item $\textrm{h}(v) > 0$ whenever $v\in \cL^+:=\cL \smallsetminus \oldstylenums{0}$, where $\oldstylenums{0} : \Gpdo \hookrightarrow \cL$ is the zero-section. 
\end{itemize}
If such $\textrm{h}$ exists, $(\cL,\nu,\textrm{h})$ is called a \emph{hermitian Real line $\cG$-bundle} (we will often omit the metric).
\end{df}

\begin{df}[The Real Picard group]
The \emph{Real Picard group} of $(\cG,\rho)$ is defined as the set of isomorphism classes of locally trivial hermitian Real line $\cG$-bundles. This "group" is denoted by $\PicR(\cG)$.
\end{df}

\begin{thm}~\label{thm:Rpic}(compare with~\cite[Theorem 2.1.8]{Bry}).
Let $(\cG,\rho)$ be a locally compact Hausdorff Real groupoid. Then $\PicR(\cG)$ is an Abelian group. Furthermore,
\[\PicR(\cG)\cong \check{H}R^1(\cG_\bullet,\uc).\] 
\end{thm}

\begin{proof}
Associated to any hermitian Real line $\cG$-bundle $\pi:(\cL,\nu)\To (\Gpdo,\rho)$, there is a Real generalized morphism $(\cL^1,\nu):(\cG,\rho)\To (\uc,\ {}^-)$ obtained by setting
\begin{equation}
\cL^1:=\left\{v\in \cL \ | \ \textrm{h}(v)=1 \right\}.
\end{equation}
$\pi: (\cL^1,\nu)\To (\Gpdo,\rho)$ is indeed an $\uc$-principal Real bundle, and $\cL^1$ is invariant under the action of $\cG$. Hence $(\cL^1,\nu)$ is indeed a Real generalized morphism. Conversely, if $(\tilde{\cL},\tilde{\nu}):(\cG,\rho)\To (\uc,\ {}^-)$ is a Real generalized morphism, define $\cL:=\tilde{\cL}\times_{\uc}\CC$, where $\uc$ acts by multiplication on $\CC$; $\nu(v,z):=(\tilde{\nu}(v),\bar{z})$, $g\cdot (v,z):=(g\cdot v,z)$ for $(g,v)\in \cG\ltimes \tilde{\cL}$, and $\textrm{h}(v,z):=|z|^2$. Then $(\cL,\nu,\textrm{h})$ is a hermitian Real line $\cG$-bundle. Moreover, it is not hard to check that if $(\cL,\nu,\textrm{h})$ and $(\cL',\nu',\textrm{h}')$ are isomorphic hermitian Real line $\cG$-bundles, then their associated Real generalized homomorphisms $(\cL^1,\nu)$ and $((\cL')^1,\nu')$ are isomorphic. We then have a map
\begin{equation}
\PicR(\cG)\To H^1(\cG,\uc)_{\rho}, \ [(\cL,\nu,\textrm{h})]\mto [\cL^1,\nu]
\end{equation}
which is clearly an isomorphism of Abelian groups. Now, applying Proposition~\ref{pro:HR1-H1}, we get the desired result. 
\end{proof}


\subsection{$\check{H}R^2$ and Ungraded Real extensions}

Let us consider the subgroup $\Rext^+(\Ga,\bfS)$ of ungraded Real $\bfS$-twists of the Real groupoid $\Ga$; \emph{i.e.} $(\wGa,\del)\in \Rext^+(\Ga,\bfS)$ if $\del=0$. Similarly, we define the subgroup $\wRExt^+(\cG,S)$ of $\wRExt(\cG,\bfS)$ of ungraded Real $\bfS$-central extensions over $\cG$. Elements of $\wRExt^+(\cG,\bfS)$ will then be denoted by pairs of the form $(\wGa,\Ga)$.

Let $\cT=\xymatrix{\bfS \ar[r] & \tilde{\cG}\ar[r]^{\pi} & \cG[\cU_0]}\in \Rext^+(\cG[\cU_0],\bfS)$ be an ungraded Real $S$-twist, for a fixed Real open cover $\cU_0=(U^0_j)_{j\in J_0}$. Consider again the pre-simplicial Real open cover $\cU_\bullet$ of $(\cG_\bullet,\rho_\bullet)$ defined by~\eqref{eq:G-cover}. Recall that the groupoid $\cG[\cU_0]$ is defined by $$\cG[\cU_0]=\left\{(j_0,g,j_1)\in J_0\times \cG\times J_0 \ | \ g\in U^1_{(j_0,j_1)}\right\}.$$
 
 Suppose that the $\bfS$-principal Real bundle $\pi: (\tilde{\cG},\tilde{\rho})\To (\cG[\cU_0],\rho)$ admits a Real family of local continuous sections $\mathsf{s}_{j_0j_1}$ relative to the Real open cover $\cV_1$ of $(\cG[\cU_0],\rho)$ given by $\cV_1=(V^1_{(j_0,j_1)})_{(j_0,j_1)\in J_1}$, where $$V^1_{(j_0,j_1)}:=\{j_0\}\times U^1_{(j_0,j_1)}\times \{j_1\}.$$ 
Then, for any $(g_1,g_2)\in U^2_{(j_0,j_1,j_2)}$, we have that
\begin{align*}
\pi(\mathsf{s}_{j_0j_1}(j_0,g_1,j_1)\cdot\mathsf{s}_{j_1j_2}(j_1,g_2,j_2)) & = \pi(\mathsf{s}_{j_0j_1}(j_0,g_1,j_1))\cdot\pi(\mathsf{s}_{j_1j_2}(j_1,g_2,j_2)) \\ & =(j_0,g_1g_2,j_2)=\pi(\mathsf{s}_{j_0j_2}(j_0,g_1g_2,j_2));
\end{align*} 
thus, there exists a unique element $\omega_{(j_0,j_1,j_2)}(g_1,g_2)\in S$ such that 
\begin{equation}~\label{eq:2-cocyc}
\mathsf{s}_{j_0j_2}(j_0,g_1g_2,j_2)=\omega_{(j_0,j_1,j_2)}(g_1,g_2)\cdot\mathsf{s}_{j_0j_1}(j_0,g_1,j_1).\mathsf{s}_{j_1j_2}(j_1,g_2,j_2).
\end{equation}
This provides a family of continuous functions $\omega_{(j_0,j_1,j_2)}:U^2_{(j_0,j_1,j_2)}\To \bfS$ determined by~\eqref{eq:2-cocyc} and that verifies clearly $\omega_{(\bar{j}_0,\bar{j}_1,\bar{j}_2)}(\rho(g_1),\rho(g_2))=\overline{\omega_{(j_0,j_1,j_2)}(g_1,g_2)},\forall (g_1,g_2)\in U^2_{(j_0,j_1,j_2)}\subset \cG_2$. It is straightforward that the family $(\omega_{(j_0,j_1,j_2)})$ verifies the cocycle condition; hence we obtain a Real \v{C}ech $2$-cocycle 
\begin{equation}~\label{eq:df-omega-T}
\omega(\cT):= (\omega_{(j_0,j_1,j_2)})_{(j_0,j_1,j_2)\in J_2}\in ZR^2(\cU_\bullet,\bfS)
\end{equation}
associated to $\cT$.

 In fact, this construction generalizes for arbitrary Real open cover $\cU_\bullet$ of $(\cG_\bullet,\rho_\bullet)$.

\begin{lem}[Compare Proposition 5.6 in~\cite{Tu1}]
Let $(\cG,\rho)$ be a topological Real groupoid. Given a Real open cover $\cU_\bullet$ of $(\cG_\bullet,\rho_\bullet)$, let $\Rext_{\cU}^+(\cG[\cU_0],\bfS)$ denote the subgroup of all twists $\xymatrix{\bfS \ar[r] & \tilde{\cG} \ar[r]^{\pi} & \cG[\cU_0]}\in \Rext^+(\cG[\cU_0],\bfS)$ such that $\pi$ admits a Real family of local continuous sections \ $\mathsf{s}_{\lambda}: \{\lambda_0\}\times U_{\lambda}\times \{\lambda_1\}\To \tilde{\cG}$ relative to the Real open cover $$\cV_1:=(\{\lambda_0\}\times U^1_{(\lambda_0,\lambda_1,\lambda_{01})}\times \{\lambda_1\})_{(\lambda_0,\lambda_1,\lambda_{01})\in \Lam_1}$$ of $(\cG[\cU_0],\rho)$. Then the canonical map 
\begin{equation}~\label{eq1:RExt+-vs-H2}
\Rext_{\cU}^+(\cG[\cU_0],\bfS)\To HR^2(\cU_\bullet,\bfS), \ [\cT]\mto [\omega(\cT)],
\end{equation}
 is a group isomorphism.
\end{lem}

\begin{proof}
First of all, we shall prove that $\Rext_\cU^+(\cG[\cU_0],\bfS)$ is a subgroup of $\Rext^+(\cG[\cU_0],\bfS)$. Let $$\cT=(\xymatrix{\bfS \ar[r] & \tilde{\cG} \ar[r]^\pi & \cG[\cU_0]}), \ \cT'=(\xymatrix{S \ar[r] & \tilde{\cG}' \ar[r]^{\pi'} & \cG[\cU_0]})$$ be representatives in $\Rext_\cU^+(\cG[\cU_0],\bfS)$. Then their tensor product (cf.~\eqref{eq:Baer}) is 
$$\cT\hat{\otimes}\cT':=(\xymatrix{\bfS \ar[r] & \tilde{\cG}\hat{\otimes}\tilde{\cG}' \ar[r]^{\pi} & \cG[\cU_0]},0),$$ 
where $\tilde{\cG}\hat{\otimes}\tilde{\cG}'=\tilde{\cG} \times_{\cG[\cU_0]}\tilde{\cG}'/\bfS$. Let $f_\lambda:\{\lambda_0\}\times U^1_\lambda \times \{\lambda_1\}\To \tilde{\cG}$ and $f'_\lambda:\{\lambda_0\} \times U^1_\lambda \times \{ \lambda_1\}\To \tilde{\cG}'$ be Real families of continuous local sections of $\pi$ and $\pi'$ respectively. Then we get a Real family of continuous local sections $\mathsf{s}_\lambda:\{\lambda_0\}\times U^1_\lambda \times \{ \lambda_1\} \To \tilde{\cG}\hat{\otimes}\tilde{\cG}'$ for $\pi$ by setting $$\mathsf{s}_\lambda(\lambda_0,g,\lambda_1):=\left[(f_\lambda(\lambda_0,g,\lambda_1),f'_\lambda(\lambda_0,g,\lambda_1))\right],$$
which implies that $\cT\hat{\otimes}\cT'\in \Rext_\cU^+(\cG[\cU_0],\bfS$.

Now let $\cT$ be an (ungraded) Real twist of $(\cG[\cU_0],\rho)$ such that $\pi$ verifies the condition of the lemma. Assume that $\cT'$ is any Real twist of $(\cG[\cU_0],\rho)$ isomorphic to $\cT$. Let $f:\tilde{\cG}\To \tilde{\cG}'$ be a Real $\bfS$-equivariant isomorphism that makes the following diagram 
\begin{equation}~\label{eq:diag1}
\xymatrix{\tilde{\cG}\ar[r]^\pi \ar[d]^f & \cG[\cU_0] \\ \tilde{\cG}' \ar[ur]^{\pi'} & }
\end{equation}
commute. 
Thus, given a Real family $\mathsf{s}_\lambda:\{\lambda_0\}\times U^1_\lambda \times \{\lambda_1\}\To \tilde{\cG}$, the maps $f\circ \mathsf{s}_\lambda:\{\lambda_0\}\times U^1_\lambda \times \{\lambda_1\}\To \tilde{\cG}'$ define a Real family of local continuous sections for $\pi'$; hence the class $[\cT]\in \Rext_\cU^+(\cG[\cU_0],\uc)$.\\

Suppose we are given a representative $$\cT=\xymatrix{\bfS \ar[r] & \tilde{\cG} \ar[r]^\pi & \cG[\cU_0] }$$
in $\Rext_{\cU}^+(\cG[\cU_0],\bfS)$. Recall that for $(\lambda_0,\lambda_1,\lambda_{01})\in \Lam_1$, $U^1_{\lambda_0\lambda_1\lambda_{01}}=U^1_{\lambda_{01}}\cap r^{-1}(U^0_{\lambda_0})\cap s^{-1}(U^0_{\lambda_1})$, and for any $\lambda=(\lambda_0,\lambda_1,\lambda_2\lambda_{01},\lambda_{02},\lambda_{12},\lambda_{012})\in \Lam_2$, we have from~\eqref{eq:U^2_lambda} that \[U^2_{\lambda}=\tilde{\ve}_1^{-1}\circ r^{-1}(U^0_{\lambda_0})\cap \tilde{\ve}_0^{-1}\circ s^{-1}(U^0_{\lambda_1})\cap \tilde{\ve}_1^{-1}\circ s^{-1}(U^0_{\lambda_2})\cap \tilde{\ve}_2^{-1}(U^1_{\lambda_{01}})\cap \tilde{\ve}^{-1}(U^1_{\lambda_{02}})\cap \tilde{\ve}_0^{-1}(U^1_{\lambda_{12}})\cap U^2_{\lambda_{012}}.\]
Then, for all $(g_1,g_2)\in U^2_\lambda$, one has 
\begin{itemize}
\item $g_1g_2=\tilde{\ve}_1(g_1,g_2)\in r^{-1}(U^0_{\lambda_0})\cap s^{-1}(U^0_{\lambda_2})\cap U^1_{\lambda_{02}} =U^1_{\lambda_0\lambda_2\lambda_{02}}$, 
\item $g_1=\tilde{\ve}_2(g_1,g_2)\in U^1_{\lambda_{01}}$, $g_2=\tilde{\ve}_0(g_1,g_2)\in s^{-1}(U^0_{\lambda_1})\cap U^1_{\lambda_{12}}$; and hence $$g_1\in r^{-1}(U^0_{\lambda_0})\cap s^{-1}(U^0_{\lambda_1})\cap U^1_{\lambda_{01}}=U^1_{\lambda_0\lambda_1\lambda_{01}},\ \text{and} $$ $$  g_2\in r^{-1}(U^0_{\lambda_1})\cap s^{-1}(U^0_{\lambda_2})\cap U^1_{\lambda_{12}}=U^1_{\lambda_1\lambda_2\lambda_{12}}.$$ 
\end{itemize}
Then as in the discussion before the lemma (cf.~\eqref{eq:df-omega-T}), there exists a Real family of functions $\omega_{\lambda}: U^2_\lambda \To \uc$ such that
\begin{equation}~\label{eq:omega_twist}
\mathsf{s}_{\lambda_0\lambda_2\lambda_{02}}(\lambda_0,g_1g_2,\lambda_{2})= \omega_\lambda(g_1,g_2)\cdot\mathsf{s}_{\lambda_0\lambda_1\lambda_{01}}(\lambda_0,g_1,\lambda_{1})\cdot\mathsf{s}_{\lambda_1\lambda_2\lambda_{12}}(\lambda_1,g_2,\lambda_{2})
\end{equation}
and $\omega_{\bar{\lambda}}(\rho(g_1),\rho(g_2))=\overline{\omega_\lambda(g_1,g_2)}$, for all $(g_1,g_2)\in U^2_{\lambda_0\lambda_1\lambda_2\lambda_{01}\lambda_{02}\lambda_{12}\lambda_{012}}$. Moreover, it is easy to verify by a routine calculation that $(\omega_\lambda)_{\lambda\in \Lam_2}$ verify the cocycle condition on $$U^3_{\lambda_0\lambda_1\lambda_2\lambda_3\lambda_{01}\lambda_{02}\lambda_{03}\lambda_{12}\lambda_{13}\lambda_{23}\lambda_{0123}}\subset \cG_2;$$ therefore, we have constructed a Real \v{C}ech $2$-cocyle $(\omega_{\lambda})_{\lambda\in \Lam_2}\in ZR^2(\cU_\bullet,\bfS)$ associated to $\cT$. \\

Assume that $(\tilde{\mathsf{s}}_\lambda)_{\lambda \in \Lam_2}$ is another Real family of continuous local sections of $\pi$, and that $(\tilde{\omega}_\lambda)_{\lambda \in \Lam_2}\in ZR^2(\cU_\bullet,\bfS)$ is its associated Real \v{C}ech $2$-cocycle. Then for any $(\lambda_0,\lambda_1,\lambda_{01})\in \Lam_1$ and $g\in U^1_{\lambda_0\lambda_1\lambda_{01}}$, there exists a unique $c_{\lambda_0\lambda_1\lambda_{01}}(g)\in \bfS$ such that 
\begin{equation}~\label{eq:1-cocycl-twist}
\tilde{\mathsf{s}}_{\lambda_0\lambda_1\lambda_{01}}(g)=c_{\lambda_0\lambda_1\lambda_{01}}(g)\cdot\mathsf{s}_{\lambda_0\lambda_1\lambda_{01}}(g),
\end{equation}
 where we abusively write, for instance, $\mathsf{s}_{\lambda_0\lambda_1\lambda_{01}}(g)$ for $\mathsf{s}_{\lambda_0\lambda_1\lambda_{01}}(\lambda_0,g,\lambda_{1})$. Since $(\tilde{\mathsf{s}}_{\lambda_0\lambda_1\lambda_{01}})$ and $\mathsf{s}_{\lambda_0\lambda_1\lambda_{01}}$ are Real families, we have that $$c_{\bar{\lambda}_0\bar{\lambda}_1\bar{\lambda}_{01}}(\rho(g))=\overline{c_{\lambda_0\lambda_1\lambda_{01}}(g)} \  \text{for all} \  g\in U^1_{\lambda_0\lambda_1\lambda_{01}}.$$ It turns out that the $c_{\lambda_0\lambda_1\lambda_{01}}$'s define an element in $CR^1(\cU_\bullet,\bfS)$. Moreover, for $\lambda \in \Lam_2$ as previously, and for $(g_1,g_2)\in U^2_\lambda$, we obtain from~\eqref{eq:omega_twist} and~\eqref{eq:1-cocycl-twist}
\[
\mathsf{s}_{_{\lambda_0\lambda_2\lambda_{02}}}(g_1g_2) = c_{_{\lambda_0\lambda_2\lambda_{02}}}(g_1g_2)^{-1}\cdot c_{_{\lambda_0\lambda_1\lambda_{01}}}(g_1)\cdot c_{_{\lambda_1\lambda_2\lambda_{12}}}(g_2)\cdot\tilde{\omega}_\lambda (g_1,g_2)\cdot\mathsf{s}_{_{\lambda_0\lambda_1\lambda_{01}}}(g_1)\cdot\mathsf{s}_{_{\lambda_1\lambda_2\lambda_{12}}}(g_2);
\]
and \[(\omega_\lambda\cdot\tilde{\omega}_\lambda^{-1})(g_1,g_2)=c_{\lambda_0\lambda_2\lambda_{02}}(g_1g_2)^{-1}.c_{\lambda_0\lambda_1\lambda_{01}}(g_1)\cdot c_{\lambda_1\lambda_2\lambda_{12}}(g_2) = (dc)_\lambda(g_1,g_2);
\]
hence $((\omega\cdot\tilde{\omega}^{-1})_\lambda)_{\lambda\in \Lam_2}\in BR^2(\cU_\bullet,\uc)$. In other words, the class in $HR^2(\cU_\bullet,\bfS)$ of the Real $2$-cocycle $(\omega_\lambda)$ does not depend on the choice of the Real family of local sections of $\pi$.\

We want now to check that the map~\eqref{eq1:RExt+-vs-H2} is well defined. To do so, suppose that $\cT$ and $\cT'$ are equivalent in $\Rext_\cU(\cG[\cU_0],\bfS)$, and that $(\mathsf{s}_{\lambda_0\lambda_1\lambda_{01}})$ and $\mathsf{s}'_{\lambda_0\lambda_1\lambda_{01}}$ are Real family of local continuous sections of $\pi$ and $\pi'$. Let us keep the diagram~\eqref{eq:diag1}. Let $(\omega_\lambda)_{\lambda\in \Lam_2}$ and $(\omega'_\lambda)_{\lambda\in \Lam_2}$ be the associated Real $2$-cocycles in $ZR^2(\cU_\bullet,\bfS)$ of $\cT$ and $\cT'$  respectively. Then we define an element $(b_{\lambda_0\lambda_1\lambda_{01}})\in CR^1(\cU_\bullet,\bfS)$ as follows: for any $g\in U^1_{\lambda_0\lambda_1\lambda_{01}}$, $b_{\lambda_0\lambda_1\lambda_{01}}(g)$ is the unique element of $\bfS$ such that 
\begin{equation}
\mathsf{s}'_{\lambda_0\lambda_1\lambda_{01}}(g)=b_{\lambda_0\lambda_1\lambda_{01}}(g)\cdot f\circ \mathsf{s}_{\lambda_0\lambda_1\lambda_{01}}(g).
\end{equation}
This is well defined since $\pi'(\mathsf{s}'_{\lambda_0\lambda_1\lambda_{01}}(g))=\pi(\mathsf{s}_{\lambda_0\lambda_1\lambda_{01}}(g))=\pi'(f\circ \mathsf{s}_{\lambda_0\lambda_1\lambda_{01}}(g))$. Furthermore, the functions $f\circ \mathsf{s}_{\lambda_0\lambda_1\lambda_{01}}, (\lambda_0,\lambda_1,\lambda_{01})\in \Lam_1$,  defines a globally Real family of local continuous sections of $\pi$. Then, for all $\lambda\in \Lam_2$ and all $(g_1,g_2)\in U^2_{\lambda}$, we can write \[f\circ \mathsf{s}_{\lambda_0\lambda_2\lambda_{02}}(g_1g_2)=\omega_\lambda(g_1,g_2)\cdot f\circ \mathsf{s}_{\lambda_0\lambda_1\lambda_{01}}(g_1)\cdot f\circ \mathsf{s}_{\lambda_1\lambda_2\lambda_{12}}(g_2),\] up to a multiplication of $\omega_\lambda$ by a Real $2$-coboundary. It then follows that 
\[\omega_\lambda(g_1,g_2)\cdot\omega'_{\lambda}(g_1,g_2)^{-1}=b_{\lambda_0\lambda_2\lambda_{02}}(g_1g_2)^{-1}\cdot b_{\lambda_0\lambda_1\lambda_{01}}(g_1)\cdot b_{\lambda_1\lambda_2\lambda_{12}}(g_2)=(db)_{\lambda}(g_1,g_2).\] 
Consequently, $(\omega_\lambda)_{\lambda\in \Lam_2}$ depends only on the class of $\cT$ in $\Rext_\cU(\cG[\cU_0],\bfS)$. The fact that $(\del_{\lambda_0\lambda_1\lambda_{01}})$ also depends only on the class of $\cT$ is straightforward. We then have proved that any element $[\cT]$ in $\Rext_\cU(\cG[\cU_0],\bfS)$ determines a unique cohomology class
\begin{equation}
[\omega(\cT)]\in HR^2(\cU_\bullet,\bfS).
\end{equation}

Conversely, given a pair $(\omega_\lambda)_{\lambda\in \Lam_2}\in ZR^2(\cU_\bullet,\bfS)$, we want to construct an ungraded Real extension of $(\cG[\cU_0],\rho)$ which is in $\Rext_\cU^+(\cG[\cU_0],\bfS)$. For this we proceed as in the proof of Proposition 5.6 in~\cite{Tu1}. For $\lambda\in \Lam_2$, put 
\begin{eqnarray*}
	\mu_{01}:= (\lambda_0,\lambda_{01},\lambda_1),\\
	\mu_{02}:= (\lambda_0,\lambda_{02},\lambda_2),\\
	\mu_{12}:= (\lambda_1,\lambda_{12},\lambda_2).
\end{eqnarray*}
Let $\frc_{\mu_{01}\mu_{02}\mu_{12}}:=\omega_\lambda$. We have $\cV_1=(V^1_{\mu_{01}})_{i\in I_1}$, where $I_1$ consists of triples $\mu_{01}=(\lambda_0,\lambda_{01},\lambda_1)$ and $V^1_{\mu_{01}}:=\{\lambda_0\}\times U^1_{\lambda_0\lambda_1\lambda_{01}}\times\{\lambda_1\}$. $I_1$ is equipped with the obvious involution, so that $\cV_1$ is a Real open cover of $\cG[\cU_0]$. We set 
\[\wGa^\omega:= \coprod_{\mu_{01}\in I_1}\{(t,g,\mu_{01})\mid t\in \bfS, g\in V^1_{\mu_{01}}\} /\sim,\]
subject to the product law
\[
[t_1,g_1,\mu_{01}]\cdot [t_1,g_2,\mu_{12}] = [t_1\cdot t_2\cdot \frc_{\mu_{01}\mu_{02}\mu_{12}}(g_1,g_2),g_1g_2,\mu_{02}],
\]

where 

\begin{equation}
 (t,g,\mu_{12}) \sim (\frc_{\mu_{01}\mu_{01}\mu_{01}}(r(g),r(g))^{-1}\cdot t\cdot\frc_{\mu_{01}\mu_{02}\mu_{12}}(r(g),g),g,\mu_{02}).	
 \end{equation}
 The projection $\pi:\wGa^\omega\To \cG[\cU_0]$ is defined by $\pi([t,g,\mu_{01}]):=g$, and the Real structure is 
 $$\overline{[t,g,\mu_{01}]}:=[\bar{t},\rho(g),\overline{\mu_{01}}].$$ 
 It is straightforward to see that these operations give $\wGa^\omega$ the structure of ungraded Real $\bfS$-twist of $\cG[\cU_0]$; what is more, the maps $\mathsf{s}_{\mu_{01}}:V^1_{\mu_{01}}\To \wGa^\omega$ defined by $\mathsf{s}_{\mu_{01}}(g):=[0,g,\mu_{01}]$ are a Real family of continuous sections of $\pi$, so that the Real extension 
 \[\cT=\xymatrix{\bfS\ar[r] & \wGa^\omega \ar[r]^\pi& \cG[\cU_0]}\]
 is in $\Rext_{\cU}^+(\cG[\cU_0],\bfS)$. It is also clear that $[\omega(\cT)]=[\omega]$.
\end{proof}

\begin{cor}
We have $\wRExt^+(\cG,\bfS)\cong \check{H}R^2(\cG_\bullet,\bfS)$.	
\end{cor}


\subsection{The cup-product $\check{H}R^1(\cdot,\ZZ_2)\times \check{H}R^1(\cdot,\ZZ_2)\to \check{H}R^2(\cdot,\uc)$}

Let $\del,\del'\in \check{H}R^1(\cG_\bullet,\ZZ_2)$, and let $L$ and $L'$ be representatives of their corresponding classes in $\Hom_{\RG}(\cG,\ZZ_2)$ (cf. Proposition~\ref{pro:HR1-H1}). Then by viewing $\ZZ_2=\{\mp1\}$ as a Real subgroup of $\uc$ (identifying $-1$ with $(-1,0)$ and $+1$ with $(1,0)$),  we define the tensor product $r^\ast L\otimes \overline{s^\ast L'} \To \cG$, and and using the same reasoning as in Example~\ref{ex:triv-twist2}, we see that this is clearly a Real $\ZZ_2$-principal bundle; thus we have an ungraded Real $\ZZ_2$-central extension 

\[\ZZ_2\To r^\ast L\otimes \overline{s^\ast L'} \To \cG .\] 

Therefore, we get an ungraded Real $\uc$-central extension $(L\smile L',\cG)$ given by
 
\begin{equation}~\label{eq:df-cup-prod-HR}
L\smile L':= (r^\ast L\otimes \overline{s^\ast L'})\times_{\ZZ_2}\uc,
\end{equation}
together with the evident Real structure and Real $\uc$-action.

\begin{df}
We define the cup product 
\[\smile \ \colon \check{H}R^1(\cG_\bullet,\ZZ_2)\times \check{H}R^1(\cG_\bullet,\ZZ_2) \To \check{H}R^2(\cG_\bullet,\uc)\]
by

\[\del\smile \del':= \omega(L\smile L'),\]
where $L\smile L'$ is determined by equation~\eqref{eq:df-cup-prod-HR}.	
\end{df}

\begin{lem}
The cup product $\smile$ defined above is a well defined bilinear map; \emph{i.e.} 
\[(\del_1+\del_2)\smile(\del_1'+\del_2')=\del_1\smile \del_1'+\del_1\smile\del_2' +\del_2\smile\del_1'+\del_2\smile\del_2'.\] 	
\end{lem}

\begin{proof}
If $\del_i$ is realized by the generalized Real homomorphism $L_i:\cG\To \ZZ_2$, then $\del_1+\del_2$ is realized by $L_1+L_2$. The result follows from the easy to check bilinearity of the tensor product $r^\ast L\otimes \overline{s^\ast L'}$ with respect to the sum in $\Hom_{\RG}(\cG,\ZZ_2)$.
\end{proof}


\subsection{Cohomological picture of the group $\wRExt(\cG,\uc)$}~\label{sect:HR2-HR1-vs-Ext}

Let $\cT=(\tilde{\cG},\del)\in \Rext(\cG[\cU_0],\uc)$, where as usual $\cU_0$ is a Real open cover of $\Gpdo$. Let $\cU_\bullet$ be the pre-simplicial Real open cover of $(\cG_\bullet,\rho_\bullet)$ defined as in~\eqref{eq:G-cover}.

 Define a continuous map $\del_{j_0j_1}: U^1_{(j_0,j_1)}\To \ZZ_2$ over all $U^1_{(j_0,j_1)}\in \cU_1$ by $\del_{j_0j_1}(g):=\del(j_0,g,j_1)$. Then, over all $U^2_{(j_0,j_1,j_2)}$, we have that $\del_{j_0j_2}(g_1g_2)=\del((j_0,g_1,j_1)\cdot(j_1,g_2,j_2))=\del_{j_0j_1}(g_1)\cdot\del_{j_1j_2}(g_2)$. Moreover, since $\del$ is a Real morphism, we have that $\del_{\bar{j}_0\bar{j}_1}(\rho(g))=\del_{j_0j_1}(g)$; hence $\cT$ determines a Real \v{C}ech $1$-cocycle 
\begin{equation}~\label{eq:del_twist}
\del(\cT):=(\del_{j_0j_1})_{(j_0,j_1)\in J_1}\in ZR^1(\cU_\bullet,\ZZ_2),
\end{equation}
Then,~\eqref{eq:del_twist} gives a Real \v{C}ech $1$-cocycle $(\del_{\lambda_0\lambda_1\lambda_{01}})\in ZR^1(\cU_\bullet,\ZZ_2)$ defined by $\del_{\lambda_0\lambda_1\lambda_{01}}(g):=\del(\lambda_0,g,\lambda_{1})$ for any $g\in U^1_{\lambda_0\lambda_1\lambda_{01}}$; this  does make sense, for we know from Section~\ref{section:HR1} that Real \v{C}ech $1$-cocycles do not depend on $\lambda_{01}$. 

If $\cT'$ is another Rg $\uc$-central extension over $\cG$, we may suppose it is represented by a Rg $\uc$-twisted $(\tilde{\cG}',\del')$ of $\cG[\cU_0]$. Then by definition of the grading of $\cT\hat{\otimes}\cT'$, we have $\del(\cT\hat{\otimes}\cT')=\del(\cT)+\del(\cT')$.

\begin{thm}[Compare Proposition 2.13~\cite{FHT}]~\label{thm:ExtR-vs-Coho}
Let $(\cG,\rho)$ be a locally compact Hausdorff Real groupoid.
There is a set-theoretic split-exact sequence 
\begin{equation}~\label{eq:seq-Ext+-vs-Ext-HR1}
	0\To\check{H}R^2(\cG_\bullet,\uc) \hookrightarrow \wRExt(\cG,\uc) \stackrel{\del}{\To} \check{H}R^1(\cG_\bullet,\ZZ_2) \To 0
\end{equation}
so that we have a canonical group isomorphism
\begin{equation}
dd\colon \wRExt(\cG,\uc)\cong \check{H}R^1(\cG_\bullet,\ZZ_2)\ltimes \check{H}R^2(\cG_\bullet,\cS^1),
\end{equation}
where the semi-direct product $\check{H}R^1(\cG_\bullet,\ZZ_2)\ltimes \check{H}R^2(\cG_\bullet,\uc)$ is defined by the operation 
\[(\del,\omega)+(\del',\omega'):=(\del+\del',(\del\smile \del')\cdot \omega\cdot\omega').\]
The image of a Real graded extension $\EE$ by $dd$ is called the \emph{Dixmier-Douady class} of \ $\EE$.
\end{thm}

\begin{proof}
The first arrow is the canonical inclusion $\wRExt^+(\cG,\uc)\subset \wRExt(\cG,\uc)$, and hence is injective. The exactness of the sequence~\eqref{eq:seq-Ext+-vs-Ext-HR1} is obvious, by definition of $\delta$ and $\wRExt^+(\cG,\uc)$. 

The map $\delta$ is well defined; indeed, if $\cT\sim \cT'$ in $\Rext(\cG[\cU_0],\uc)$, they differ from a twist coming from an element of $\PicR(\cG[\cU_0])$, and hence by construction of $\delta$, one has $\delta(\cT)=\delta(\cT')$. Moreover, $\delta$ is surjective, for if $L\in \Hom_{\RG}(\cG,\ZZ_2)$ represents the Real $1$-cocycle $(\ve_{j_0j_1}) \in ZR^1(\cU_\bullet,\ZZ_2)$, then $L\smile L$ is graded as follows:
$$L\smile L:=(\uc \To (r^\ast L\otimes \overline{s^\ast L})\times_{\ZZ_2}\uc\To \cG[\cU_0],\del'),$$
 where 
\[
\del'((j_0,\g,j_1)):=\ve_{j_0j_1}(\g).
\]
We see that $\del(L\smile L)=\ve$. 
Finally, note that the operation law comes from the definition of the sum in $\wRExt(\cG,\uc)$.
\end{proof}


\subsection{The proper case}

In this subsection, we are interested in some particular Abelian Real sheaves on $(\cG_\bullet,\rho_\bullet)$, where $(\cG,\rho)$ is a proper groupoid. More precisely, we aim to generalize a result by Crainic (see~\cite[Proposition 1]{Crai}) stating that for a proper Lie groupoid $\cG$, and "\emph{representation}" $E$ of $\cG$ (~\cite[1.2]{Crai}), the \emph{differentiable} cohomology $H_d^n(\cG,E)=0$ for all $n\ge 1$. Let us first introduce some few notions and properties.

\begin{df}[Real Haar measure]
Let $(\cG,\rho)$ be a locally compact Real groupoid, and let $\{\mu^x\}_{x\in \Gpdo}$ be a (left) \emph{Haar system} for $\cG$ (cf.~\cite[\S.2]{Ren}). Define a new family  $\{\mu_\rho^x\}_{x\in \Gpdo}$ of measures $\mu_\rho^x$, with support $\cG^x$ for all $x\in \Gpdo$, defined by 
\begin{equation}
\mu_\rho^x(C):=\mu^{\rho(x)}(\rho(C)), \ \text{for all measurable subset} \ C\subset \cG^x.
\end{equation}
We say that $\{\mu^x\}_{x\in \Gpdo}$ is \emph{Real} if 
\begin{equation}
\mu^x=\mu_\rho^x, \ \forall x\in \Gpdo.
\end{equation}
\end{df}

\begin{lem}
Any Haar system for $\cG$ gives rise to a Real one.
\end{lem}

\begin{proof}
Assume $\{\mu^x\}$ is a Haar system for $\cG$. For every $x\in \Gpdo$, we set
\begin{eqnarray}
	\tilde{\mu}^x:=\frac{1}{2}(\mu^x+\mu_\rho^x).	
\end{eqnarray}	
It is clear that $\{\tilde{\mu}^x\}_{x\in \Gpdo}$ is a Haar system for $\cG$; measurable subsets for $\tilde{\mu}^x$ being exactly those for $\mu^x$. Moreover, one has 
\begin{align*}
\tilde{\mu}_\rho^x=\frac{1}{2}\left(\mu^{\rho(x)}\circ \rho +\mu_\rho^{\rho(x)}\circ \rho\right)=\frac{1}{2}\left(\mu_\rho^x+\mu^x\right)=\tilde{\mu}^x, \ \forall x\in \Gpdo.
\end{align*}
\end{proof}

\begin{rem}
From the lemma above, we will always assume Haar systems for $\cG$ to be Real. 
\end{rem}

In what follows, the Real group $\KK$ is either the additive group $\RR$ equipped with the Real structure $t\mto \bar{t}:=-t$, or the additive group $\CC$ equipped with the complex conjugation  $z\mto \bar{z}$ as Real structure.

\begin{df}
Let $(\cG,\rho)$ be a locally compact Real groupoid. A \emph{Real representation} of $(\cG,\rho)$ is a locally trivial Real $\KK$-vector bundle $\pi: (E,\nu)\To (\Gpdo,\rho)$ endowed with a (left) continuous Real $\cG$-action; that is a Real open cover $(U_j)$ of $(\Gpdo,\rho)$  and isomorphisms $\phi_j: U_j\times \KK^r\To E_{|U_j}$ such that $\nu(\phi_j(x,(a_1,\ldots,a_r)))=\phi_{\bar{j}}(\rho(x),(\bar{a}_1,\ldots,\bar{a}_r)), \ \forall x\in U_j, (a_1,...,a_r)\in \KK^r$, and
\begin{itemize}
\item $\forall x\in \Gpdo$, the induced isomorphism $\nu_x: E_x\To E_{\rho(x)}$ is $\KK$-antilinear: 
$$\nu_x(\xi\cdot a)=\nu_x(\xi)\cdot\bar{a}, \ \forall \xi\in E_x, a\in \KK;$$ 
\item $\forall g\in \cG$, the isomorphism $E_{s(g)}\To E_{r(g)}$, induced by the $\cG$-action, is linear. 
\end{itemize}
\end{df}

Note that such a Real representation $(E,\nu)$ can be viewed as a Real $\cG$-module in the following way: $E$ is the groupoid $\xymatrix{E \dar[r] & \Gpdo}$ with $r_E(\xi)=s_E(\xi):=\pi(\xi)$ for every $\xi\in E$, for any $x\in \Gpdo$, $E_x=E^x=E_x^x$ is isomorphic to the group $\KK$, then the product in $E$ is defined by the sum on the fibres. The Real sheaf on $(\cG_\bullet,\rho_\bullet)$ associated to the Real $\cG$-module $(E,\nu)$ will be denoted $(E^\bullet,\nu^\bullet)$.

\begin{rem}~\label{rem:Real-repr-p-q}
More generally, we may define a Real representation of of type $\RR^{p,q}$ as a locally trivial real vector bundle $E\To \Gpdo$ of rank $p+q$, together with a Real structure $\nu:E\To E$, and a Real $\cG$-action on $E$ with respect to the projection map, such that locally, the Real space $(E,\nu)$ identifies with $\RR^{p,q}$; that is there is a Real open cover $(U_j)$ of $\Gpdo$ and commutative diagrams
\[
\xymatrix{U_j\times \RR^{p,q} \ar[r]^{\phi_j} \ar[d]^{\rho\times bar} & E_{|U_j} \ar[d]^{\nu} \\ U_{\bar{j}}\times \RR^{p,q}\ar[r]^{\phi_{\bar{j}}} & E_{|U_{\bar{j}}}}
\]
where $bar:\RR^{p,q}\To \RR^{p,q}$ is the Real structure defined in the first section.
\end{rem}

\begin{df}~\label{df:proper} (~\cite[Definition 2.20]{TLX})
A locally compact Real groupoid $(\cG,\rho)$ is said to be \emph{proper} if any of the following equivalent conditions is satisfied:
\begin{itemize}
\item[(i)] the Real map $(s,r): \cG \To \Gpdo\times \Gpdo$ is proper;
\item[(ii)] for every $K\subset \Gpdo$ compact, $\cG^K_K$ is compact.
\end{itemize}
\end{df}

Proper Real groupoids can be characterized by the following (we refer to Propositions 6.10 and 6.11 in~\cite{Tu2} for a proof) 

\begin{pro}~\label{pro:proper-cutoff}
Let $(\cG,\rho)$ be a locally compact Real groupoid with a Haar system $\{\mu^x\}_{x\in \Gpdo}$. Then $(\cG,\rho)$ is proper if and only it admits a \emph{cutoff} Real function; that is, a function $x: \Gpdo \To \RR_+$ such that
\begin{itemize}
\item[(i)] $\forall x\in \Gpdo$, $c(\rho(x))=c(x)$;
\item[(ii)] $\forall x\in \Gpdo$, $\int_{\cG^x} c(s(g))d\mu^x(g)=1$;
\item[(iii)] the map $r: \text{supp}(\mathsf{c}\circ s)\To \Gpdo$ is proper; \emph{i.e.} for every $K\subset \Gpdo$ compact, $\supp(c)\cap s(\cG^K)$ is compact.
\end{itemize} 	
\end{pro}

\medskip

\begin{thm}~\label{thm:HR_vs_proper}
Suppose $(\cG,\rho)$ is a locally compact proper Real groupoid with a Haar system. Then, for any Real  representation $(E,\nu)$ of $(\cG,\rho)$, we have 
\[\check{H}R^n(\cG_\bullet,E^\bullet)=0, \ \forall n\ge 1.\] 
\end{thm}

\medskip

To prove this result, we shall recall fundamentals of vector-valued integration exposed, for instance, in~\cite[Appendix B.1]{Wil1}, and then adapt them to the case when we deal with Real structures. Let $X$ be a locally compact Hausdorff space, and let $B$ be a separable Banach space. Let $\mu$ be a Radon measure on $X$. Then measurable functions $f:X\To B$ are defined as usual, and such function is \emph{integrable} if \[\|f\|_1:=\int_X\|f(x)\|d\mu(x) < \infty . \]
The collection of all $B$-valued integrable functions on $X$ is denoted by $\cL^1(X,B)$, and the set of equivalence classes of functions in $\cL^1(X,B)$ is a Banach space denoted by $L^1(X,B)$ (~\cite[Proposition B.31]{Wil1}). Furthermore, $\cC_c(X,B)$ is dense in $L^1(X,B)$. The $B$-valued integration of elements of $L^1(X,B)$ is defined as a linear map $I:\cC_c(X,B)\To B$ given by 
\begin{eqnarray}~\label{eq:integral}
I(f):=\int_Xf(x)d\mu(x), \ \text{and} \ \|I(f)\|\leq \|f\|_1.
\end{eqnarray}
Moreover, this integral is characterized by the following

\begin{pro}~\label{pro:integral}(cf. Proposition B.34~\cite{Wil1})
Let $\mu$ be a Radon measure on $X$, and let $B$ be a Banach space. Then, the integral is characterized by 
\begin{itemize}
\item[(a)] for all $f\in \cC_c(X,B)$ and $\vp\in B^\ast$, \[\vp\left(\int_Xf(x)d\mu(x)\right)=\int_X\vp(f(x))d\mu(x);\]
\item[(b)] if $L: B\To B'$ is any bounded linear map between two Banach spaces, than 
\[L\left(\int_Xf(x)d\mu(x)\right)=\int_X L(f(x))d\mu(x).\]
\end{itemize}
\end{pro}

Now suppose $(X,\rho)$ is a locally compact Hausdorff Real space, $\mu$ is a Real Radon measure; \emph{i.e.} $\mu(\rho(C))=\rho(C)$ for every measurable set $C\subset X$. Let $(B,\varsigma)$ be a separable Real Banach space. Then from the above, we deduce the

\begin{lem}~\label{lem:Real_integral}
Let $\cC_c(X,B)$ be equipped with the Real structure denoted by $\tilde{\rho}: \cC_c(X,B)\To \cC_c(X,B)$, and given by $\rho(f)(x):=\varsigma(f(\rho(x)))$. Then, under the above assumption, the integral $\int:\cC_c(X,B)\To B$ is Real, in that it commutes with the Real structures $\varsigma$ and $\tilde{\rho}$; i.e 
\begin{eqnarray}
\int_X \varsigma(f(\rho(x)))d\mu(x)=\varsigma\left(\int_X f(x)d\mu(x)\right), \forall f\in \cC_c(X,B).
\end{eqnarray}
\end{lem}

\begin{proof}
For any $\vp\in B^\ast$, define $\bar{\vp}\in B^\ast$ by $\bar{\vp}(b):=\overline{\vp(\varsigma(b))}$. Then, from Proposition~\ref{pro:integral} (a) and the definition of $\bar{\vp}$, one has 
\begin{align*}
\overline{\vp\left(\varsigma\left(\int_X f(x)d\mu(x)\right)\right)}=\int_X \overline{\vp(\varsigma(f(x)))}d\mu(x)=\overline{\int_X \vp(\varsigma(f(x)))d\mu(x)}.
\end{align*}
Thus,
\begin{align*}
\vp\left(\varsigma\left(\int_Xf(x)d\mu(x)\right)\right)=\int_X\vp(\varsigma(f(x)))d\mu(x).
\end{align*}
Again from (b) of Proposition~\ref{pro:integral} and from the fact that $\mu$ is Real, we then get
\begin{align*}
\vp\left(\varsigma\left(\int_Xf(x)d\mu(x)\right)\right)=\vp\left(\int_X\varsigma(f(\rho(x)))d\mu(x)\right), \forall \vp\in B^\ast,
\end{align*}
and the result holds.
\end{proof}

Let us investigate the case of a Real groupoid $(\cG,\rho)$ together with a Real representation $(E,\nu)$. Let $\mu=\{\mu^x\}_{x\in \Gpdo}$ be a Real Haar system for $(\cG,\rho)$. For any $x\in \Gpdo$, we can apply~\eqref{eq:integral} to $E_x$ and get the integral $\int_{\cG^x}: \cC_c(\cG^x,E_x)\To E_x$. Further, it is very easy to check that 
\begin{eqnarray}~\label{eq:G_integral}
\nu_x\left(\int_{\cG^x}f(\g)d\mu^x(\g)\right)=\int_{\cG^{\rho(x)}}\nu_x(f(\rho(\g)))d\mu^{\rho(x)}(\g), \ \forall f\in \cC_c(\cG^x,E_x).
\end{eqnarray}

\begin{proof}[Proof of Theorem~\ref{thm:HR_vs_proper}]
Fix a Real Haar system $\{\mu^x\}_{x\in \Gpdo}$ for $(\cG,\rho)$ and a cutoff Real function $c:\Gpdo \To \RR_+$. Let $\cU_\bullet$ be a Real open cover of $(\cG_\bullet,\rho_\bullet)$. Let $\lambda:=(\lambda_0,\lambda_1,\ldots,\lambda_{01\ldots n})\in \Lam_n$ and $U^n_\lambda \in {}_\natural \cU_n$. Denote by $\Lam_{n+1|\lambda}$ the subset of $\Lam_{n+1}$ consisting of those $\tilde{\lambda}\in \Lam_{n+1}$ such that $\tilde{\lambda}(S)=\lambda_{S}$ for all $\emptyset \neq S\subseteq [n]$. Then, if for any $x\in U^0_{\lambda_n}$, we denote 
\[(U^n_\lambda \star \cG^x)\cap \supp(\mathsf{c}\circ s):=\{(g_1,\ldots,g_n,\g)\in U^n_\lambda \times (\cG^x\cap \text{supp}(\mathsf{c}\circ s)) \ | \ s(g_n)=r(\g)=x\},\]
we have that 
\begin{eqnarray}~\label{eq:1st-union}
(U^n_\lambda \star \cG^{x})\cap \supp(\mathsf{c}\circ s)\subset \bigcup_{\tilde{\lambda}\in \Lam_{n+1|\lambda}}U^{n+1}_{\tilde{\lambda}}.
\end{eqnarray}
Notice that for $\tilde{\lambda}$ running over $\Lam_{n+1|\lambda}$, only its images $\tilde{\lambda}_S \in \Lam_{\#S-1}$, for $S \subseteq [n+1]$ containing $n+1$, are led to vary. On the other hand, since $\cG^x\cap \supp(\mathsf{c}\circ s)$ is compact in $\cG$ (by (iii) of Proposition~\ref{pro:proper-cutoff}), the union~\eqref{eq:1st-union} is finite. In particular, for every $S\in S(n+1):=\{S\subseteq [n+1] \ | \ n+1\in S\neq \emptyset \}$, where elements of $S(n+1)$ are ranged in cardinality and in lexicographic order, there is $\tilde{\lambda}_S^{l_S}\in \Lam_{\#S-1}$, $l_S=0,\ldots,m_S$, such that 
\begin{eqnarray}
(U^n_\lambda\star \cG^x)\cap \supp(\mathsf{c}\circ s) \subset \bigcup_{l=(l_S)_{S\in S(n+1)}} U^{n+1}_{\lambda^l},
\end{eqnarray}
where for any $l=(l_S)_{S\in S(n+1)}\in \NN^{2^{n+1}}$ written as $$l=\left(l_{\{n+1\}},l_{\{0,n+1\}},l_{\{1,n+1\}},\ldots,l_{\{n,n+1\}},\ldots,l_{\{1,\ldots,n+1\}},l_{\{0,1,\ldots,n+1\}}\right),$$ the element $\lambda^l\in \Lam_{n+1|\lambda}$ is given by the following 
\begin{eqnarray}
\left\{ \begin{array}{ll}
\lambda^l(S):=\lambda_S, & \text{for any} \ S\subseteq [n];\\
\lambda^l(S):=\lambda^{l_S}_S, & \text{for any} \ S\in S(n+1). 
\end{array}
\right.
\end{eqnarray}
Now for each $S\in S(n+1)$, $\ve^{n+1}_S=:\ve_S: [\#S-1]\To [n+1]$ denotes the unique morphism in $\Hom_{\Delta'}([\#S-1],[n+1])$ whose range is exactly $S$. It is then clear that 
\begin{eqnarray}
\tilde{\ve}_S((U^n_\lambda \star \cG^x)\cap \supp(\mathsf{c}\circ s))\subset \bigcup_{l_S}U^{\#S-1}_{\lambda^{l_S}_S}, \ \forall S\in S(n+1).
\end{eqnarray}
Next, choose for every $S\in S(n+1)$, a partition of unity \[\vp_{\lambda^{l_S}_S}: \tilde{\ve}_S((U^n_\lambda \star \cG^x)\cap \supp(\mathsf{c}\circ s))\To \RR_+\] subordinate to the open covering $\left(U^{\# \tiny{S}-1}_{\lambda^{l_S}_S}\right)_{l_S=0}^{m_S}$.

For all $n\ge 1$, we define the map 
$h^n: CR^{n+1}_{ss}(\cU_\bullet,E^\bullet)\To CR^n_{ss}(\cU_\bullet,E^\bullet)$ by 

\begin{eqnarray}~\label{eq:h^n} 
(h^n f)_\lambda (g_1,\ldots,g_n):=  (-1)^{n+1}\int_{\cG^{s(g_n)}}\sum_{l=(l_S)_{S\in S(n+1)}}f_{\lambda^l}(g_1,\ldots,g_n,\g)\cdot \prod_{S\in S(n+1)} \nonumber \\ \ \ \ \ \ \ \ \ \ \ \ \ \ \ \ \ \ \ \ \ \ \ \ \ \ \ \ \ \ \ \ \ \ \ \ \ \ \prod_{l_S}\vp_{\lambda^{l_S}_S}(\tilde{\ve}_S(g_1,\ldots,g_n,\g))\cdot c(s(\g))d\mu^{s(g_n)}(\g).
\end{eqnarray}
Notice that $$(U^n_{\bar{\lambda}}\star \cG^{\rho(x)})\cap \supp(\mathsf{c}\circ s \circ \rho)\subset \bigcup_{l=(l_S)_{S\in S(n+1)}}U^{n+1}_{\bar{\lambda}^l},$$ where the $\bar{\lambda}^l \ $' s are defined in the obvious way. Hence, we get a partition of unity of $\tilde{\ve}_S((U^n_{\bar{\lambda}}\star \cG^{\rho(x)})\cap \supp(\mathsf{c}\circ s \circ \rho))$ subordinate to the open covering $\left(U^{\#S-1}_{\bar{\lambda}^{l_S}_S}\right)_{l_S=0}^{m_S}$ by setting $\vp_{\bar{\lambda}^{l_S}_S}(\tilde{\ve}_S(\rho(g_1),\ldots,\rho(g_n))):=\vp_{\lambda^{l_S}_S}(\tilde{\ve}_S(g_1,\ldots,g_n))$. Next, using (~\ref{eq:G_integral}), it is straightforward that \[(h^nf)_{\bar{\lambda}}(\rho(g_1),\ldots,\rho(g_n))=\nu_{|U^n_\lambda}\circ (h^nf)_\lambda(g_1,\ldots,g_n),\]
which means that $((h^nf)_\lambda)_{\lambda\in \Lam_n}\in CR^{n}_{ss}(\cU_\bullet,E^\bullet)$. 

Assume now that $(f_\lambda)_{\lambda\in \Lam_n}\in CR_{ss}^n(\cU_\bullet,E^\bullet)$. Then, for every $U^n_\lambda \in {}_\natural \cU_n$ and $(g_1,\ldots,g_n)\in U^n_\lambda$, one has
\begin{eqnarray}~\label{eq:hd}
(h^nd^nf)_\lambda(g_1,\ldots,g_n)  = (-1)^{n+1}\int_{\cG^{s(g_n)}}\sum_{(l_S)_{S\in S(n+1)}}(d^nf)_{\lambda^l}(g_1,\ldots,g_n,\g)\cdot\prod_{S\in S(n+1)} \nonumber \\ 
 \ \ \ \ \ \ \ \ \prod_{l_S}\vp_{\lambda^{l_S}_S}(\tilde{\ve}^{n+1}_S(g_1,\ldots,g_n,\g))\cdot c(s(\g))d\mu^{s(g_n)}(\g) \nonumber \\ 
 = f_\lambda(g_1,\ldots,g_n)-A_\lambda(g_1,\ldots,g_n),
\end{eqnarray}
where 
\begin{align*}
A_\lambda(g_1,\ldots,g_n):=(-1)^n \sum_{k=0}^n (-1)^k\int_{\cG^{s(g_n)}}\sum_{(l_S)_{S\in S(n+1)}}f_{\tilde{\ve}^{n+1}_k(\lambda^l)}(\tilde{\ve}^{n+1}_k(g_1,\ldots,g_n,\g))\cdot\prod_{S\in S(n+1)} \\
\prod_{l_S}\vp_{\lambda^{l_S}_S}(\tilde{\ve}^{n+1}_S(g_1,\ldots,g_n,\g))\cdot c(s(\g))d\mu^{s(g_n)}(\g).
\end{align*}
We want to show that
\begin{eqnarray}
A_\lambda(g_1,\ldots,g_n)= (d^{n-1}h^{n-1}f)_\lambda(g_1,\ldots,g_n).
\end{eqnarray}
One has
\begin{align}~\label{eq:dh}
(d^{n-1}h^{n-1}f)_\lambda(g_1,\ldots,g_n)
= (-1)^n \sum_{k=0}^{n-1}\int_{\cG^{s(g_n)}} \sum_{r_k:=(r_{k,T})_{T\in S(n)}}f_{\tilde{\ve}^n_k(\lambda)^{r_k}}(\tilde{\ve}^n_k(g_1,\ldots,g_n),\g)\cdot \prod_{T\in S(n)} \nonumber \\
\prod_{r_{k,T}}\vp_{\tilde{\ve}^n_k(\lambda)^{r_{k,T}}_T}(\tilde{\ve}^n_T(\tilde{\ve}^n_k(g_1,\ldots,g_n),\g))\cdot c(s(\g))d\mu^{s(g_n)}(\g) \nonumber \\
+ \int_{\cG^{s(g_{n-1})}}\sum_{r_n:=(r_{n,T})_{T\in S(n)}}f_{\tilde{\ve}^n_n(\lambda)^{r_n}}(g_1,\ldots,g_{n-1},\g)\cdot \prod_{T\in S(n)} \nonumber \\
\prod_{r_{n,T}}\vp_{\tilde{\ve}^n_n(\lambda)^{r_{n,T}}_T}(\tilde{\ve}^n_T(g_1,\ldots,g_{n-1},\g))\cdot c(s(\g))d\mu^{s(g_{n-1})}(\g) \nonumber \\
= B_\lambda(g_1,\ldots,g_n) + C_\lambda(g_1,\ldots,g_n).
\end{align}
Notice that by the left-invariance of $\{\mu^x\}_{x\in \Gpdo}$, the second integral $C_\lambda$ in the right hand side of (~\ref{eq:dh}) can be written as
\begin{align}
C_\lambda(g_1,\ldots,g_n)= \int_{\cG^{s(g_n)}}\int_{(r_{n,T})_{T\in S(n)}}f_{\tilde{\ve}^n_n(\lambda)^{r_n}}(g_1,\ldots,g_{n-1},g_n\g)\cdot \prod_{T\in S(n)} \nonumber \\ 
\prod_{r_{n,T}}\vp_{\tilde{\ve}^n_n(\lambda)^{r_{n,T}}_T}(\tilde{\ve}^n_T(g_1,\ldots,g_{n-1},g_n\g))\cdot c(s(\g))d\mu^{s(g_{n-1})}(\g) \nonumber \\
= \int_{\cG^{s(g_n)}} \sum_{(r_{n,T})_{T\in S(n)}}f_{\tilde{\ve}^n_n(\lambda)^{r_n}}(\tilde{\ve}^{n+1}_n(g_1,\ldots,g_n,\g))\cdot \prod_{T\in S(n)} \nonumber \\
\prod_{r_{n,T}}\vp_{\tilde{\ve}^n_n(\lambda)^{r_{n,T}}_T}(\tilde{\ve}^n_T(\tilde{\ve}^{n+1}_n(g_1,\ldots,g_{n-1},g_n,\g)))\cdot c(s(\g))d\mu^{s(g_{n-1})}(\g).
\end{align}
On the other hand, for any $k=0,\ldots,n-1$, one has $(\tilde{\ve}^n_k(g_1,\ldots,g_n),\g)=\tilde{\ve}^{n+1}_k(g_1,\ldots,g_n,\g)$; hence 
\begin{align*}
B_\lambda(g_1,\ldots,g_n)=(-1)^n\sum_{k=0}^{n-1}(-1)^k\int_{\cG^{s(g_n)}}\sum_{(r_{k,T})_{T\in S(n)}}f_{\tilde{\ve}_k^n(\lambda)^{r_k}}(\tilde{\ve}^{n+1}_k(g_1,\ldots,g_n,\g))\cdot \prod_{T\in S(n)} \\
\prod_{r_{k,T}}\vp_{\tilde{\ve}^n_k(\lambda)^{r_{k,T}}_T}(\tilde{\ve}^n_T(\tilde{\ve}^{n+1}_k(g_1,\ldots,g_n,\g)))\cdot c(s(\g))d\mu^{s(g_{n-1})}(\g).
\end{align*}
Thus, (~\ref{eq:dh}) becomes
\begin{align}~\label{eq:dh1}
(d^{n-1}h^{n-1}f)_\lambda (g_1,\ldots,g_n)=(-1)^n\sum_{k=0}^n(-1)^k\int_{\cG^{s(g_n)}}\sum_{(r_{k,T})_{T\in S(n)}}f_{\tilde{\ve}^n_k(\lambda)^{r_k}}(\tilde{\ve}^{n+1}_k(g_1,\ldots,g_n,\g)).  \nonumber \\
\prod_{T\in S(n)}\prod_{r_{k,T}}\vp_{\tilde{\ve}^n_k(\lambda)^{r_{k,T}}_T}(\tilde{\ve}^n_T(\tilde{\ve}^{n+1}_k(g_1,\ldots,g_n,\g)))\cdot c(s(\g))d\mu^{s(g_{n-1})}(\g).
\end{align}
Now, for any $k=0,\ldots,n$, $r_k=(r_{k,T})_{T\in S(n)}$, let $\g\in \cG^{s(g_n)}$ such that $\tilde{\ve}^{n+1}_k(g_1,\ldots,g_n,\g)\in U^n_{\tilde{\ve}^n_k(\lambda)^{r_k}}$. Then, there exists $l=(l_S)_{S\in S(n+1)}$ such that $(g_1,\ldots,g_n,\g)\in U^{n+1}_{\lambda^l}$, so that \[\tilde{\ve}^{n+1}_k(g_1,\ldots,g_n,\g)\in U^n_{\tilde{\ve}^n_k(\lambda)^{r_k}} \bigcup U^n_{\tilde{\ve}^{n+1}_k(\lambda^l)}.\]
One can then suppose that for any $k\in [n]$ and any family $r_k=(r_{k,T})_{T\in S(n)}$, there exists a family $l=(l_S)_{S\in S(n+1)}$ such that $\tilde{\ve}^n_k(\lambda)^{r_k}=\tilde{\ve}^{n+1}_k(\lambda^l)$. Moreover, in virtue to the identities~\eqref{eq:face_degen}, it is straightforward that for each $k\in [n]$ and any $T\in S(n)$, there exists a unique $S\in S(n+1)$ such that $\ve^{n+1}_S=\ve^{n+1}_k\circ \ve^n_T$, so that $\tilde{\ve}^{n+1}_S=\tilde{\ve}^n_T\circ \tilde{\ve}^{n+1}_k$. Therefore, we obtain from~\eqref{eq:dh1} that
\begin{align}
(d^{n-1}h^{n-1}f)_\lambda(g_1,\ldots,g_n)=(-1)^n\sum_{k=0}^n(-1)^k\int_{\cG^{s(g_n)}}\sum_{(l_S)_{S\in S(n+1)}}f_{\tilde{\ve}^{n+1}(\lambda^l)}(\tilde{\ve}_k^{n+1}(g_1,\ldots,g_n,\g)). \nonumber \\
\prod_{S\in S(n+1)}\prod_{l_S}\vp_{\lambda^{l_S}_S}(\tilde{\ve}^{n+1}_S(g_1,\ldots,g_n,\g)).c(s(\g))d\mu^{s(g_n)}(\g) \nonumber \\
= A_\lambda(g_1,\ldots,g_n).
\end{align}
Combining with (~\ref{eq:hd}), we thus have shown that
\begin{eqnarray}
h^n\circ d^n +d^{n-1}\circ h^{n-1}=\Id_{CR^n_{ss}(\cU_\bullet,E^\bullet)}, \ \forall n\ge 1;
\end{eqnarray}
\emph{i.e.} $h^\star$ defines a contraction of $CR^\star_{ss}(\cU_\bullet,E^\bullet)$ for any Real open cover $\cU_\bullet$ of $(\cG_\bullet,\rho_\bullet)$ and this ends our proof.
\end{proof}

\begin{rem}
It is straightforward, using the same arguments, that Theorem~\ref{thm:HR_vs_proper} remains true for a Real representation of type $\RR^{p,q}$ (cf. Remark~\ref{rem:Real-repr-p-q}).	
\end{rem}

\begin{cor}
Let $\cG$ be a proper groupoid. Let $E\To \Gpdo$ be a representation of $\cG$ in the sense of Crainic~\cite{Crai}; that is, a real $\cG$-equivariant vector bundle of rank $p$. Then $\check{H}^n(\cG_\bullet,E^\bullet)=0, \forall n\geq 1$.	
\end{cor}

\begin{proof}
Let $\cG$ be endowed with the trivial Real structure. Form the Real representation $(F,\nu)$ of type $\RR^{p,p}$ of $(\cG,\Id)$ by $F:=E\oplus E$ endowed with the diagonal $\cG$-action and the Real structure $\nu(e_1,e_2):=(e_1,-e_2)$. Then by Theorem~\ref{thm:HR_vs_proper}, we have $\check{H}R^n(\cG_\bullet,F^\bullet)=0$ for all $n\geq 1$. But since the Real structure is trivial, we have $\check{H}R^n(\cG_\bullet,F^\bullet)=\check{H}^n(\cG_\bullet,{}^rF^\bullet)$, thanks to the discussion following Proposition~\ref{pro:HR-vs-H}. Moreover, we obviously have ${}^rF^\bullet =E^\bullet$.
\end{proof}

\medskip 

\subsection*{Acknowledgements}
This work was done during my PhD at the universities of Metz and Paderborn, while I was suported by the German Research Foundation (DFG) via the IRTG 1133 "\emph{Geometry and Analysis of Symmetries}". I am greatly indebted to my supervisor Jean-Louis Tu, and I am grateful to Camille Laurent for many helpful comments and suggestions on an earlier version of this manuscript.


\end{document}